\documentclass[a4paper,11pt,reqno]{article}

\usepackage{amssymb,amsmath,amsthm,amsfonts,amscd,amsxtra}
\usepackage{mathrsfs}
\usepackage{float}
\usepackage{esint,latexsym}
\usepackage{graphicx}
\usepackage[colorlinks=true, urlcolor=blue, bookmarks=true, bookmarksopen=true,
  citecolor=blue]{hyperref}
\usepackage{color}

\usepackage{array}
\usepackage{fullpage}
\usepackage[normalem]{ulem}
\usepackage{url}
\usepackage{enumerate,bbm}
\usepackage[all]{xy}
\usepackage[OT2,T1]{fontenc}
\usepackage{xspace}
\usepackage{makeidx}
\usepackage[notref,notcite,color,final]{showkeys}

\DeclareSymbolFont{cyrletters}{OT2}{wncyr}{m}{n}
\DeclareMathSymbol{\Sha}{\mathalpha}{cyrletters}{"58}

\usepackage{calligra}

\makeindex

\theoremstyle{plain}
\newtheorem{thm}{Theorem}[section]
\newtheorem{theorem}[thm]{Theorem}

\newtheorem{corollary}[thm]{Corollary}
\newtheorem{lem}[thm]{Lemma}
\newtheorem{lemma}[thm]{Lemma}
\newtheorem{prop}[thm]{Proposition}
\newtheorem{proposition}[thm]{Proposition}

\newtheorem{conjecture}[thm]{Conjecture}

\theoremstyle{remark}
\newtheorem{remark}[thm]{Remark}

\theoremstyle{definition}

\newtheorem{defn}[thm]{Definition}
\newtheorem{definition}[thm]{Definition}
\newtheorem{defn+lem}[thm]{Definition and Lemma}

\numberwithin{equation}{section}

\newcommand{\BA}{{\mathbb{A}}} \newcommand{\BB}{{\mathbb{B}}}
\newcommand{\BC}{{\mathbb{C}}} 
 \newcommand{\BF}{{\mathbb{F}}}
\newcommand{\BG}{{\mathbb{G}}} \newcommand{\BH}{{\mathbb{H}}}

 \newcommand{\BN}{{\mathbb{N}}}
 
\newcommand{\BQ}{{\mathbb{Q}}} \newcommand{\BR}{{\mathbb{R}}}

\newcommand{\BW}{{\mathbb{W}}} 
 \newcommand{\BZ}{{\mathbb{Z}}}

\newcommand{\CA}{{\mathcal{A}}} 
\newcommand{\CC}{{\mathcal{C}}} 
 \newcommand{\CF}{{\mathcal{F}}}
\newcommand{\CG}{{\mathcal{G}}} 
\newcommand{\CI}{{\mathcal{I}}} 
\newcommand{\CK}{{\mathcal{K}}} \newcommand{\CL}{{\mathcal{L}}}
 
\newcommand{\CO}{{\mathcal{O}}} \newcommand{\CP}{{\mathcal{P}}}
 
\newcommand{\CS}{{\mathcal{S}}} \newcommand{\CT}{{\mathcal{T}}}
\newcommand{\CU}{{\mathcal{U}}} \newcommand{\CV}{{\mathcal{V}}}
\newcommand{\CW}{{\mathcal{W}}} \newcommand{\CX}{{\mathcal{X}}}

\newcommand{\fg}{{\mathfrak{g}}}

\newcommand{\fm}{{\mathfrak{m}}} 
 \newcommand{\fp}{{\mathfrak{p}}}
 
 \newcommand{\ft}{{\mathfrak{t}}}

 \newcommand{\fX}{{\mathfrak{X}}}
\newcommand{\fY}{{\mathfrak{Y}}} 

\newcommand{\HT}{\mathrm{HT}}

      \newcommand{\ad}{{\mathrm{ad}}}
      \newcommand{\alg}{{\mathrm{alg}}}

\newcommand{\Char}{{\mathrm{Char}}}  
      
  \newcommand{\Cond}{{\mathrm{Cond}}}

    \newcommand{\cl}{{\mathrm{cl}}}
\newcommand{\Cl}{{\mathrm{Cl}}}

\newcommand{\dR}{{\mathrm{dR}}}      
\newcommand{\disc}{{\mathrm{disc}}}

\newcommand{\End}{{\mathrm{End}}}    \newcommand{\ext}{{\mathrm{ext}}}

\newcommand{\Frac}{{\mathrm{Frac}}}  \newcommand{\Fr}{{\mathrm{Fr}}}

      \newcommand{\Gal}{{\mathrm{Gal}}}
\newcommand{\GL}{{\mathrm{GL}}}      

      \newcommand{\Hom}{{\mathrm{Hom}}}
\newcommand{\Ha}{{\mathrm{Ha}}}      \newcommand{\Hdg}{{\mathrm{Hdg}}}

\newcommand{\id}{{\mathrm{id}}}

\newcommand{\IG}{{\mathrm{IG}}}      
\newcommand{\Isom}{{\mathrm{Isom}}}

\newcommand{\KS}{{\mathrm{KS}}}      
\newcommand{\Lie}[1]{{\mathrm{Lie}}(#1)}  
    \newcommand{\mom}{{\mathrm{mom}}}

\newcommand{\op}{{\mathrm{op}}}
\newcommand{\ord}{{\mathrm{ord}}}    
    
        \newcommand{\pr}{{\mathrm{pr}}}
  
\renewcommand{\mod}{\ \mathrm{mod}\ }\renewcommand{\Re}{{\mathrm{Re}}}

\newcommand{\red}{{\mathrm{red}}}

\newcommand{\Spf}{{\mathrm{Spf}\hspace{2pt}}}
\newcommand{\SO}{{\mathrm{SO}}}      
\newcommand{\St}{{\mathrm{St}}}      
\newcommand{\Sym}{{\mathrm{Sym}}}

\newcommand{\TSym}{{\mathrm{TSym}}}

      \newcommand{\RTr}{{\mathrm{Tr}}}
  \newcommand{\ur}{{\mathrm{ur}}}
\newcommand{\val}{{\mathrm{val}}}    \newcommand{\Vol}{{\mathrm{Vol}}}

\newcommand{\Fil}{\mathrm{Fil}}

\renewcommand{\-}{\text{--}}

\newcommand{\dfn}[1]{\textit{#1}}

\newcommand{\bs}{\backslash}

\newcommand{\sk}{\medskip}   
\newcommand{\s}{\sk\noindent}

\newcommand{\sheafhom}{\mathrm{H}\kern -.5pt om}

\newcommand{\CIG}{{\mathcal{IG}}} \newcommand{\fIG}{{\mathfrak{IG}}}

\newcommand{\udelta}{{\underline{\delta}}} 
\newcommand{\uomega}{{\underline{\omega}}}

\newcommand{\idele}{id\`{e}le\xspace}

\newcommand{\Poincare}{Poincar\'{e}\xspace}

\makeatletter
\def\varW@#1#2{%
  \vtop{\m@th\ialign{##\cr
  \hfil$#1 \mathrm{colim}$\hfil\cr
  \noalign{\nointerlineskip\kern1.5\ex@}#2\cr
  \noalign{\nointerlineskip\kern-\ex@}\cr}}
}
\def\colim{%
  \mathop{\mathpalette\varW@{}}\nmlimits@
}
\makeatother

\def\XXint#1#2#3{{\setbox0=\hbox{$#1{#2#3}{\int}$}
    \vcenter{\hbox{$#2#3$}}\kern-.5\wd0}}

\begin{document}
\title{$p$-adic Waldspurger Formula for Non-split Primes and Converse of Gross--Zagier and Kolyvagin Theorem}

\author{Yangyu Fan and Xin Wan}

\date{}



\maketitle
\begin{abstract}
Let $p$ be a prime and $\mathcal{K}$ be an imaginary quadratic field. In this paper we generalize a recent construction of a new type of $p$-adic $L$-function and $p$-adic Waldspurger formula by Andreatta-Iovita for $p$ non-split in $\mathcal{K}$, as long as the $\mathrm{GL}_2$ automorphic representation is principal series at $p$. Then we develop a new kind of anticyclotomic local $\pm$-Iwasawa theory at $p$ for self-dual Hecke characters over imaginary quadratic fields (including elliptic curves $E/\mathbb{Q}$ with complex multiplication) which is valid for all ramification types of $p$ (split, inert and ramified, and allowing $p=2$). As the main consequence, we prove the converse of the Gross--Zagier--Kolyvagin theorem for self-dual CM characters: if the Selmer rank of $\psi$ is 1, then the analytic rank of $L(\psi,s)$ at $s=1$ is also 1. As corollaries, we prove Sylvester's conjecture (1879) on sums of two rational cubes, and Goldfeld's conjecture for CM elliptic curves over $\mathbb{Q}$ (conditional on work of A. Smith).
\end{abstract}
\section{Introduction}
\subsection{Background}
The deep relations between special values of $L$-functions and arithmetic objects constitute a central problem in number theory. Let $p$ be a prime number. The famous Bloch--Kato conjecture \cite{BK} predicts that for a motive whose $p$-adic realization is $\rho$, the rank of its $p$-adic Selmer group is equal to the vanishing order of the $L$-function $L(\rho^\vee(1),s)$ at $s=0$. This conjecture remains open in general.

When $\rho$ is the Galois representation associated to an elliptic curve $E$
  over $\mathbb{Q}$, the converse of the Gross--Zagier and Kolyvagin
  theorem is an important special case predicted by the Bloch--Kato
  and BSD conjectures:  if the Selmer rank of $E$ is
  $0$ or $1$, then the vanishing order of $L(E,s)$ at $s=1$ is
  exactly $0$ or $1$ respectively. We focus on the rank $1$ case.

When $E$ has no complex multiplication and $p$ is odd, the result is proved by Skinner \cite{Skinner} under the assumption of finiteness of the $p$-part of the Shafarevich--Tate group $\Sha_E$, and also by Zhang \cite{Zhang} when $E$ is ordinary at $p$. These assumptions are removed later by the work of the second author \cite{WAN} and Castella--Wan \cite{CW} by using anticyclotomic Iwasawa theory. This converse theorem has important arithmetic applications, including results on the average analytic rank of elliptic curves and the result of Bhargava-Skinner-Zhang  \cite{BSZ} that at least $66\%$  of elliptic curves satisfy the rank part of the BSD conjecture.

In the CM case, the rank $0$ case is proved unconditionally by Burungale--Tian (\cite{BurungaleTian}). In the rank $1$ case, Rubin proved this converse theorem in the $p$-ordinary case (thus $p$ is split in the  CM field $\mathcal{K}$) under the  assumption that  the $p$-part of the Shafarevich--Tate group $\Sha_E$ is finite. Recently Burungale--Tian \cite{BT} removed the finiteness assumption   when $p\geq 5$. When $E$ has good supersingular reduction (thus $p$ is inert in $\mathcal{K})$, Rubin set up a general local Iwasawa theory at $p$.  Let $\mathcal{K}^-_\infty$ be the anticyclotomic $\mathbb{Z}_p$-extension of $\mathcal{K}$ and $\Gamma_\CK^-:=\mathrm{Gal}(\mathcal{K}^-_\infty/\mathcal{K})\simeq \mathbb{Z}_p$. Let $\mathcal{K}^-_n$ be the sub-extension of $\mathcal{K}^-_\infty$ of index $p^n$ over $\mathcal{K}$. Rubin defined the $\pm$-subspaces $H^1_+$ ($H^1_-$) of the rank two (over the anticyclotomic Iwasawa algebra) local Iwasawa module $H^1_{\mathrm{Iw}}(\mathcal{K}^-_{\infty, p}, \psi)$ to be elements specializing to elements in $H^1_f$ (the finite part) at arithmetic points $\phi$ corresponding to characters of $\Gamma_\CK^-$ of odd (even) powers of  $p$, respectively. Rubin also made a fundamental conjecture stating that
$$H^1_{\mathrm{Iw}}(\mathcal{K}^-_{\infty,p}, \psi)=H^1_+\oplus H^1_-.$$
On the other hand, Rubin constructed cohomology classes $\kappa_{p^n}\in H^1(\mathcal{K}_n^-,\psi)$ from Heegner points of conductor $p^{n+1}$ (for modular forms of prime-to-$p$ level) which satisfy the norm relation
$$\mathrm{tr}_{\mathcal{K}^-_{n+2}/\mathcal{K}^-_{n+1}}(\kappa_{p^{n+2}})=-\kappa_{p^n},$$
and formed the $\pm$-family $\kappa_{\pm}$ by modifying by the $\pm$ $p$-adic logarithm function.

This conjecture of Rubin was proved recently by Burungale--Kobayashi--Ota \cite{BKO}, and from it they also deduced an anticyclotomic Iwasawa main conjecture involving the $\pm$ $p$-adic $L$-functions. The converse of the Gross--Zagier and Kolyvagin theorem was proved in this case  under the assumption that the $p$-part of the Shafarevich--Tate group is finite (this assumption arises from their different approach; see \cite[Theorem 1.5]{BKO1}).

A key tool in the proofs of \cite{Skinner}, \cite{WAN} and
  \cite{CW} is the Bertolini--Darmon--Prasanna (BDP) formula
  \cite{BDP} for   a
  Rankin pair $(f,\chi/\CK)$, i.e., $f$ is a $\mathrm{GL}_2/\mathbb{Q}$
  eigenform and $\chi$ is a Hecke character on an imaginary quadratic field $\mathcal{K}$ whose
  restriction to $\mathbb{A}^\times_\mathbb{Q}$ is the inverse of
  the central character of $f$ when the prime $p$ is split in $\mathcal{K}$. This Waldspurger type formula  relates the $p$-adic logarithm of
  the Heegner point to a special value of an anticyclotomic  $p$-adic
  $L$-function outside its range of interpolation. The Iwasawa theory of BDP $p$-adic $L$-function is thus closely related to the Iwasawa theory of Heegner points.

Recently, Andreatta--Iovita \cite{AI} constructed this anticyclotomic $p$-adic $L$-function for a prime $p$ non-split in $\mathcal{K}$ and established a $p$-adic Waldspurger formula, under the assumption that $p>2$, the automorphic representation $\sigma(f)$ is unramified at $p$, $\chi$ is sufficiently ramified at $p$ and a classical Heegner hypothesis (forcing the associated Shimura curve to be the modular curve). Their construction uses a technique of iterating Gauss--Manin connections in $p$-adic families on some overconvergent locus of the modular curve. These restrictions arise from geometric difficulties in estimating certain convergence radii.
\subsection{Main Results}
The main purpose of this paper is twofold: 
\begin{enumerate}[(A)]
    \item Generalize the work of Andreatta-Iovita \cite{AI} to construct the anticyclotomic $p$-adic $L$-function and establish the $p$-adic Waldspurger formulae for general Rankin pairs allowing all primes (including $2$), all ramifications and all weights with the only assumption  that  $\sigma(f)$ is a principal series at $p$.
    \item Develop a new kind of $\pm$-local Iwasawa theory which applies to all ramification types of $p$ in $\mathcal{K}$ (split, inert or ramified), and apply it to show the converse of the Gross--Zagier and Kolyvagin theorem for self-dual CM characters of $\mathcal{K}$ (including CM elliptic curves over $\mathbb{Q}$ as special cases).
\end{enumerate} 
Our main result is the $p$-converse theorem for self-dual Hecke characters on $\CK$. 
\begin{theorem}\label{ThmM}Let $\psi$ be a Hecke character of $\mathcal{K}$ with Archimedean type $(-1,0)$ whose restriction to $\mathbb{A}^\times_\mathbb{Q}$ is $|\cdot|^{-1}\chi_{\CK}$ where $\chi_{\CK}$ is the quadratic character corresponding to $\mathcal{K}/\mathbb{Q}$. 

If the rank of the Selmer group for $\psi$ is $1$, then the vanishing order of $L(\psi,s)$ at $s=1$ is also $1$.
\end{theorem}
The following theorem proves a conjecture of Sylvester (1879)
  \cite{Sylvester} (see the nice survey article \cite[Section 1]{DV} by Dasgupta and Voight for background).
  \begin{theorem}
  Every prime $p \equiv 4, 7,$ or $8 \pmod{9}$ can be written as
  $x^3 + y^3$ for $x, y \in \mathbb{Q}$.
  \end{theorem}
\begin{proof}
Consider the elliptic curve $E_p: x^3+y^3=p$. It has complex multiplication by $\mathbb{Q}(\sqrt{-3})$. It is well known that the global root number for the $L$-function of $E_p$ is $-1$. By \cite[Theorem 2.9]{Sta}, the $3$-adic Selmer rank of $E_p$ is less than or equal to $2$. By the parity result \cite[Theorem 1.9]{DD}, the $3$-adic Selmer rank is odd. Thus it must be $1$. By Theorem \ref{ThmM}, the $L$-function of $E_p$ has order $1$. Then by the theorem of Gross--Zagier and Kolyvagin, the Mordell--Weil rank of $E_p$ is also one. Consequently, the equation $x^3+y^3=p$ has infinitely many rational solutions.
\end{proof}
  
Another application is to Goldfeld's conjecture.
\begin{conjecture}
Let $E$ be an elliptic curve over $\mathbb{Q}$. Consider all quadratic twists of $E$ ordered by the conductor. Then $ 50\%$ of them have analytic rank $0$, and $50\%$  have analytic rank $1$.
\end{conjecture}
Recently Alexander Smith obtained some results (see \cite[Theorem 1.2]{Smith} and \cite[Theorem 1.1]{Smith1}) that in the above quadratic twist family, $50\%$  have $2$-adic Selmer rank $1$, and $50\%$ have $2$-adic Selmer rank $0$. See also the work \cite{PT25} of Pan and Tian. We have the following theorem (conditional on \cite[Theorem 1.1]{Smith1}). 
\begin{theorem}
Goldfeld's conjecture is true for CM elliptic curves over $\mathbb{Q}$.
\end{theorem}
\begin{proof}
Let $E/\mathbb{Q}$ be an elliptic curve with complex multiplication. Then \cite[Theorem 1.1]{Smith1} implies that in its quadratic twist family ordered by conductor, $50$ percent have $2$-adic Selmer ranks $0$ and $50$ percent have $2$-adic Selmer ranks $1$. Now by Theorem \ref{ThmM} (for rank $1$) and by the main result of \cite{BurungaleTian} (for rank $0$) for $p=2$, we get Goldfeld's conjecture for $E$.
\end{proof}
\subsection{Strategy and New Ideas}
\subsubsection{Strategy}
Now we discuss the difficulties in the proof and our ideas of the argument. The idea of proving the $p$-converse theorems via Perrin-Riou's conjecture already appeared in the second author's work with Castella \cite{CW}, which is also the main strategy here.
To show the analytic rank is $1$, it is enough to (see \cite[Introduction]{YZZ})
\begin{itemize}
\item produce a Rankin pair $(f,\chi/\CK)$ where the eigenform $f$ has weight 2 and has CM by $\CK$  such that the $L$-function $L(f_\CK\times\chi,s)$ splits up into two Hecke $L$-functions $L(\psi,s)$ and $L(\psi',s)$ with $L(\psi', 1)\not=0$;
\item show that the Heegner point $\kappa_{f,\chi}$ constructed from $(f,\chi)$ is not torsion.
\end{itemize}
Given these two facts, the Gross--Zagier formula applied to the
  Rankin pair $(f,\chi)$ implies $L'(\psi,1) \neq 0$, hence
  $\mathrm{ord}_{s=1}L(\psi,s) = 1$. The rest of the introduction
  explains how each step is achieved.

\subsubsection{Construction of $p$-adic $L$-function and $p$-adic Waldspurger formula}
In this paper we construct the anticyclotomic $p$-adic $L$-function and establish the $p$-adic Waldspurger formula for general Rankin pairs $(f,\chi/\CK)$ under the only assumption that the automorphic representation $\pi$ for $f$ is a principal series at $p$. In particular, we allow all ramification types of $p$ and the Shimura curve associated to $(f,\chi/\CK)$ to be general. Our construction also works for Eisenstein series and the resulting $p$-adic $L$-function is the Katz $p$-adic $L$-function for CM fields constructed in \cite{Katz} when $p$ is split. 

Our construction  largely follows the approach of Bertolini--Darmon--Prasanna \cite{BDP}, Liu--Zhang--Zhang \cite{LZZ18} and Andreatta--Iovita \cite{AI}, namely using the theory of $p$-adic iteration of Gauss-Manin connections. Our key new idea is to systematically combine representation theory, particularly the theory of local test vector at $p$ and the $p$-adic properties of CM points  to replace the explicit estimation of  convergence radii in \cite{AI}. Note that since the theory of $p$-adic iteration of Gauss--Manin connections  has been generalized to higher dimensional Shimura varieties (\cite{GPR25,Gra26} etc.) but only in some small neighborhood of the ordinary locus, our observation suggests that one may expect to have constructions of $p$-adic $L$-functions in general contexts without knowing the exact size of that small neighborhood. 

Now we briefly explain our new ingredients. Firstly, we use the theory of vector bundles with marked sections and marked splittings (VBMSS) to define the space of nearly overconvergent modular forms and use the theory of Serre--Tate expansions to establish the $p$-adic iteration of Gauss--Manin connections.  This idea appears in \cite{Mol21} and \cite{Kaz24}; our
  contribution is to make the theory fully explicit, including the
  convergence radii, and to supply the arguments needed for
  applications to anticyclotomic $p$-adic $L$-functions.  We note that using the VBMSS theory, the $p$-adic iteration of Gauss-Manin connections is locally analytic for all characters in the weight space of $\BZ_p^\times$. For more details, we refer to 
Section \ref{NOQMF}.

Secondly, the key geometric input is to choose the embedding $\mathcal{K}_p \hookrightarrow M_2(\mathbb{Q}_p)$
 and the local test vector $\varphi_p \in \pi_p$ so that every CM point, after translation by a specific matrix $g_{\ell,n} \in \mathrm{GL}_2(\mathbb{Q}_p)$, can have arbitrarily small Hasse invariant and  land in the half-canonical locus of the Shimura curve. Concretely, the half-canonical locus is the region of the Shimura curve with $p$-level structure
 $$\left\{\begin{pmatrix}a & b\\ c & d\end{pmatrix}\in\mathrm{GL}_2(\mathbb{Z}_p)\,\Big|\, a-1\equiv d-1\equiv 0\bmod p^n,\ p^{2n}\mid c\right\}$$
 where the level subgroup and the canonical subgroup at $p$ have intersection of rank at least $p^n$. For details, see Subsection \ref{Half canonical}.  This translation not only ensures  the convergence of the Gauss--Manin connection, but also guarantees that the local toric integral at $p$ is nonzero by the test-vector theory (Appendix \ref{Local test vector}). In the non-split case this test vector is described via the $\chi_p$-eigenvector in the Kirillov model of $\pi_p$ using Jacquet module theory, which requires $\pi_p$ to be a principal series. See Subsections \ref{CM points}, \ref{Half canonical} and Appendix \ref{Local test vector} for details. We remark that the introduction of the twisting element $g_{\ell,n}$ is inspired by the beautiful ideas in Cai--Shu--Tian \cite{CST} and Andreatta--Iovita \cite{AI}.

Finally, to establish the $p$-adic Waldspurger formula, 
we follow the approach in \cite{HW24} to use syntomic formalism with coefficients, which reduces the computation of the Abel--Jacobi map on the Shimura curve to much simpler computation on CM points via the adjunction formula. However for the generality exhibited in this paper, we need to handle the semi-stable reduction case. This relies heavily on the recent theory of Andreatta--Bertolini--Seveso--Venerucci \cite{ABSV}. Note that this theory only requires the semistable model of the Shimura curve. See Section \ref{Explicit reciprocity law} for more details.
\begin{remark}
Another approach to the $p$-adic Waldspurger formula uses finite polynomial cohomology of the Kuga--Sato variety, as in \cite{BDP}. However, this will require the semi-stable model of the Kuga--Sato variety, whose existence is unclear in the generality we need\footnote{We thank Brian Conrad for pointing this out.}.
\end{remark}

The $p$-adic Waldspurger formula will be used to compare the anticyclotomic $p$-adic $L$-function with the big Heegner classes in Appendix \ref{Heegner cycles in family} whose construction in the exhibited generality generalizes  \cite{JLZ} and \cite{Wal25}. Note that one uses the $p$-stabilized vector to vary Heegner cycles  in $p$-adic analytic families, which is different from the  local test vector in the construction of  the anticyclotomic $p$-adic $L$-function. For the comparison, we apply local multiplicity one to identify the defect factor, which depends only on the local representations at $p$.

\subsubsection{$\pm$-Iwasawa theory}
We develop two anticyclotomic Iwasawa theories. The first works
  over the ring of analytic functions on a small disc around the
  origin (varying Archimedean weights); the second works over a
  localization of the anticyclotomic Iwasawa algebra (varying finite
  order twists of $p$-power conductor). Both theories suffice for our purpose, since to prove the $p$-converse theorem we only need to study the local behavior near the origin. However
\begin{itemize}
\item The first theory is more flexible, and that the coefficient ring is precisely where the various $p$-adic $L$-functions live in, and we allow all primes including $2$.
\item The second theory has the advantage that by considering growth condition we are able to construct bounded Heegner family, making it enough to work with infinitely many finite order twists (thus no need to use Coleman families). It also only needs the $p$-adic Waldspurger formulae for sufficiently ramified characters. In view of \cite{AI} these can be done following the approach in \emph{loc.cit.} without much innovation. However due to lack of \cite[Proposition 2]{RohrlichRN} for $p=2$, we do not cover $p=2$ for this theory here. Nevertheless, we were informed that Burungale and He proved the analogue of \cite[Proposition 2]{RohrlichRN} in a work in progress. With this in hand, the theory for $p=2$ can be developed in a similar way.
\end{itemize}
The following table indicates the main difference between the two theories
\begin{center}
\renewcommand{\arraystretch}{1.6}
\begin{tabular}{|p{3.2cm}|p{5.4cm}|p{5.4cm}|}
\hline
& \textbf{First theory (small disc)}
& \textbf{Second theory (localization)} \\
\hline
\textbf{Base ring}
& $\mathcal{O}_L[[p^{-m}X]]\otimes L$
& $S^{-1}\Lambda^-_L$ \\
\hline
\textbf{Varies}
& Archimedean weights
& Finite order twists of $p$-power conductor \\
\hline
\textbf{$p$-adic $L$-function}
& Full generality required
& Sufficiently ramified only; follows \cite{AI} \\
\hline
\textbf{Covers $p=2$}
& Yes
& No (temporarily) \\
\hline
\textbf{Heegner family}
& \emph{Virtual}; via Coleman family \cite{JLZ}
& \emph{Bounded}; via half logarithms \cite{Pollack} \\
\hline
\textbf{Key input}
& Vanishing of one $p$-adic $L$-function
  via global root numbers
& Nonvanishing of $L$-values at half
  the arithmetic points \cite{Jia} \\
\hline
\multicolumn{3}{|p{14.0cm}|}{%
\emph{Both theories give the same $\pm$-submodules as
\cite{BKNO} (see Lemmas~\ref{characterize1},~\ref{characterize2}),
but on a smaller domain, with the key advantage of an explicit
connection to the $p$-adic Waldspurger formula.}} \\
\hline
\end{tabular}
\end{center}
Note that although the second approach does not need the full strength of our $p$-adic Waldspurger formulae above,  we believe the generality and the method developed in this article have their own importance. In a forthcoming work, we plan to use our result  to prove the refined BSD formula for CM elliptic curves in the rank $1$ case, where the $p$-adic Waldspurger formula in this generality will be indispensable.

Now we discuss the first Iwasawa theory with more details.  Let $\Lambda^-:=\mathcal{O}_L[[\Gamma_\CK^-]]\simeq\mathcal{O}_L[[X]]$ be the anticyclotomic Iwasawa algebra, where the variable $X$ is mapped to $\gamma^--1$ for a topological generator $\gamma^-\in\Gamma_\CK^-$.

Our method differs from those in the literature: we develop a new kind of $\pm$-local Iwasawa theory, which is valid for all ramification types for $p$ (split, inert and ramified cases). Our theory is analogous to Rubin's $\pm$-theory in format but quite different in nature: we divide the arithmetic points in $\Lambda^-$ into two parts $\mathfrak{X}_1$ and $\mathfrak{X}_2$ depending on the Archimedean types instead of the conductors at $p$. We define the $+$-part ($-$-part) submodules of rank $1$ of the local Iwasawa theoretic Galois cohomology at $p$ to be the subspace whose specializations lie  in $H^1_f$ (the kernel of the dual exponential map $\exp^*$) at arithmetic points in $\mathfrak{X}_1$ ($\mathfrak{X}_2$) respectively.  Our construction of the $\pm$-theory uses explicit reciprocity law for families of elliptic units for various auxiliary Hecke characters whose $L$-function has vanishing orders $0$ and $1$ respectively. To get the analogue of Rubin's conjecture we need to shrink the weight space (this suffices for our purpose because we are only focusing on the local analytic behavior near the origin point) --- this means we take some large number $m$ and define $\Lambda^{-,\prime}$ to be $\mathcal{O}_L[[p^{-m}X]]\otimes L$, whose spectrum parametrizes the open disc $|X|_p <
  p^m$ in weight space--- and then the analogous version of Rubin's conjecture in \cite{Rubin} becomes much easier to prove, i.e., we can easily prove that
$$H^1(\mathcal{K}_p, \boldsymbol{\psi})=H^1_+\oplus H^1_-$$
after shrinking the weight space (here $\boldsymbol{\psi}$ is the anticyclotomic deformation of $\psi$).

While we are writing up this version, Burungale-Kobayashi-Ota-Nakamura \cite{BKNO} developed similar $\pm$-theory, using an involution map constructed by Colmez. As shown in  Lemmas ~\ref{characterize1} and ~\ref{characterize2}, the $\pm$-submodules of our theory coincide with those of \cite{BKNO}, in our smaller defining domains. Their argument is purely local, and is valid over the anticyclotomic Iwasawa algebra $\Lambda^-$, thus is stronger than our theory in this aspect. A $p$-adic $L$-function is constructed upon choosing some basis  of the $\pm$-module, and involves the specialization formula for this chosen basis. However our construction, together with the explicit interpolation formula, makes it more conveniently related to the $p$-adic Waldspurger formula, and thus to Iwasawa theory of Heegner points. This is crucial for our purpose.

\subsubsection{Heegner family}

Perrin-Riou's conjecture requires the construction of a family of Heegner points (it states that the index of the Heegner family in certain Iwasawa cohomology module gives the characteristic ideal of the torsion part of the corresponding Selmer group). In applying our first theory over the small disc, we define a so-called ``virtual Heegner family'' $\kappa^{\mathrm{virt}}_{f,\boldsymbol{\chi}}$ which means an element in the global cohomology module $H^1_+(\mathcal{K}, \boldsymbol{\psi})$ (a rank $1$ module) whose localization at $p$ satisfies the explicit reciprocity law expected for the family of Heegner cycles, and whose specialization to $\phi_0$ is $\kappa_{f,\chi}$, where $\phi_0$ is the origin point in the weight space where $\boldsymbol{\chi}$ specializes to $\chi$. Then this family can play the role of Heegner family in \cite{WAN} and \cite{CW} in the argument.

Note that Rubin's construction of the $\pm$ Heegner family uses crucially that the prime $p$ is unramified in $\mathcal{K}$ and $E$ has prime to $p$ conductor. To display a virtual Heegner family, we use a different idea -- we consider a Coleman family $\mathcal{F}$ passing through $f$ (note that a generic member of $\mathcal{F}$ is not CM), and bring in the construction of Jetchev-Loeffler-Zerbes \cite{JLZ} for the two-variable Heegner families over the Coleman family $\mathcal{F}$. The reason to invoke Coleman families is that there are not  enough arithmetic points on the $1$-dimensional small disc in the fiber  of $\mathcal{F}$ at $f$. Note that the CM form associated to $\psi$ is supercuspidal at $p$. Thus in order to make this work our idea is to take a form corresponding to a CM character unramified at $p$ so that the CM form has finite slope, and move the ramification of $\psi$ at $p$ to the CM character $\chi$ part. (The idea of making appropriate choice of the pair $(f,\chi)$ to meet various purposes is already employed by Burungale--Tian in \cite{BT}.)

In applying our second theory over the localization of the Iwasawa algebra, we produce a bounded Heegner family, using generalizations of Pollack's theory \cite{Pollack}. We define various kinds of half $p$-adic logarithms generalizing those of Pollack. Using results of \cite{Jia} that the values of the corresponding $L$-values are nonzero at half of the arithmetic points, we show that the specializations of the Heegner family associated to a stabilization $(f_\alpha,\chi)$ at half of the finite order anticyclotomic twist points are $0$, when projecting to the $\psi$-component. We show that the anticyclotomic Iwasawa cohomology of $\psi$ is free of rank $1$, and using the growth property of the Heegner family, we show that the $\psi$-component can be divided by this half $p$-adic logarithm and obtain a bounded family over the localization of the Iwasawa algebra. Unlike the bounded $\pm$-Heegner family in \cite{Rubin} and in \cite{CW} which are constructed from the local norm relation at $p$ for un-stabilized $f$, our construction of the bounded family relies on a global argument.
\subsubsection{Regulator maps -- the principle of taking quotients}
When $p$ is split in $\mathcal{K}$ we have a well-understood local regulator map which is important for carrying out the Iwasawa theory argument. Equivalently this means having  certain basis of various local Iwasawa modules with explicit interpolation formulas. When $p$ is non-split, we do not have a good understanding of such regulator map. 
\begin{itemize}
\item The key principle is: instead of constructing an explicit regulator map, we work with ratios between elliptic unit families (or Heegner point families) we study and those for auxiliary characters whose $L$-functions have vanishing orders $0$ or $1$.
\end{itemize}

The advantage is, although the interpolation formula for an abstract basis of the local Iwasawa cohomology is rather mysterious, the ratios between elliptic units (or Heegner point families) for the character $\psi$, and those for the auxiliary characters, are far more tractable and enable us to prove various relations between elliptic units and Heegner families (see e.g. Proposition \ref{bddHeegnerfamily}).

One difficulty is to show that these families for auxiliary characters indeed form a local basis over appropriate localizations. Lemma \ref{AV} plays an important role, which says that the localization map from the Mordell-Weil group of a $\mathrm{GL}_2$-type abelian variety to the local Galois cohomology at $p$ of its $\mathfrak{p}$-adic Tate module (where $\mathfrak{p}$ is any prime of its Hecke field above $p$) is injective. It is clearly true for elliptic curves over $\mathbb{Q}$, but our argument requires knowing it for general $\mathrm{GL}_2$-type abelian varieties. The lemma is proved by the second-named author with Burungale and Skinner. The proof is an application of the $p$-adic analytic subgroup theorem, which was  a significant breakthrough in transcendental number theory.
\subsubsection{Applying the $p$-adic Waldspurger formula}
To show that the Heegner family can be made to run our Iwasawa theory argument, we use our $p$-adic Waldspurger formula. However, one subtlety is that the local images of the Heegner cycles at $p$ actually involve two $p$-adic $L$-functions, depending on the choice of an embedding $\iota: \bar{\mathbb{Q}}\hookrightarrow \mathbb{C}_p$. In the simpler case where $p$ splits in $\mathcal{K}$, these correspond to the localizations at the two primes above $p$. To do Iwasawa theory we only need to consider one of them. However when $p$ is non-split in $\mathcal{K}$, we are unable to separate the two components. Thus the actual argument is more convoluted. A crucial observation to make the argument work is that when $f$ is CM, one of the two Rankin-Selberg $p$-adic $L$-functions  is identically zero. Indeed, one of them is the product of two CM $p$-adic $L$-functions with global root numbers $+1$, and the other is the product of two CM $p$-adic $L$-functions with global root numbers $-1$, and thus this second one must be identically zero. Using this observation, we can take some appropriate auxiliary anticyclotomic twists $\chi\eta$ and $\chi\eta'$ of $\chi$ and form a linear combination of the Heegner families associated to $(\mathcal{F},\chi\eta)$ and $(\mathcal{F},\chi\eta')$ which specializes under $\mathcal{F}\mapsto f$ to a generator of $H^1_+$ at $p$, which we denote as $v_+$, such that $\kappa_{f,\boldsymbol{\chi}}=\mathcal{L}_{f,\chi}\cdot v_+$. This can be proved by applying the $p$-adic Waldspurger formula at a Zariski dense set of arithmetic points. This provides the virtual Heegner family which is enough for the Iwasawa theory argument.

In applying our second theory, this issue is simpler -- since we just work with one CM form instead of a Coleman family, the theory of complex multiplication enables us to split the two components with respect to the $2$ characters.  See Section \ref{VHC} for more details.

\subsubsection{Iwasawa main conjecture and Perrin-Riou conjecture}
The final step uses the Iwasawa main conjecture and a global duality argument
to show that the virtual Heegner family is non-torsion at the origin.

The $p$-adic $L$-function $\mathcal{L}^{\bullet}_\psi$ interpolates the algebraic
parts of $L(\boldsymbol{\psi}_\phi, 1)$ for arithmetic points $\phi$ in
$\mathfrak{X}_2$.  The Iwasawa main conjecture for it asserts
\[
  \mathrm{char}_{\Lambda^{-,\prime}_L}(X^-_\psi)
  \;=\;
  \bigl(\mathcal{L}^{\bullet}_\psi\bigr).
\]
This is deduced from the main conjecture for elliptic units, proved by Rubin
\cite{Rubin2} and in full generality by Johnson-Leung-Kings \cite{JK}.

With the main conjecture established, the Poitou-Tate exact sequences
\[
  0 \to H^1_+(\mathcal{K}^S, \boldsymbol{\psi})
    \to H^1_+(\mathcal{K}_p, \boldsymbol{\psi})
              \to X^{\mathrm{rel}}_\psi
    \to X^{\mathrm{+}}_\psi
    \to 0
\]
and
\[
  0 \to H^1(\mathcal{K}^S, \boldsymbol{\psi})
    \to \frac{H^1(\mathcal{K}_p, \boldsymbol{\psi})}
             {H^1_-(\mathcal{K}_p, \boldsymbol{\psi})}
    \to X^-_\psi
    \to X^{\mathrm{str}}_\psi
    \to 0
\]
(here the $X^{\pm}_\psi$, $X^{\mathrm{rel}}_\psi$ and $X^{\mathrm{str}}_\psi$ are the Pontryagin duals of the $\pm$, relaxed and strict Selmer groups respectively, see Definition \ref{defineSelmer} and \ref{defineSelmer1}) relates $\mathrm{char}(X^-_\psi)$, the torsion part of $\mathrm{char}(X^+_\psi)$, and the
cokernel of the global-to-local map for $\boldsymbol{\psi}$.  Combining
with the main conjecture gives an analogue of Perrin-Riou's conjecture for
the virtual Heegner family $\kappa^{\mathrm{virt}}$:
\[
  2\,\mathrm{length}_P\!\left(
    \frac{H^1_+(\mathcal{K}, \boldsymbol{\psi})}
         {\Lambda^{-,\prime}_L\,\kappa^{\mathrm{virt}}}
  \right)
  \;=\;
  \mathrm{length}_P\!\left(X^+_{\psi,\mathrm{tor}}\right),
\]
where $P = (X)$ is the height-one prime corresponding to the origin $\phi_0$.

Under the hypothesis that the Selmer rank of $\psi$ is $1$, the Pontryagin
dual $X^+_\psi$ has $\Lambda^{-,\prime}_L$-rank $1$, so its torsion part
$X^+_{\psi,\mathrm{tor}}$ contributes trivially at $P$.  The Perrin-Riou
analogue then forces the left-hand cokernel to vanish at $P$, meaning that
$\kappa^{\mathrm{virt}}$ generates $H^1_+(\mathcal{K}, \boldsymbol{\psi})$
at $P$.  In particular, its specialization at $\phi_0$, which is by
construction the Heegner point $\kappa_{f,\chi}$, is non-torsion.  The
Gross--Zagier and Kolyvagin theorem then implies that the vanishing order of
$L(\psi, s)$ at $s = 1$ is exactly $1$.

\subsection{Notations and conventions}\label{NC}
We conclude the introduction by fixing some notations which will be in force throughout the paper. Fix a prime number $p$. Fix an algebraic closure $\bar{\BQ}\subset \BC$. Fix an isomorphism $\BC\cong \BC_p$, which induces an injection $\iota: \bar{\mathbb{Q}}\hookrightarrow \mathbb{C}_p$.

Fix an indefinite quaternion algebra $B$ over $\BQ$ with main involution $c$ and a maximal order $\CO_B$. Assume {\bf $B$ is split at $p$} and fix an isomorphism $B_p\cong{M_2(\BQ_p)}$ such that $\CO_{B,p}\cong M_2(\BZ_p)$. Fix an isomorphism $B_\BR\cong M_2(\BR)$.

Fix an imaginary quadratic field $\CK\subset \bar{\BQ}$ over $\BQ$ with non-trivial automorphism $c$ and associated quadratic character $\chi_{\CK}$. Let $\iota^c=\iota\circ \tilde{c}: \bar{\mathbb{Q}}\hookrightarrow \mathbb{C}_p$ where $\tilde{c}:\ \bar{\BQ}\to\bar{\BQ}$ is some involution such that $\tilde{c}|_{\CK}=c$; in particular, $\iota^c$ restricts to $\iota\circ c$ on $\CK$.
For any $d\in \BN$, let $\CO_d:=\BZ+d\CO_{\CK}$. Assume $B=\CK\oplus \CK J$ for some $J\in B^\times$ such that $Jt=t^c J$ for all $t\in \CK$ and $J^2\in \BQ^\times$. When $p$ is inert (resp. ramified) in $\CK$, take $\tau_0\in \CO_{\CK_p}$ with $p$-adic valuation $0$ (resp. $1/2$)  such that $\CO_{\CK_p}=\BZ_p[\tau_0]$. When $p$ is split in $\CK$, fix an isomorphism $\CK_p\cong\BQ_p\oplus \BQ_p$ and take $\tau_0=(1,0)$.
Assume that the induced local embeddings at $p$ and $\infty$ are
\[\CK_p\to M_2(\BQ_p),\quad a+b\tau_0\mapsto \begin{pmatrix}a & -b\\ bN_{\CK_p/\BQ_p}(\tau_0) & a+b\RTr_{\CK_p/\BQ_p}(\tau_0) \end{pmatrix}\]
\[\BC\to M_2(\BR),\quad re^{i\theta}\mapsto r\begin{pmatrix} \cos\theta & \sin\theta \\ -\sin\theta & \cos\theta\end{pmatrix}.\]

\s{\bf Coefficient fields.} Let $\mathbb{Q}_p\subset L\subset \BC_p$ be a finite coefficient field with ring of integers $\mathcal{O}_L$ and residue field $k_L$. We will also denote the induced embedding $\CK\to L$ by $\iota$.

\s{\bf Hecke characters.} We shall denote a $p$-adic Hecke character and its $\infty$-avatar by the same symbol. When necessary, we shall put the superscript $(p)$ and $(\infty)$ to distinguish them. A Hecke character $\mu$ of $\CK$ is said to have \emph{infinity type} $(a,b)$ if $\mu(z)=z^a\bar{z}^b$ for $z\in \CK_\infty^\times\cong \BC^\times$.

\s{\bf Field extensions, Galois groups and Iwasawa algebras.}  Let $\mathcal{K}_\infty^-$ be the anticyclotomic $\mathbb{Z}_p$-extension of $\mathcal{K}$ and $\Gamma_\CK^-:=\mathrm{Gal}(\mathcal{K}_\infty^-/\mathcal{K})\simeq \mathbb{Z}_p$. Let $\mathcal{K}^-_n$ be the sub-extension of $\mathcal{K}_\infty^-$ of index $p^n$ over $\mathcal{K}$. Let $\Lambda^-=\CO_L[[\Gamma_\CK^-]]$ be the Iwasawa algebra and $\Lambda_L^-=\Lambda^-\otimes L$.

Let $\mathcal{K}_\infty$ be the  $\mathbb{Z}^2_p$-extension of $\mathcal{K}$ and $\Gamma_\CK:=\mathrm{Gal}(\mathcal{K}_\infty/\mathcal{K})\simeq \mathbb{Z}^2_p$. Let $\Lambda=\CO_L[[\Gamma_\CK]]$ be the Iwasawa algebra and $\Lambda_L=\Lambda\otimes L$.

\s{\bf Adelic norms.} Let $\BA$, $\BA_\CK$ and $\BB$ be the adeles of $\BQ$, $\CK$ and $B$ respectively. Let $|\cdot|$ be the norm map on $\BQ^\times\bs\BA^\times$ and $|\cdot|_\CK=|N_{\CK/\BQ}(-)|$ be the norm map on $\CK^\times\bs \BA_\CK^\times$.

For any Hecke character $\mu$ (resp. $\xi$) on $\CK$ (resp. $\BQ$) and any $a\in\BZ$, set $\mu_a=\mu|\cdot|_\CK^a$ (resp. $\xi_a=\xi|\cdot|^a$).

\s{\bf Level subgroups.} For any $1\leq n\in\BN$, let \[U_1(n)=\begin{pmatrix} \BZ_p^\times & \BZ_p\\ p^n\BZ_p & 1+p^n\BZ_p\end{pmatrix}\subset U_0(n)=\begin{pmatrix} \BZ_p^\times & \BZ_p\\ p^n\BZ_p & \BZ_p^\times\end{pmatrix}.\]
By convention, set $U_1(0)=U_0(0)=\GL_2(\BZ_p)$.

\s{\bf Special elements.} For $(\ell,n,u)\in \BN\times\BN\times\BZ_p^\times$,  set $g_{\ell,n,u}=\begin{pmatrix} p^{2n+1-\ell-s} & p^{-\ell-s+n+1}u\\ 0 & 1\end{pmatrix}$. Here $s=1$ (resp. $s=2$) when $p$ is inert (resp. ramified) in $\CK$. When $u=1$, we simply denote $g_{\ell,n,u}$ by $g_{\ell,n}$.

\s{\bf Roots of unity.} Fix a compatible system $\{\zeta_{p^n}\}$ of primitive $p^n$-th roots of unity.

\s{\bf Haar measures.} Let $d^\times t$ be the Haar measure on $\BA^{\infty,\times} \CK^\times\bs \BA_{\CK}^{\infty,\times}$ with total volume one.

\noindent\underline{Acknowledgement}: We thank Fabrizio Andreatta, Christophe Cornut, Ming-Lun Hsieh, David Loeffler, David Rohrlich and Ye Tian for useful communications.

\section{Nearly overconvergent quaternionic modular forms and $p$-adic iteration of Gauss-Manin connections}\label{NOQMF}
  \subsection{Shimura curves}
For any (small enough) open compact subgroup 
$U\subset \BB^{\infty,\times}$, the locally symmetric space $B^\times\backslash(\BC-\BR)\times\BB^{\infty,\times}/U$  admits a canonical model  $Y_U$ over $\BQ$. It is well-known (\cite{Bra13}) that $Y_U$ solves the moduli problem which sends a $\BQ$-scheme $S$ to the isomorphism classes of triples $(A,i,\bar{\eta})$ over $S$ where 
\begin{itemize}
\item $A$ is an abelian scheme over $S$ of relative dimension $2$;
\item $i:\ \CO_B\hookrightarrow\End_S(A)$ is an injective homomorphism;
\item $\bar{\eta}$ is the $U$-orbit of $\CO_B$-linear isomorphism $\eta:\ \CO_B\otimes_{\BZ}\hat{\BZ}\to{T}(A):=\prod_{\ell}T_\ell(A)$ (locally in the \'etale topology).
\end{itemize}
 The compactification $X_U$ of $Y_U$ is a smooth proper curve over $\BQ$.  When $B\neq M_2(\BQ)$, $X_U=Y_U$  and when $B=M_2(\BQ)$,  $X_U=Y_U\sqcup\{\text{cusps}\}$.

Let $\pr:\ \CA\to Y_U$ be the universal abelian surface. Take $t\in\CO_B$ such that $t^2=-\disc(B)$. According to \cite[Page 592]{Buz97},
 there exists a unique principal polarization $$\lambda_\CA:\ \CA\to{\CA^\vee}$$ such that for any geometric point $\bar{s}\in{S}$, the Rosati involution on $\End(\CA_{\bar{s}})$ is compatible with $$-^\dagger: \ B\to{B},\quad x\mapsto{t^{-1}\bar{x}t}.$$
Take a non-trivial idempotent $e\in\CO_{B,p}$  such that $e^\dagger=e$.
 When $B=M_2(\BQ)$ and $\CO_B=M_2(\BZ)$,  we can take $e=\begin{pmatrix}1 & 0 \\ 0 & 0 \end{pmatrix}\in \CO_B$. In this case,  one recovers  the usual description of $X_U$ in terms of elliptic curves by considering $e\CA$.

Fix a small enough tame level $U^p$. For $U=U^pU_p$ with $U_p=U_0(n)$ or $U_1(n)$, denote $X_U$  by $X_0(n)$ and $X_1(n)$ respectively. In the case $U_p=U_0(n)$ (resp. $U_1(n)$), the following are equivalent:
\begin{itemize}
\item a $U$-orbit of level structure $\bar{\eta}$,
\item a $U^p$-orbit of  away-from-$p$ level structure $\eta^p:\ \CO_B\otimes_{\BZ}\hat{\BZ}^p\to{T}^{p}(A):=\prod_{\ell\neq{p}}T_\ell(A)$ together with a cyclic subgroup $D\subset e A[p^n]$ of order $p^n$(resp. a point $Q\in e A[p^n]$ of exact order $p^n$);
\item a $U^p$-orbit of  away-from-$p$ level structure $\eta^p$ and a cyclic $\CO_{B,p}$-submodule $H\subset A[p^n]$ of order $p^{2n}$(resp. a point $Q\in  A[p^n]$ of exact order $p^n$).
\end{itemize}
We shall refer to $H$ as the level subgroup (resp. submodule). Note that when $n=0$, the moduli problem for $X=X_0(0)=X_1(0)$  imposes no level structure at $p$.

 \subsection{de Rham vector bundle and Hodge line bundle}\label{deRhamhodge}

According to \cite{Car86}, 
$Y_{U^p\GL_2(\BZ_p)}$ has a canonical  smooth model $\fY$ over $\BZ_{(p)}$ solving the same moduli problem, which is proper when $B\neq M_2(\BQ)$.  Let $\pr:\ \CA\to\fY$ be the universal abelian surface and  $\CG=e\CA[p^\infty]$ be  the $p$-divisible group of dimension $1$, height $2$  cut out  by $1-e$. Then (see \cite[Section 2.4]{LZZ18})
$$\uomega:=e\pr_*\Omega_{\CA/\fY}^1,\quad \BH:=eR^1\pr_*\Omega_{\CA/\fY}^\bullet$$
are locally free $\CO_\fY$-modules of rank one and two respectively. Note that over the $p$-adic completion $\hat{\fY}$ of $\fY$,  $\uomega$ (resp. $\BH$)  coincides with the sheaf of invariant differentials of $\CG$ (resp.  the universal vector extension of $\CG$). Moreover, 
\begin{itemize}
    \item the canonical isomorphism of $\CO_B^{\op}$-sheaves $R^1\pr_*\CO_\CA\cong\Lie{\CA^\vee}$ together with the principal polarization $\lambda_\CA$ gives an isomorphism
$\uomega^{-1}\cong{eR^1\pr_*\CO_\CA}$ of $\CO_\fY$-sheaves;
\item the  Hodge-to-de Rham exact sequence 
$$0\to\pr_*\Omega_{\CA/\fY}^1\to{R^1\pr_*\Omega_{\CA/\fY}^\bullet}\to{R^1\pr_*\CO_\CA}\to0$$
induces an exact sequence of $\CO_\fY$-modules
\begin{equation*}
 0\to\uomega\to{\BH}\to\uomega^{-1}\to0;
\end{equation*}
\item the Gauss-Manin connection  $$\nabla:\ {R^1\pr_*\Omega_{\CA/\fY}^\bullet}\to{R^1\pr_*\Omega_{\CA/\fY}^\bullet}\otimes_{\CO_\fY}\Omega_{\fY/\BZ_p}^1$$
gives an integrable connection $\nabla:\ \BH\to{\BH\otimes_{\CO_\fY}\Omega_{\fY/\BZ_p}^1}$;
\item  the Kodaira-Spencer map
$$\pr_*\Omega_{\CA/\fY}^1\to{R^1\pr_*\CO_{\CA}\otimes_{\CO_{\fY}}\Omega_{\fY/\BZ_p}^1}$$
induces an isomorphism \begin{equation*}
 \KS:\ \uomega^{\otimes2}\cong\Omega_{\fY/\BZ_p}^1.
\end{equation*}
The isomorphism $eR^1\pr_*\CO_\CA\cong \uomega^{-1}$ shows that Kodaira-Spencer isomorphism is Hecke-equivariant up to twist by $|\det|$, where $\det$ is the reduced norm map on $B^\times$.
\end{itemize}

When $B=M_2(\BQ)$, let $\fX=\fY\cup\{\text{cusps}\}$ be the compactification of $\fY$, which is also smooth. It is well-known that  $\uomega$ and $\BH$ extend uniquely to $\fX$ and fit into the Hodge-to-de Rham exact sequence. Moreover the Gauss-Manin connection extends to an integrable connection $\nabla:\ \BH\to{\BH\otimes_{\CO_\fX}\Omega_{\fX/\BZ_p}^1(\log(\text{cusps}))}$ and the Kodaira-Spencer map extends to  an isomorphism $\uomega^{\otimes2}\cong\Omega_{\fX/\BZ_p}^1(\log(\text{cusps}))$. As we will work away from the cusps, we will sometimes omit  $\log(\text{cusps})$ in notations.

 \subsection{Strict neighborhood and Partial Igusa tower}

For $*=\uomega$, $\BH$, $\CG$, $\fX$, let $\bar{*}$ be  the mod-$p$ reduction of  $*$. The Verschiebung map $V:\ \bar{\CG}^{(p)}\to\bar{\CG}$ induces a section $\Ha\in{H}^0(\bar{\fX}, \bar{\uomega}^{\otimes(p-1)})$,  the \dfn{Hasse invariant}. It is known that $\Ha$ has simple zeroes over $\bar{\fX}$ (see \cite[Lemma 5.2]{Kas99}).
The \dfn{Hodge ideal} $\Hdg\subset\CO_{\hat{\fX}}$ of $\CG$ is the inverse image of $\uomega^{\otimes(1-p)}\Ha\subset\CO_{\bar{\fX}}$.
Zariski locally, the ideal $\Hdg$ is generated by two elements and if $p\in\Hdg^2$,  $\Hdg$ is invertible (see \cite[Lemma A.1]{AIP18}).

Take $2\leq r\in\BN$ and let $\fX(r)\to\fX$ be the functor sending a $p$-adically complete $\BZ_p$-algebra $S$ to the equivalence classes of pairs $(f,u)$ where $f:\Spf(S)\to \fX$ is a map and $u\in H^0(\Spf(S), f^*(\uomega^{\otimes{(1-p)r}}))$ is a section such that $u\bar{f}^*(\Ha)^r=p \mod p^{2}$.
Here two pairs $(f,u)\sim (f^\prime, u^\prime)$ are equivalent if $f=f^\prime$ and there exists $h\in S$ such that $u=(1+ph)u^\prime$. By \cite[Proposition 3.1]{AIP18} and  \cite[Lemma 3.4]{AIP18} , the functor $\fX(r)$ is representable by a formal {\em normal} scheme which is flat over $\BZ_p$.  Locally on $\Spf(R)\subset \fX$ such that $\Ha$ admits a lift, also denoted by $\Ha$, one has \[\fX(r)\times_{\fX}\Spf(R)\cong \Spf\left(R\langle u\rangle/(u\Ha^r-p)\right)\]
The adic generic fiber $\CX(r)\subset \CX$ is the open subset on which $p|\Ha|^r\geq 1$. 

By \cite[Appendix A]{AIP18}, the $p$-divisible group $\CG$ over $\fX(r)$ admits the level-$n$ canonical subgroup $\CC_n\subset \CG[p^n]$ when $r\geq (p+1)p^{n-2}$. In particular, the group $\CC_n=\ker F^n\ \mod p\Hdg^{-\frac{p^n-1}{p-1}}$. 
On the adic generic fiber $\CX(r)$,  $\CC_n^D$ is \'etale and locally isomorphic to $\BZ/p^n\BZ$.
Let $\CIG_{n}(r)$ be the adic space over $\CX(r)$ parameterizing all trivializations $\CC_n^D\cong\BZ/p^n\BZ$. Then $\CIG_{n}(r)$ is Galois over $\CX(r)$ with Galois group isomorphic to $(\BZ/p^n\BZ)^\times$.  By \cite[Lemma 3.2]{AIP18}, the normalization $\fIG_{n}(r)$ of $\CIG_{n}(r)$ over $\fX(r)$ is well-defined.  By \cite[Corollary A.3]{AIP18},  there exists an ideal $\udelta$ over $\fIG_1(r)$ such that $\udelta^{p-1}=\Hdg$. 

  The functor sending a $p$-adically complete $\BZ_p$-algebra $S$ to the set of morphisms $f:\ \Spf(S)\to\fX$ such that $f^*(\Hdg)=S$ is represented by the open formal subscheme $\fX^{\ord}\subset \fX$ on which (any local lift) $\Ha$ is invertible. We can similarly define $\fIG_n^{\ord}$.  The following result measures the difference between the strict neighborhood $\fX(r)$ and $\fX^{\ord}$.
\begin{lem}\label{cokernel}\cite[Lemma 3.4]{AI21}  Denote the natural map $\fX^{\ord}\to \fX(r)$ (resp. $\fIG_n^{\ord}\to \fIG_n(r)$) by $j$.   Then for any $h\in\BN$, the kernel of the maps \[j:\ \CO_{\fX(r)}/p^{h}\CO_{\fX(r)}\to \CO_{\fX^\ord}/p^{h}\CO_{\fX^\ord};\quad j:\ \CO_{\fIG_n(r)}/p^{h}\CO_{\fIG_n(r)}\to \CO_{\fIG_n^\ord}/p^{h}\CO_{\fIG_n^\ord} \]
are annihilated by $\Hdg^{hr}$ and $\udelta^{hr(p-1)+p^n-p}$ respectively.

\end{lem}
For later applications to estimate the convergence radii of $p$-adic iteration of Gauss-Manin connections, we also need to measure relative differentials.
\begin{lem}\label{relative differential}Let $\pr: \fIG_i(r)\to\fX$ be the structure map. Then the cokernel of the induced map 
$\pr^*(\Omega_{\fX/\BZ_p}^1)\to \Omega^1_{\fIG_i(r)}$ is annihilated by  $\udelta^{p^i-1}$.
\end{lem}
\begin{proof}The structure map $\pr$ is the composition of 
\[\fIG_i(r)\xrightarrow{p_1} \fIG_1(r)\xrightarrow{p_2} \fX\]
Note that $\fIG_{i}(r)$ is the normalization of $\fIG_{1}(r)\times_{\CC_1^D}\CC_i^D$ and by \cite[Proposition 2.2.17]{Fan}, $\mathcal{D}(\CC_i^D/\CC_1^D)=\udelta^{p^i-p}\CO_{\CC_i}$. 
Thus by the relation between different ideal and differential forms, the cokernel of \[p_1^*(\Omega_{\fIG_1(r)/\BZ_p}^1)\to \Omega_{\fIG_i(r)/\BZ_p}^1\] is annihilated by $\udelta^{p^i-p}$. Locally on $\Spf(R)\subset \fX$ such that $\Ha$ admits a lift, also denoted by $\Ha$, one has \[\fIG_1(r)\times_{\fX}\Spf(R)\cong \Spf\left(R\langle u,v\rangle/(u\Ha^r-p, v^{p-1}-\Ha)\right)\]
This explicit description implies the cokernel of $p_2^*(\Omega_{\fX/\BZ_p}^1)\to \Omega_{\fIG_1(r)/\BZ_p}^1$ is annihilated by $\Hdg$. Combining these together, we obtain the desired results.
\end{proof}

  \subsection{Vector bundles with marked sections and marked splittings}       
 \begin{defn}Take $(r,i)\in\BN^2$ such that $r\geq \frac{p^i}{p-1}(>(p+1)p^{i-2})$. Let $\gamma\in{\CC_i^D(\fIG_{i}(r))}$ be the universal section. Let $\Omega=\udelta \uomega\subset\uomega$ be the $\CO_{\fIG_{i}(r)}$-submodule generated by any lift of $s_{\HT}=\HT(\gamma)$ where $\HT$ is the Hodge-Tate map $$\HT:\ \CC_i^D(\fIG_{i}(r))\to\frac{\uomega}{p^i\uomega}.$$
  Let $\BH^\sharp=\Omega+\udelta^p\BH$ and   $\CI=p^i\Hdg^{-\frac{p^i}{p-1}}$. 
  \end{defn}
\begin{lem}\label{Canonical Splitting} Over $\fIG_i(r)$,  the  modified Hodge-to-Tate exact sequence
 \[0\to \Omega\to\BH^\sharp\to \udelta^p\uomega^{-1}\to0.\]
 admits a unique splitting modulo $\CI$ \[\BH^\sharp \mod \CI = \Omega \mod \CI \oplus Q\]
extending the unit-root splitting on $\fIG_i^{\ord}$. Moreover, let $\BH_i^\sharp$ be the modified de Rham sheaf associated to $\CG/\CC_i$ and $f_i:\ \CG\to \CG/\CC_i$
be the quotient isogeny. Then $\udelta^{-p^i}f_i^*(\BH_i^\sharp)\subset \BH^\sharp$ and one can check $\udelta^{-p^i}f_i^*(\BH_i^\sharp) \mod \CI\subset Q$.
\end{lem}
\begin{proof} As $\BH^\sharp \mod \CI$ is locally free of rank two, there can exist at most one splitting on $\fIG_i(r)$ extending the unit-root splitting.  Take    any lift $s_0$ of $s_{\HT}$ and any generator $D$ of the space of derivative on $\fX$. Then  by \cite[Lemma 5.1]{Mol21},
 the subspace $Q:=\Hdg (\nabla(D)(s_0)) (\mod \CI)$ induces the desired splitting 
 \[\BH^\sharp \mod \CI=\langle s_{\HT} \rangle \oplus Q\] 
To show  $\udelta^{-p^i}f_i^*(\BH_i^\sharp)\subset \BH^\sharp$, we apply \cite[Lemma A.4]{AI21} to $f_i^\vee$ and use the  fact $f_i\circ f_i^\vee=p^i$.  One checks directly $\udelta^{-p^i}f_i^*(\BH_i^\sharp) \mod \CI\ \bigcap\ \Omega \mod \CI=0$ and deduces $\udelta^{-p^i}f_i^*(\BH_i^\sharp) \mod \CI\subset Q$ from the uniqueness.
 \end{proof}
Now apply the machinery of vector bundles with marked sections and marked splittings (see \cite[Section 4.2]{Mol21}) to the vector bundle  $\BH^\sharp$ with the marked section $s_{\HT}$ and marked splitting $Q$ over $\fIG_{i}(r)$. One obtains the formal scheme $\CV_{0,Q}(\BH^\sharp,s_{\HT})$ over $\fIG_{i}(r)$ such that for any $\fIG_i(r)$-scheme $t:\ T\to\fIG_i(r)$,  $$\CV_{0,Q}(\BH^\sharp,s_{\HT})(T)=
   \{v\in{H^0(T,t^*(\BH^\sharp)^\vee)}|\ \bar{v}(t^*(s_{\HT}))\equiv1,\ \bar{v}(t^*(Q))\equiv0 \mod \CI \}$$
and the formal scheme $V_Q(\BH^\sharp,s_{\HT})$ (see \cite[Section 4.4]{Mol21}) such that for any $\fX(r)$-scheme $T$,
  $$V_Q(\BH^\sharp,s_{\HT})(T)=\{(\rho,v)\in\fIG_{i}(r)(T)\times{H^0(T, \rho^*(\BH^\sharp)^\vee)}\mid \bar{v}(\rho^*(s_{\HT}))\equiv 1,\ \bar{v}(\rho^*(Q))\equiv 0 \mod \CI\}.$$
 The extended torus $\CT^{\ext}:=\BZ_p^\times(1+\CI{\BG_a})$ acts on $\CV_Q(\BH^\sharp,s_{\HT})$ by the rule: for any $(\rho,v)\in\CV_Q(\BH^\sharp,s_{\HT})(T)$,
  \begin{itemize}
  \item  for any $\lambda\in{\CT:=1+\CI{\BG_a}}$, $\lambda\ast(\rho,v):=(\rho,\lambda^{-1}v)$,
  \item for any $\lambda\in\BZ_p^\times$,
$\lambda\ast(\rho,v):=(\bar{\lambda}\circ\rho, \lambda^{-1}v\circ\gamma_\lambda^{-1})$
with $\gamma_\lambda:\ \BH^\sharp\cong \bar{\lambda}^*\BH^\sharp$  characterized by $$\gamma_\lambda(s)\equiv\lambda^{-1}s \mod \CI.$$
\end{itemize}
Let $\CT^{\ext}$ act on functions via pull-back, i.e., $(t\ast{f})(\rho,v)=f(t\ast(\rho,v))$.   Let $V(\Omega,s_{\HT})$ be the formal scheme over $\fX(r)$ such that
for any $\fX(r)$-scheme $T$,
  $$V(\Omega,s_{\HT})(T)=\{(\rho,v)\in\fIG_{i}(r)(T)\times{H^0(T, \rho^*(\Omega)^\vee)}\mid \bar{v}(\rho^*(s))\equiv 1 \mod \CI\}.$$
Then by \cite[Section 4.2.1]{Mol21}, the natural map $V(\Omega,s_{\HT})\to V_Q(\BH^\sharp,s_{\HT})$ is $\CT^{\ext}$-equivariant and $\CO_{\CV_Q(\BH^\sharp,s_{\HT})}$ is equipped with a canonical filtration $\Fil_\bullet$  preserved by  the $\CT^\ext$-action such that  $\Fil_0=\CO_{\CV(\Omega,s_{\HT})}$.

\subsection{Nearly overconvergent quaternionic modular forms} \label{NO}
In this subsection, we introduce  nearly overconvergent quaternionic modular forms following \cite{AI21}.   Again assume $r\geq \frac{p^i}{p-1}$  and consider the partial Igusa tower $\fIG_i(r)$. Take a  $p$-adically complete and separated $\BZ_p$-algebra $R$ and let $\nu:\ 1+p^i\BZ_p\to R^\times$ be a continuous character such that 
 $$\nu(z)=\exp(v\log(z))=\sum_{k\geq0}{ v \choose k}(z-1)^k $$
 for all $z\in 1+p^i\BZ_p$ and some $v\in R[\frac{1}{p}]$. Note that 
$$v_p({v \choose k})>\begin{cases} k(v_p(v)-\frac{1}{p-1}) & v_p(v)<0 \\ k \frac{v_p(v)-n-1-\frac{1}{p-1}}{p^{n+1}}\  & \ n\leq v_p(v)<n+1,\ n\in\BN.\end{cases}$$
  Then for any 
 $p$-adically complete and separated  $R$-algebra $A$ with a map $\Spf(A)\to\fIG_i(r)$, 
  $\exp(v\log(-))$ is well-defined on $\CT(A)$ and $\nu$ extends to a character $ \CT(A)\to A^\times$ if 
  \[i-\frac{p^i}{(p-1)r}+\frac{v_p(v)-n-1-\frac{1}{p-1}}{p^{n+1}}>0,\quad n\leq v_p(v)<n+1\] 
Let $\nu:\ \BZ_p^\times \to R^\times$ be a continuous character whose restriction    $\nu|_{1+p^i\BZ_p}$ satisfies the conditions above. Then $\nu$ extends uniquely to a character of $\CT^{\ext}$. 
 \begin{defn}Let $\pr_1: V(\Omega,s_{\HT})\to \fX(r)$ and $\pr_2: V_Q(\BH^\sharp,s_{\HT})\to \fX(r)$  be the structure maps. Set
 \begin{itemize}
 \item $\fm^{\nu}:=\pr_{1,*}\CO_{V(\Omega,s_{\HT})}[\nu^{-1}]$ and 
$\BW^{\nu}:=\pr_{2,*}\CO_{V_Q(\BH^\sharp,s_{\HT})}[\nu^{-1}]$; 
\item For any $h\in\BN$, $\Fil_h\BW^{\nu}:=\Fil_h(\pr_{2,*}\CO_{\CV_{Q}(\BH^\sharp,s_{\HT})})[\nu^{-1}]$.
\end{itemize}
Here $[\nu]$ means taking the $\nu$-isotypic component with respect to the $\CT^{\ext}$-action. Note that $\Fil_0\BW^{\nu}=\fm^{\nu}$. 

 Let $\pr_1^0: V(\Omega,s_{\HT})\to \fIG_i(r)$ and $\pr_2^0: V_Q(\BH^\sharp,s_{\HT})\to \fIG_i(r)$  be the structure maps. Then with respect to $\CT$-action, we can define 
 \begin{itemize}
 \item $\fm^{\nu,0}:=\pr^0_{1,*}\CO_{V(\Omega,s_{\HT})}[\nu^{-1}]$ and
$\BW^{\nu,0}:=\pr^0_{2,*}\CO_{V_Q(\BH^\sharp,s_{\HT})}[\nu^{-1}]$;
\item For any $h\in\BN$, $\Fil_h\BW^{\nu,0}:=\Fil_h(\pr^0_{2,*}\CO_{\CV_{0,Q}(\BH^\sharp,s_{\HT})})[\nu^{-1}]$.
\end{itemize}
 \end{defn}


 \begin{prop}\label{GD}  The sheaf  $\BW^\nu$ and the filtration $\Fil_\bullet\BW^\nu$ are functorial  in varying   $(R,\nu)$. Moreover, 
 \begin{itemize}
 \item $\Fil_h\BW^{\nu}$ is a locally free $\CO_{\fX(r)}$-module of rank $h+1$;
 \item $\BW^{\nu}$ is the $p$-adic completion of $\varinjlim_h\Fil_h\BW^{\nu}$;
  \item If $\nu$ specializes to the $k$-th power character $$\BZ_p^\times\to\BZ_p^\times;\quad x\mapsto{x^k}$$  via $[k]:\ \Spf(\BZ_p)\to\Spf(R)$,
  then as filtered modules, $$[k]^*(\Fil_k\BW^\nu)\mid_{\CX(r)}=\Sym^k\BH\mid_{\CX(r)}$$  where  we consider the Hodge-de Rham filtration  on the right hand side.
 \end{itemize}

 \end{prop}
 \begin{proof}It suffices to work locally.  Let $\Spf(A)\subset\fIG_{1}(r)$ be an affine open subset such that over $\Spf(A)$
 \begin{itemize}
 \item $\uomega$ is generated by $\omega$ and $\BH$ is generated by $\{\omega,\eta\}$;
 \item  $\udelta$ is generated by $\delta$ and $\CI$ is generated by $\beta$;
 \item $\Omega$ is generated by $f:=\delta\omega$ and $\BH^\sharp$ is generated by $\{f, g:=\delta^p\eta\}$.
 \end{itemize}
Let $\Spf(A_i)$ be the pre-image of $\Spf(A)$ in $\fIG_i(r)$. Then locally on $\Spf(A_i)$, $\CO_{\CV_Q(\BH^\sharp,s_{\HT})}=A_i\langle Z,T\rangle$
where $X,Y,Z,T\in\CO_{\CV_Q(\BH^\sharp,s_{\HT})}$ such that for any $(\rho,v)\in\CV_Q(\BH^\sharp,s_{\HT})(R)$, $$X(\rho,v):=v(\rho^*(f)),\quad  Y(v):=v(\rho^*(g)),\quad 1+\beta Z=X,\quad Y=\beta T$$ 
 Arguing as in \cite[Lemma 2.3.4]{Fan}, locally on $\Spf(A_i)$
 \begin{equation*}
 \BW^{\nu,0}=\{\sum_{k\geq0}a_k\nu(1+\beta Z)(\frac{T}{1+\beta Z})^k\mid a_k\in{A_i\hat{\otimes}R}, a_k\to0\}\subset A_i\hat{\otimes}R\langle{Z,T}\rangle,
 \end{equation*}
 with  $$\Fil_h\BW^{\nu,0}=\{\sum_{k=0}^{h}a_k\nu(1+\beta Z)(\frac{T}{1+\beta Z})^k\mid a_k\in{A_i\hat{\otimes}R}\}$$

 Since $\udelta$ and $\CI$ are defined on $\fIG_1(r)$, the argument in \cite[Lemma 2.3.3]{Fan} shows that $1+p\BZ_p\subset \BZ_p^\times$ acts on the vector $\nu(1+\beta Z)(\frac{T}{1+\beta Z})^k$ via $\nu$. By construction 
 $\fIG_i(r)^{1+p\BZ_p}=\fIG_1(r)$. Thus the  $\nu^{-1}$-isotypic component of   $\Fil_h W^{\nu,0}$  with respect to the $(1+p\BZ_p)\CT$-action are all free of rank $h+1$ over $\fIG_1(r)$. When $p=2$, we are done since $\fIG_1(r)=\fX(r)$. When $p$ is odd, we only need to consider the residue action of $\BF_p^\times$ as in \cite[Proposition 4.9]{AI}.

 For the functorial property, see \cite[Section 2.3.3]{Fan}.
 \end{proof}

\begin{defn}The space of $R$-families of nearly overconvergent quaternionic modular forms with overconvergence radius $r$ (resp. and of order $h$), tame level $U^p$, weight $\nu$ is $$N^{\dagger,r,\nu}(U^p, R):=H^0(\fX(r)\hat{\otimes}R, \BW^{\nu}),\quad \text{resp.}\  N^{\dagger,r,\nu}_h(U^p, R):=H^0(\fX(r)\hat{\otimes}R, \Fil_h\BW^{\nu}).$$
When $h=0$, $M^{\dagger,r,\nu}(U^p, R):=N^{\dagger,r,\nu}_0(U^p, R)$ is the space of $R$-families of  overconvergent quaternionic modular forms.
\end{defn}

\subsection{Hecke operators}
We first consider Hecke operators away from $p$.
 For any $g\in\BB^{p\infty,\times}$ and tame levels $U_1^p,U_2^p\subset \BB^{p\infty,\times}$ such that $U_1^p\subset{gU_2^pg^{-1}}$, the right translation
$$T_g:\ B^\times\backslash{(\BC-\BR)}\times\BB^\times/U_1^pU_p\xrightarrow{g}B^\times\backslash{(\BC-\BR)}\times\BB^\times/U_2^pU_p;\quad [z,h]\mapsto[z,hg]$$
extends to a map
$T_g:\ \fX_{U_1^p}\to\fX_{U_2^p}.$
Since the induced isogeny $\tau_g:\ \CG\to{T_g^*(\CG)}$  has degree prime to $p$,  $T_g^*(\Hdg)=\Hdg$ and $\CC_n\cong T_g^*(\CC_n)$. Consequently the diagram $$\xymatrix{
 \CV_Q(\BH^\sharp,s_{\HT}) \ar[r]^{T_g} \ar[d] & \CV_Q(\BH^\sharp,s_{\HT}) \ar[d] \\
\fIG_{U^p_1,i}(r) \ar[d] \ar[r]^{T_g} & \fIG_{U^p_2,i}(r) \ar[d] \\
  \fX_{U^p_1}(r) \ar[d] \ar[r]^{T_g} & \fX_{U^p_2}(r) \ar[d] \\
  \hat{\fX}_{U^p_1} \ar[r]^{T_g} & \hat{\fX}_{U^p_2}   }$$
 is commutative and  compatible  with the action of $\CT^{\ext}$. 
 
 In particular for $\ell\neq p$ such that  $B_\ell$ is split and $U^p$ is maximal at $\ell$, the usual Hecke correspondence $T_\ell$ on $X$ induces the Hecke operator $T_\ell$ on $N^{\dagger,r,\nu}(R)$, which preserves each $N^{\dagger,r,\nu}_h(R)$, $h\in\BN$. Here if we fix an isomorphism $B_\ell\cong M_2(\BQ_\ell)$ such that $U_\ell\cong \GL_2(\BZ_\ell)$, \[T_\ell=\frac{1}{\ell}\left(\sum_{i\in \BZ/\ell\BZ}\begin{pmatrix}\ell & i\\ 0 & 1\end{pmatrix}+\begin{pmatrix} 1 & 0 \\ 0 & \ell\end{pmatrix}\right)\]

Now we consider the $U_p$ and $V_p$-operators.  Let $\phi_{\Fr}: \fX(pr)\to\fX(r)$ be the Frobenius map characterized by $\phi^*(\CG)=\CG/\CC_1$ and $\Phi_{\Fr}:\ \fIG_i(pr)\to\fIG_i(r)$ be the normalization of the morphism $\CIG_{i}(pr)\to\CIG_{i}(r)$ characterized by the similar  property. By \cite[Proposition 2.2.16]{Fan}, one has the following commutative diagram   $$\xymatrix{
   \fIG_{1}(pr) \ar[d]_{p_1} \ar[dr]^{p_2} \ar[r]^{\Phi_\Fr} & \fIG_{1}(r) \ar[d]^{p_1} \\
   \fX(pr) \ar[r]^{\phi_\Fr} & \fX(r)   }$$
   where $p_1$ is the forgetful map and $p_2=\phi_{\Fr}\circ{p_1}$.
   Denote by $Q$ the quotient isogeny $\CG\to\CG/\CC_1$ and $Q^\prime:\ \CG/\CC_1\to \CG$ the dual isogeny over $\fIG_i(r)$. By \cite[Lemma 2.3.17, 2.3.19]{Fan},  they induce   $$\frac{Q^*}{p}:\ p_2^*(\Omega)\cong{p_1^*(\Omega)},\quad (Q^\prime)^*:\ p_1^*(\BH^\sharp)\to{p_2^*(\BH^\sharp)}$$
and consequently morphisms
         $$\CV:\ \phi_\Fr^*(\fm^{\nu})\cong\fm^{\nu},\quad \CU:\ \BW^{\nu}\to\phi_\Fr^*(\BW^{\nu}). $$
         Here the latter is compatible with  the filtration $\Fil_\bullet\BW^{\nu}$.
   \begin{defn}\label{UV}One defines the $V_p$-operator to be  the composition
$$\CV\circ\phi_\Fr^*:\ M^{\dagger,r,\nu}(R)\to M^{\dagger,pr,\nu}(R)$$
and the $U_p$-operator on $N^{\dagger,r,\nu}(R)$ to be the  composition
$$j\circ{\frac{1}{p}\RTr}\circ\CU:\ N^{\dagger,r,\nu}(R)\to N^{\dagger,r,\nu}(R)$$
where $j$ is induced by the natural embedding $\fX(pr)\to\fX(r)$ and $\RTr$ is the trace map of $\phi_{\Fr}$. 

By construction,  $U_p$ respects the filtration on $N^{\dagger,r,\nu}(R)$ and  the composition $$U_p\circ{V_p}:\ M^{\dagger,r,\nu}(R)\to M^{\dagger,pr,\nu}(R)$$
coincides with the natural embedding $j$.
\end{defn}

Take any finite order character $\xi:\ \BZ_p^\times\to R^\times$ and assume $r>(p+1)p^{2\Cond\xi-2}$. By \cite[Section 3.6]{AI} and \cite[Lemma 2.3.29 \& Corollary 2.3.30 \& Proposition 2.4.17]{Fan}, there exists a twist-by-$\xi$ operator $\theta^\xi$ with the following properties: 
   \begin{itemize}
       \item $\theta^{\xi}(N_h^{\dagger,r,\nu}(R)[1/p])\subset{N_h^{\dagger,r,\nu+2\xi}(R)[1/p]}$,
       \item $U_p\circ\theta^\xi=0$ and if $U_p(F)=0$, then $\theta^{\xi^{-1}}\theta^\xi(F)=F$.
   \end{itemize}
We refer to \cite[Section 4.3]{BH24} for the nice adelic formulation of twisting by finite order characters.

\subsection{$p$-adic iteration of Gauss-Manin connections}
Let $\CIG_{i}^\prime(r)$ be the finite \'etale Galois cover of  $\CX(r)$ which parameterizes  compatible trivializations $$(\BZ/p^i\BZ)^2\cong{\CG^D[p^i]},\quad \BZ/p^i\BZ\cong{\CC_i^D}.$$  Since $\fX(r)$ is normal, the normalization $\fIG_{i}^\prime(r)$ of $\fX(r)$ in $\CIG_{i}^\prime(r)$ is well-defined and finite over $\fX(r)$. 
\begin{lem}\label{GMI}\cite[Proposition A.3]{AI21}  The Gauss-Manin connection $\nabla$ on $\BH$ induces a connection
 $$\nabla^\sharp:\ \BH^\sharp\to\BH^\sharp\otimes\Omega^1_{\fIG_{i}^\prime(r)/\BZ_p}$$
which satisfies the Griffiths transversality and $\nabla^\sharp\mid_{\Omega}\equiv0 \mod \CI$. 
\end{lem}
Define $\CV_Q(\BH^\sharp,s_{\HT})$ over $\fIG_{i}^\prime(r)$.
Then by \cite[Section 4]{Mol21}, we obtain a connection
$$\nabla^\sharp:\ \CO_{\CV_Q(\BH^\sharp,s_{\HT})}\to{\CO_{\CV_Q(\BH^\sharp,s_{\HT})}}\otimes\Omega_{\fIG_{i}^\prime(r)/\BZ_p}^1$$
which satisfies the Griffiths transversality. 
 Take any $R$-valued locally analytic character $\nu$ as in Subsection \ref{NO}.
\begin{prop}\label{GM}The connection $\nabla^\sharp$  induces  a  connection
$$\nabla:\ \BW^{\nu}\to \BW^{\nu}\hat{\otimes}_{\CO_{\fX(r)}}\Omega^1_{\fX(r)/\BZ_p}[1/p]$$
which satisfies the Griffiths transversality.
\end{prop}
\begin{proof}We have a natural candidate $\nabla^\sharp$ on $\BH^\sharp$, which admits the  following local description. Choose  $\Spf(A)$ and   $\{\omega,\eta, f,g\}$ as in Proposition \ref{GD}.
Assume the Gauss-Manin connection $\nabla$ on $\fX$ satisfies
$$\nabla(\omega)=\omega\otimes{x}+\eta\otimes{y},\quad \nabla(\eta)=\omega\otimes{z}+\eta\otimes{w}$$
for some $x,y,z,w\in\Omega_{\fX/\BZ_p}^1$.
Then for $f=\delta\omega$ and $g=\delta^p\eta$
$$\nabla^\sharp{f}=f\otimes a_0+g\otimes{b};\quad \nabla^\sharp(g)=f\otimes{c}+g\otimes{d_0}$$
with
$$a_0:=x+\frac{d\delta}{\delta},\ b:=\frac{y}{\delta^{p-1}},\ c:=\delta^{p-1}z,\ d_0:=w+p\frac{d\delta}{\delta}.$$ 
Linearize $\nabla^\sharp$  as in \cite[Section 2.4]{AI21}. For $X=1+\beta Z$, $V=\frac{T}{1+\beta Z}$ and $v=\lim_{t\to1}\frac{\log(\nu(t))}{t}$,  by \cite[Lemma 2.3.9]{Fan} 
\begin{equation}\label{key}
\nabla^\sharp(X^{\nu}V^h)=(v-h)X^{\nu}V^h\otimes{a_0}+hX^{\nu}V^h\otimes{d_0}+(v-h)\beta X^{\nu}V^{h+1}\otimes{b}+\beta^{-1}hX^{\nu}V^{h-1}\otimes{c}.
\end{equation}
  Together with the local description in Proposition \ref{GD}, we obtain  
\[\nabla^\sharp:\ \BW^{\nu,0}\to \BW^{\nu,0}\hat{\otimes}_{\CO_{\fIG_i(r)}}\Omega^1_{\fIG_{i}(r)/\BZ_p}[1/p]\]
From the residue action of $(\BZ/p^i\BZ)^\times$, we obtain the desired connection $\nabla$ on $\BW^\nu$ by descent.  
\end{proof}
 Consider the  multiplication
$$ \BW^{\nu}\otimes_{\CO_{\fX(r)}}\fm^{2}\to\BW^{\nu+2}$$
induced by the
multiplicative structure on $\CO_{\CV_Q(\BH^\sharp,s_{\HT})}$.
Together with the Kodaira-Spencer isomorphism $\fm^2[1/p]\cong \uomega^2[1/p]\cong \Omega_{\fX(r)/\BZ_p}^1[1/p]$ on $\fX(r)$,
we obtain a map  
$\nabla: \BW^\nu\to \BW^{\nu+2}[1/p]$.
 \begin{prop}The $R$-linear map $\nabla:\ N^{\dagger,r,\nu}(R)\to{N^{\dagger,r,\nu+2}(R)[1/p]}$
 satisfies that:
 \begin{enumerate}
 \item[(i)] For any $h\in\BN$, $\nabla(N_h^{\dagger,r,\nu}(R))\subset N_{h+1}^{\dagger,r,\nu+2}(R)[1/p]$ ;
 \item[(ii)] When $T_\ell$ is defined, $T_\ell\circ\nabla=\ell\nabla\circ{T_\ell}$;
 \item[(iii)]  $U_p\circ\nabla=p\nabla\circ{U_p}$;
 \item[(iv)] For any finite order character $\xi:\ \BZ_p^\times\to R^\times$, $\nabla_{\nu+2\xi}\circ\theta^{\xi}=\theta^{\xi}\circ\nabla_{\nu}$.
 \end{enumerate}
 Moreover if the $k$-th power character is a specialization of $\nu$, then the isomorphism
 $$[k]^*(\Fil_k\BW^\nu)\mid_{\CX(r)}\cong\Sym^{k}\BH\mid_{\CX(r)}$$
 identifies $[k]^*(\nabla)$ with the composition
 $$\Sym^k\BH\xrightarrow{\nabla}\Sym^k\BH\otimes\Omega^1_{\fX(r)}\xrightarrow{\KS^{-1}}\Sym^k\BH\otimes\uomega^{\otimes2}\hookrightarrow\Sym^{k+1}\BH.$$
 \end{prop}
 \begin{proof} Note that $\nabla$ satisfies Griffiths transversality, which implies Item $(i)$. Item (ii) follows from the functoriality of the Gauss-Manin connection and the properties of  the Kodaira-Spencer map.  Item $(iii)$ and Item $(iv)$ can be proven by expansion in the Serre-Tate local coordinates.  We refer to \cite[Proposition 2.4.15(ii)]{Fan}  and
\cite[Proposition 2.4.17]{Fan} for details. Finally,  the local computation in the proof of Proposition \ref{GM}  implies the identification.
\end{proof}

The following observation is crucial for $p$-adically interpolating $\nabla$. We work over $\fIG_i(r)$.
\begin{prop}\label{Integral property} For each $N\in\BN$,
$\Hdg^{cN}\nabla^N( \fm^{\nu,0})\subset\BW^{\nu+2N,0}$ with $c=
\frac{p^i+1}{p-1}$.
\end{prop}

\begin{proof}Let $\Pr:\ \fIG_i(r)\to\fX(r)$ be the structure map.
The Kodaira-Spencer isomorphism $\uomega^2\cong \Omega_{\fX/\BZ_p}^1$ on $\fX$ induces $\udelta^2\Pr^*(\Omega_{\fX/\BZ_p}^1)\cong\Omega^2$ over $\fIG_i(r)$. Combining the key identity \eqref{key} and Lemma \ref{relative differential}, we have $\Hdg^{c}\nabla(\fm^{\nu,0})\subset \BW^{\nu+2,0}$. To conclude the proof, we apply the key identity \eqref{key} to do induction. The desired result follows by noting that in the induction process,  terms containing $\beta^{-1}$ will not contribute since we start with $h=0$.
\end{proof}
Now we turn to the $p$-adic iteration of the Gauss-Manin connection. Take $(r,i)\in\BN^2$ such that $r\geq (p+1)p^{2i-2}$.
\begin{itemize}
\item Let $\nu:\ \BZ_p^\times \to R^\times$ be a locally analytic character such that there exists $v\in R[\frac{1}{p}]$ such that $\nu(z)=\sum_{k\geq0}{v\choose k}(z-1)^k$ for all $z\in 1+p^i\BZ_p$ and 
\[i-\frac{p^i}{(p-1)r}+\frac{v_p(v)-n-1-\frac{1}{p-1}}{p^{n+1}}>0,\quad n\leq v_p(v)<n+1\] 
This guarantees that $\nu$ extends to a character of $\CT^{\ext}$.
\item Take  $F\in M^{\dagger,r,\nu}(R)$ such that $U_p(F)=0$.

\item Let $\mu:\ \BZ_p^\times\to R^\times$ be another  locally analytic character such that there exists $u\in R[\frac{1}{p}]$ such that $$\mu(z)=\exp(u\log(z))=\sum_{k\geq0}{ u \choose k}(z-1)^k, \quad  \forall\ z\in 1+p^{i}\BZ_p $$
\item  Let $c=
\frac{p^i+1}{p-1}$ and $d=\begin{cases} v_p(u)-\frac{1}{p-1},\ & \text{if}\ v_p(u)<0;\\ \frac{v_p(u)-n-1-\frac{1}{p-1}}{p^{n+1}}\  & \ n\leq v_p(u)<n+1,\ n\in\BN.\end{cases}$
\end{itemize}

\begin{prop}\label{explicit iteration}Assumptions and notations as above. Then there exists an explicitly constructed element
 $\nabla^\mu(F)\in N^{\dagger, r^\prime, \nu+2\mu}(R)[1/p]$ with $r^\prime>\frac{i^2r+c}{i+d}$  such that for any classical specialization $k$ and $\ell$ of $\nu$ and $\mu$, one has 
 $$[k+2\ell]^*\nabla^\mu(F)=\nabla_k^{2\ell}([k]^* F).$$

\end{prop}
\begin{proof}
By the property of twisting by finite order character and the definition of $\nabla^\mu$, we may view $F$ as a section of $H^0(\fIG_i(r),\fm^{\nu})$ and work on $\fIG_i(r)$. By the assumption $r\geq (p+1)p^{2i-2}$, we have the twist-by-$\xi$ operator $\theta^\xi$ for any character $\xi$ of $(\BZ/p^i\BZ)^\times$ after base change to $\BZ_p[\zeta_{p^i}]$ and taking normalization. Thus we  follow \cite[Section 6.5]{Mol21} to  consider 
\[\nabla^\mu:=\sum_{\gamma\in (\BZ/p^i\BZ)^\times}\mu(\gamma)\sum_{k\geq0}{u \choose k} (\frac{\nabla-\gamma\mathrm{Id}}{\gamma})^k \circ \nabla^{\gamma+p^i\BZ_p}\] where $\nabla^{\gamma+p^i\BZ_p}=\frac{\sum_\xi \xi(\gamma)^{-1}\theta^\xi}{\sharp(\BZ/p^i\BZ)^{\times}}$ with the summation runs over all characters of $(\BZ/p^i\BZ)^\times$. 
Now it suffices to show that $\{{w \choose k} (\frac{\nabla-\gamma\mathrm{Id}}{\gamma})^k \circ \nabla^{\gamma+p^i\BZ_p}(F)\}$ forms a Cauchy sequence when restricted to $\fIG_i(r^\prime)$. 

By  \cite[Proposition 2.4.16]{Fan} (see also \cite[Section 6.4]{Mol21}), locally on the ordinary locus  $\fIG_i^{\ord}$ and in the Serre-Tate local coordinates (see \textit{loc.cit.,} for notations)
\[\nabla^{\gamma+p^i\BZ_p}(a_n(1+t)^nU_{\nu,h})=\begin{cases} a_n(1+t)^n U_{\nu,h},&\ n\equiv \gamma\pmod {p^i};\\
0 &\ \text{otherwise}\end{cases}\]
where $U_{\nu,h}:=(1+p^i Z)^\nu \frac{T^h}{(1+p^i Z)^h}$. 
Note we have $$\nabla^N(g(t)U_{\nu,h}):=\sum_{j=0}^Np^{ij}a_{N,\nu,h,j}\partial^{N-j}(g(t))U_{\nu+2N,j+h}$$
where $\partial=(1+t)\frac{d}{dt}$, $a_{N,\nu,h,j}=1$ if $j=0$ and for $1\leq{j}\leq{N}$,
$$a_{N,\nu,h,j}={N\choose{j}}\prod_{t=0}^{j-1}(v-h+N-1-t).$$
By Proposition \ref{Integral property} and induction on $k$, 
\[\Hdg^{ck}(\frac{\nabla-\gamma\mathrm{Id}}{\gamma})^k\circ\nabla^{\gamma+p^i\BZ_p}(F)\in H^0(\fIG_i(r),\oplus_{j=0}^k\BW^{\nu+2j})\cap p^{ik}H^0(\fIG_i^{\ord},\oplus_{j=0}^k\BW^{\nu+2j})\]
By Lemma \ref{cokernel}, we deduce
\[\Hdg^{ck+i^2kr+\frac{p^i-p}{p-1}}(\frac{\nabla-\gamma\mathrm{Id}}{\gamma})^k\circ\nabla^{\gamma+p^i\BZ_p}(F)\in p^{ik}H^0(\fIG_i(r),\oplus_{j=0}^k\BW^{\nu+2j})\]
By the estimation of $v_p({u\choose k})$, we find that  when  restricted to $\fIG_i(r^\prime)$,
\[i+\frac{v_p({u \choose k})}{k}-(c+i^2r)v_p(\Hdg)>0\]
and consequently $\{{u \choose k} (\frac{\nabla-\gamma\mathrm{Id}}{\gamma})^k \circ \nabla^{\gamma+p^i\BZ_p}(F)\}$ forms a Cauchy sequence.
For the interpolation property, we refer to \cite[Proposition 2.5.2]{Fan}.
\end{proof}
For $*=0,1$, consider the rigid analytic curves $\CX_*(n)$  associated to $X_*(n)$ and  for $r>(p+1)p^{n-2}$, let $\CX_*(n;r)_c$  be the canonical loci of $\CX_*(n)$, i.e., the open admissible subspace on which the level subgroup  coincides with the canonical subgroup (of level $n$). It is well-known that the morphism 
 \[\gamma:\ \CX(r) \to \CX_0(n;r)_c,\quad A\mapsto (A,\CC_n)\]
is actually an isomorphism which is compatible with the $U_p$-action on both sides. It induces a $U_p$-compatible isomorphism $\CIG_n(r) \cong\CX_1(n;r)_c$. Via $\gamma$, we can transfer all the results concerning nearly overconvergent quaternionic modular forms on $\CX(r)$ and $\CIG_n(r)$ to $\CX_0(n;r)_c$ and $\CX_1(n;r)_c$.

\subsection{Theory on the half canonical loci}\label{Half canonical}
In this subsection,  we adapt the beautiful idea of Andreatta-Iovita in the first arXiv version of \cite{AI}{\footnote{This idea is no longer necessary for the published version due to the restriction on ramifications of the Hecke characters and consequently the  specific  choice of test vectors.}} to extend the theory of nearly overconvergent quaternionic modular forms to a slightly larger domain, which we refer to as the half canonical loci. We shall work over $\BQ_p(\zeta_{p^n})$. Take $n/2\leq d\leq n$ and let $\CX_0(n;r)_c^{\geq d}$ be the preimage of $\CX_0(d;r)_c$ in $\CX_0(n)$. Let $X_0(n,d)=X_0(n)\times_{X_0(d)}X_1(d)$ and $\CX_0(n,d;r)=\CX_0(n;r)\times_{\CX_0(d;r)}\CX_1(d;r)$.  Let $\CX_0(n,d;r)_c^{\geq d}$ be the preimage of $\CX_0(n;r)^{\geq d}_c$.

Assume that $r>(p+1)p^{2n-2d-2}$. Then the universal false  elliptic curve $\CA$ on $\CX(r)$ admits the canonical subgroup of level $\max\{2(n-d),n\}$.  For $d\leq k\leq n$, let \[H^k=(\CA/\CC_{2k-n})[p^{n-k}]/(\CC_{k}/\CC_{2k-n})\subset \CA^{k}=\CA/\CC_{k}.\] Then $\CA^k$ admits canonical subgroup $\CC_{n-k}^k$ of level $n-k$ and $\CA^k[p^{n-k}]\cong \CC_{n-k}^k\times H^k$.  Moreover by \cite[Lemma 2.3.28]{Fan}  and \cite[Section 3.8]{AI21}, our fixed root of unity induces a  bijection over $\CX_0(n,d;r)_c^{\geq d}$
 $$\eta:\ \Hom(H^k,\CC_{n-k}^{k})\cong\Hom(\BZ/p^{n-k}\BZ,\mu_{p^{n-k}})\cong \BZ/p^{n-k}\BZ$$
where the last isomorphism is the inverse of 
$$\BZ/p^{n-k}\BZ\cong\Hom(\BZ/p^{n-k}\BZ,\mu_{p^{n-k}}),\quad a\mapsto(1\mapsto\zeta_{p^{n-k}}^a).$$
For any $j\in(\BZ/p^{n-k}\BZ)^\times$, consider the isogeny $\rho_j$ corresponding to $j$ and set $H_j^k:=(\rho_j\times\id)(H^k)$. Let $C_{k,j}$ be the preimage of $H_j^k$ in $\CA$, which is generically cyclic.  
\begin{lem} Assume  $r>(p+1)p^{2n-2d-2}$ and  let $S_{\geq d}=\sqcup_{k=d}^n S_k$ with $S_k=\{(k,j)|j\in (\BZ/p^{n-k}\BZ)^\times\}$. For $(k,j)\in S_{\geq d}$, let $\gamma_{k,j}$ be the map \[\CIG_d(r)\to \CX_0(n,d;r),\quad \CA\mapsto (\CA, C_{k,j})\] and $\CX_0(n,d;r)_{k,j}$ be the image of $\gamma_{k,j}$.  Then one has an isomorphism \[\sqcup_{(k,j)\in S_{\geq d}} \CIG_d(r)\xrightarrow{\gamma=(\gamma_{k,j})} \sqcup_{(k,j)\in S_{\geq d}}\CX_0(n,d;r)_{k,j}\cong \CX_0(n,d;r)_c^{\geq d}\] 
Moreover, \begin{itemize}
    \item $U_p(\CX_0(n,d;r)_{k,j})\subset \CX_{0}(n,d;r)_{k+1,j}$ for $k\leq n-1$ and $U_p(\CX_0(n,d;r)_{n,j})\subset \CX_{0}(n,d;r)_{n,j}$
    \item $\gamma_{k+1,j}^{-1}\circ U_p\circ \gamma_{k,j}$ is just the $U_p$-correspondence on $\CIG_d(r)$.
\end{itemize}
\end{lem}
\begin{proof}Clearly by construction, $\CX_0(n,d;r)_{k,j}\subset \CX_0(n,d;r)_c^{\geq d}$ and $\CX_0(n,d;r)_{k,j}$ are all disjoint. Note that $U_0(d)/U_0(n)$ has a complete set of representative $\sqcup_{d\leq k\leq n}\{\begin{pmatrix} 1 & 0\\ p^k j & 1\end{pmatrix}|  (j\in \BZ/p^{n-k}\BZ)^\times\}$. Thus by the moduli interpretation of $X_0(n,d)$, a counting argument shows that  \[\sqcup_{(k,j)\in S_{\geq d}}\CX_0(n,d;r)_{k,j}\cong \CX_0(n;r)_c^{\geq d}\]and the map $\gamma=\{(\gamma_{k,j})\}$ gives an isomorphism $\sqcup_{(k,j)\in S_{\geq d}} \CIG_d(r)\cong \CX_0(n,d;r)_c^{\geq d}$.

For the final statement, take any point $P=(A,H, Q)\in X_0(n,d;r)_{k,j}$ where $H$ is the level subgroup and $Q\in H[p^d]$ is a generator. Then  the $U_p$-correspondence  $U_p(P)=\{(A_i, H_i, Q_i)\}$ where $A_i$'s are $A$ quotient out by the $p$-cyclic $\CO_B$-module  which intersects the canonical subgroup trivially and $H_i$, $Q_i$ is the image of $H$ and  $Q$.  By the functorial property of canonical subgroups, we deduce $H_i[p^{k+1}]$ is the canonical subgroup of $A_i$ of level $k+1$ as $A_i/H_i[p]\cong A$.  Thus $U_p(P)\subset X_0(n,d;r)_{k+1,j}$ and the final statement follows.
\end{proof}
\begin{remark}Consider the $p$-level $U_0(n)\cap U_1^1(d)$ Shimura curve  $X_U$ where \[U_1^1(d)=\{\begin{pmatrix} a& b\\ c & e\end{pmatrix}\in\GL_2(\BZ_p)| a-1\equiv e-1\equiv c\equiv 0 \mod p^d\}\]
Then the chosen root of unity determines a splitting of the covering map $\CX_U\to \CX_0(n,d)$ on $\CX_0(n,d;r)_c^{\geq d}$. View $\CX_0(n,d;r)_{k,j}$ as subspace of $\CX_U$ via this splitting. Then $\gamma_{k,j}=T_{\gamma_{k,j}}\circ\gamma$ where  $T_{\gamma_{k,j}}$ is the Hecke translation given by $\gamma_{k,j}=\begin{pmatrix} 1 & 0 \\ p^kj & 1\end{pmatrix}$ and we deduce the statement on $U_p$-correspondence above via the Hecke translation description $U_p=\frac{1}{p}\sum_{i\in\BZ/p\BZ}\begin{pmatrix}p & i\\ 0 & 1\end{pmatrix}$.
\end{remark}

 The isomorphism $\gamma$ extends the theory of nearly overconvergent quaternionic modular forms  to  $\CX_0(n,d;r)_c^{\geq d}$. In particular for any character $\nu$, $\mu$ as in previous sections, there exists the nearly overconvergent quaternionic modular sheaf $\fm^\nu\subset \BW^\nu$ on  $\fX_0(n,d;r)^{\geq d}$ interpolating $\omega^k$ and $\Sym^k \BH$ for  $k\in\BN$ and if $F\in H^0(\CX_0(n,d;r)_c^{\geq d}, \fm^\nu)$ satisfying $U_p(F)=0$, there  exists  $\nabla^\mu F\in H^0(\CX_0(n,d; r^\prime)^{\geq d}_c,\BW^{\nu+2\mu})$ for $r^\prime$ as in Proposition \ref{explicit iteration}.

\section{Anticyclotomic $p$-adic $L$-function}\label{Anticyclotomic p-adic L-function}
In this section, we construct the one and two-variable anticyclotomic $p$-adic $L$-function for  Rankin pairs. Besides the notations fixed in Subsection \ref{NC}, fix a finite order Hecke character $\omega=\otimes_v\omega_v$ over $\BQ$. 


\subsection{CM points on Shimura curve}\label{CM points}
 Attached to the data $(B,\CO_B)$, we have the Shimura curve $\varprojlim_U X_U$. The embedding $\CK^\times\to B^\times$  determines a subvariety, which is defined over $\CK$. We shall refer to points in this subvariety  by CM points. Take a small enough tame level $U^p\subset \BB^{p\infty,\times}$. Let $X_0(n)$ (resp. $X_1(n)$) be the Shimura curve $X_{U^p U_p}$ for $U_p=U_0(n)$ (resp. $U_1(n)$)  (see Section \ref{deRhamhodge}). For $n=0$, denote $X_0(n)=X_1(n)$ by $X$. 
 
 Recall that   $X_0(n)$ (resp. $X_1(n)$) parameterizes quadruples $(A,i,\bar{\eta},H)$ (resp. $(A,i,\bar{\eta},Q)$) where $(A,i)$ is a false elliptic curve, $\bar{\eta}$ is the $U^p$-orbit of $\CO_B$-linear isomorphism $\hat{\CO}_B^p\cong T^p(A)$ and  $H\subset A[p^n]$ is a sub-$\CO_B$-module  of rank $p^{2n}$ (resp. $Q\in A[p^n]$ is a point of exact order $p^n$).
  Unraveling  the identification between group-theoretic interpretation and moduli interpretation of Shimura curve, the (false) elliptic curve $(A,i)$ underlying a CM point $P\in X_0(n)$  has complex multiplication by $\CO_\CK\cap\CO_B=\CO_d:=\BZ+d\CO_\CK$ for some positive integer $d$ prime to $p$. Assume that $L/\BQ_p$ is large enough so that $(A,i,H)$ extends  to a smooth model over $\CO_L$ and denote by $\Ha(A)$ its Hasse invariant.

When $p$ splits in $\CK$, all CM points lie in the ordinary locus (see \cite[Lemma 2.4.2]{LZZ18}{\footnote{The embedding in \cite{LZZ18} is $(a,b)\mapsto \begin{pmatrix} a & 0\\ 0 & b\end{pmatrix}$. Its  conjugation by $\begin{pmatrix} 0 & 1\\-1 & 0\end{pmatrix}\begin{pmatrix}1 & 1\\ 0 & 1\end{pmatrix}$ is our embedding.}}).
Now we analyze the  Hasse invariants of the CM points when  $p$ is non-split.
\begin{lem}\label{CMI}Assume $p$ is non-split in $\CK$ and take $n\geq1$. Then for each CM point in $X$, there are precisely $C_0=\begin{cases} (p+1)p^{n-1} & p\ \text{inert}\\ p^{n-1} & p\ \text{ramified}\end{cases}$ (resp. $p^{n-1}(p-1)C_0$) CM points lying over it in  $X_0(n)$ (resp. $X_1(n)$). 
\end{lem}
\begin{proof}
When $p$ is inert, $U_0(n)\cap \CO_{\CK_p}^\times=\CO_{\CK_p,p^n}^\times$  and 
when $p$ is ramified, $U_0(n)\cap \CO_{\CK_p}^\times=\CO_{\CK_p,p^{n-1}}^\times$. Here 
 $\CO_{\CK_p,p^k}:=\BZ_p+p^k\CO_{\CK_p}$ for $k\geq0$. Let $V^p=U^p\cap \hat{\CO}_{\CK}^{p,\times}$. Then CM points at level $U^pU_0(n)$ can be identified with $\CK^\times\bs \BA_{\CK}^{\infty,\times}/V^p\CO^\times_{\CK_p,p^n}$ (resp. $\CK^\times\bs \BA_{\CK}^{\infty,\times}/V^p\CO^\times_{\CK_p,p^{n-1}}$).
As we assume that $U^p$, hence $V^p$, are small enough, the fiber of the map \[\CK^\times\bs \BA_{\CK}^{\infty,\times}/V^p\CO^\times_{\CK_p,p^n}\to \CK^\times\bs \BA_{\CK}^{\infty,\times}/V^p\CO^\times_{\CK_p}\]
has size $C_0$. The case for $X_1(n)$ is similar.
\end{proof}
 Let $(A,i)$ be the (false) elliptic curve  underlying a CM point $P\in X_0(n)$. 
\begin{lem}\label{CMII} When $p$ is inert (resp. ramified), $\val_p(\Ha(A))\geq \frac{p}{p+1}$(resp. $\val_p(\Ha(A))=1/2$). When $p$ is ramified, $A$ has canonical subgroup of level $1$ and the level subgroup $H$ at $U_0(1)$-level coincides with the canonical subgroup.
\end{lem}
\begin{proof}Note the Hasse invariant only depends on the underlying $p$-divisible subgroup $\CG$. Thus the first statement follows from \cite[Lemma 3.1 \& 3.2]{AI}.  When $p$ is ramified, $\CG$ has canonical subgroup $\CC$ of level one. By Lemma \ref{CMI},  there is only one  $\CO_B$-submodule $H$ of rank $p^2$ appearing at the $U_0(1)$-level. By the functorial property of canonical subgroups and CM theory respectively, both $\CC$ and $H$ are stable under the $\CO_{\CK_p}$-action. If $\CC\cap eH=(0)$, then $\CG/eH$ has  canonical subgroup of level $p^2$, which is impossible by the functorial property of canonical subgroup and the $\CO_{\CK_p}$-action. Thus $\CC=eH$ and the stated result follows.
\end{proof}

To define the anticyclotomic $p$-adic $L$-function, we need to  evaluate $\nabla^\mu(F)$ at CM points. This requires the CM points 
to lie in the strict neighbourhood $\mathcal{X}(r')$ for the $r^\prime$ given by Proposition \ref{explicit iteration}; equivalently, $v_p(\mathrm{Ha}(A)) \leq 1/r'$. However by Lemma \ref{CMII},  the Hasse invariant of a CM point satisfies $v_p(\mathrm{Ha}(A)) \geq p/(p+1)$ (resp. $=1/2$) when $p$ is inert (resp. ramified), regardless of the level. In view of the local test vector theory, this bound prevents evaluating at CM points for Hecke characters with small $p$-adic conductor — the arithmetic characters of primary interest — and makes a straightforward generalization of \cite{AI} impossible. The following lemma (and Appendix \ref{Local test vector}) shows that we can resolve this issue by a simple translation: for any CM point $P$
and any target radius $r'$, one can find $g=g_{\ell,n,u}$ such that $T_g(P)$ lies in the half-canonical locus $\mathcal{X}_0(2n,n;\ell)_c^{\geq n}$ where the theory of nearly overconvergent forms is extended in  Subsection \ref{Half canonical}.

 Note that  for $g=\begin{pmatrix} p^{-1} & 0\\ 0 & 1\end{pmatrix}$, the $g$-translation $T_g$ induces the map \[V_p:\ X_0(n+1)\to X_0(n),\quad (A, i, \bar{\eta}, H)\mapsto (A/H[p], i, \bar{\eta}, H/H[p])\] 

\begin{lem}\label{CMIII} Take $(\ell,n,u)\in \BN\times \BN\times\BZ_p^\times$ such that $\ell\geq 2n$. Let \[g:=g_{\ell,n,u}=\begin{pmatrix} p^{1-s} & 0\\ 0 & 1\end{pmatrix}\begin{pmatrix}p^{2n-\ell} & up^{-\ell+n}\\ 0 & 1\end{pmatrix}\] where $s=2$ (resp. $1$) when $p$ is ramified (resp. inert). Then for any CM point $P\in \varprojlim_UX_U$, $T_g(P)\in \CX_0(2n,n;\ell)$.
\end{lem}
\begin{proof}Assume that $T_g(P)\in X_0(2n)$ corresponds to a quadruple $(A_g,i,\bar{\eta}, H_g)$ and let $\CC_{g,i}\subset A_g$ be the canonical subgroup of level $i$.  We claim that  
\begin{itemize}
    \item $\val_p(\Ha(A_g))=\frac{1}{(p+1)p^{\ell-1}}$(resp. $\frac{1}{2p^{\ell}}$) when $p$ is inert (resp. ramified),
    \item $H_g\cap \CC_{g,2n}=\CC_{n}$ and $A_g/\CC_{g,\ell}=A$. 
\end{itemize}
Clearly, this claim implies that $T_g(P)\in \CX_0(2n,n;\ell)_c^{\geq n}$. To see the claim, 
let $m=\ell-n$. Assume the image of $T_{s_0}P$ in $X_0(2m)$ corresponds to $(A,i,\bar{\eta},H)$ and  the point $T_h(T_{s_0}P)\in X_0(2m)$ corresponds to a quadruple $(A^\prime,i^\prime,\bar{\eta}, H^\prime)$. Here $s_0=\begin{pmatrix}p^{1-s-m} & 0\\ 0 & 1\end{pmatrix}$ and $h=\begin{pmatrix} 1 & up^{-m}\\ 0 & 1\end{pmatrix}$. Note that $H$ intersects the canonical subgroup of $A$ trivially by Lemma \ref{CMII} and the property of $V_p$. Let $\CC_{i}^\prime\subset A^\prime$ be the  level-$i$ canonical subgroup. From the property of $V_p$, it suffices to show that 
\begin{itemize}
    \item $\val_p(\Ha(A^\prime))=\frac{1}{(p+1)p^{2m-1}}$(resp. $\frac{1}{2p^{2m}}$) when $p$ is inert (resp. ramified),
    \item $H^\prime\cap \CC_{2m}^\prime=\CC_m^\prime$ and $A^\prime/\CC^\prime_{2m}\cong A$. 
\end{itemize}
Let $\tilde{A}=A/H[p^m]$ and $\tilde{H}=H/H[p^m]$, $\tilde{H}^\prime=A[p^m]/H[p^m]$.
Then $\tilde{A}[p^m]=\tilde{H}\times \tilde{H}^\prime$ and  $\Isom(\tilde{H}^\prime,\tilde{H})\cong (\BZ/p^m\BZ)^\times$. For $i\in (\BZ/p^m\BZ)^\times$, let $\rho_i$ be the corresponding isomorphism and set $\tilde{H}_i:=(\rho_u\times \id)(\tilde{H}^\prime)$. Then there exists $i$ such that $A^\prime=\tilde{A}/\tilde{H}_i$ is the false elliptic corresponding to $T_h(P)$. By \cite[Lemma 3.1 \& 3.2]{AI}, we deduce the valuation of $\val_p(\Ha(A^\prime))=\frac{1}{(p+1)p^{2m-1}}$(resp. $\frac{1}{2p^{2m}}$) when $p$ is inert (resp. ramified). 

To see $H^\prime\cap \CC_{2m}^\prime=\CC_m^\prime$,  we translate by $s_1=\begin{pmatrix}p^{-m} & 0\\ 0 &1\end{pmatrix}$. Since \[\begin{pmatrix} 1 & u/p^m  \\ 0 & 1\end{pmatrix}\begin{pmatrix} p^{-m} & 0 \\ 0 & 1\end{pmatrix}=\begin{pmatrix} p^{-m} & 0  \\ 0 & 1\end{pmatrix}\begin{pmatrix} 1 & u \\ 0 & 1\end{pmatrix},\]
the image of $T_{s_1}(T_h(T_{s_0}P)) \in X_0(m)$ is  $T_{s_1}(T_{s_0}P)$, corresponding to $(A^\prime/H^\prime[p^m],i,\bar{\eta}, H^\prime/H^\prime[p^m])$ and  also $(A/H[p^m],i,\bar{\eta}, H/H[p^m])$. Now by the Katz-Lubin theory, we deduce that \[\quad    A^\prime/\CC_m^\prime=A/H[p^m],\quad \CC_{2m}^\prime/\CC_m^\prime=A[p^m]/H[p^m],\quad H^\prime/H^\prime[p^m]=H/H[p^m].\]  This implies that $H^\prime\cap \CC_{2m}^\prime=\CC_m^\prime$ and $A^\prime/\CC_{2m}^\prime\cong A$.
\end{proof}

\subsection{$p$-adic  Hecke characters}\label{Hecke characters}

A $p$-adic Hecke character $\mu:\ \CK^\times\bs \BA_{\CK}^{\infty,\times}\to\BC_p^\times$ is called algebraic of type $(a,b)\in\BZ^2$ if $\mu_p(z)=z^a(z^c)^b$ for all $z\in 1+p^n\CO_{\CK_p}$ and some $n\in\BN$.  Attached to such an algebraic $p$-adic Hecke character, one has the $\infty$-avatar
 \[\CK^\times\bs\BA_\CK^\times\to\BC^\times,\quad z\mapsto \mu(z^{\infty})z_p^{-a}(z_p^c)^{-b}z_\infty^a\bar{z}_\infty^b\] which is an algebraic Hecke character of infinity type $(a,b)$. 
 
 Similarly, 
a $p$-adic   Hecke character $\xi:\ \BQ_{>0}^\times\bs\BA_\BQ^{\infty,\times}\to\BC_p^\times$ is called algebraic of type $a\in\BZ$ if  $\xi_p(z)=z^a$ for all $z\in 1+p^n\BZ_p$ and some $n\in\BN$. For such $\xi$, one has the $\infty$-adic avatar
$$ \BQ^\times\bs \BA^{\times}\to\BC^\times,\quad z\mapsto \xi(z^{\infty})z_p^{-a}z_\infty^a.$$

Fix an open compact subgroup $V^p\subset \hat{\CO}^{p,\times}_\CK$. Let \[\Gamma=\Gamma_{V^p}:=\CK^\times\bs \BA_\CK^{\infty,\times}/ V^p,\quad \Gamma^-=\Gamma_{V^p}^{-}=\CK^\times\bs \BA_\CK^{\infty,\times}/\BA^{\infty,\times}V^p.\]  Let $\CW(\Gamma)$ and $\CW(\Gamma^-)$ be the weight space of $\Gamma$ and $\Gamma^-$, i.e., the rigid analytic variety parameterizing continuous characters on  $\Gamma$ and $\Gamma^-$.  Let $\CW$ (resp. $\CW^{(n)}$) be the weight space of $\BZ_p^\times$ (resp. $1+p^{n}\BZ_p$).  Note that all involved groups are abstractly isomorphic to $T\times \BZ_p^t$ for some finite order group $T$  and some $t$. The weight space of $\BZ_p^t$ is isomorphic to $D^t$ where $D$ is the open disk. See \cite{ST01} for more details.

%
Let $\CW^{(n)}(\Gamma)\subset \CW(\Gamma)${\footnote{Let $\fg$ be the Lie algebra of $\Gamma$. As in \cite{ST01}, the exponential map $\fg\to\Gamma$ induces the locally analytic map \[\CW(\Gamma)(R[\frac{1}{p}])\to\Hom_{\BQ_p}(\fg, R[\frac{1}{p}]),\quad \mu\mapsto d\mu:\ \ft\mapsto \frac{d}{dt}\chi(\exp(t\ft))|_{t=0}\] 
Bearing this map in mind and by a computation using $\exp$ and $\log$, one can describe  $\CW^{(n)}(\Gamma)$ as the subvariety parameterizing characters $\mu$ whose $d\mu$ lies in the open disk with radius roughly as $p^{-n}$. For the sharp description of the radius in the $\BZ_p^\times$, we refer to the beginning of Subsection \ref{NO}.}} be the subvariety whose $\BC_p$-points consist of characters $\mu$ such that there exists $\kappa_1$,$\kappa_2\in\BC_p$  such that \[\mu(z)=\exp(\kappa_1\log(z))\exp(-\kappa_2\log(\bar{z}))\quad\forall z\in 1+p^{n}\CO_{\CK_p}\]
Note that $\cup_n\CW^{(n)}(\Gamma)=\CW(\Gamma)$ by the fact $U_n^+\cap U_n^-=\{1\}$ and \[1+p^{n+1}\CO_{\CK_p}\subset U_n^+\times U_n^-\subset 1+p^{n}\CO_{\CK_p}\]where $U_n^-=\{x\in 1+p^{n}\CO_{\CK_p}| x\bar{x}=1\}$ and $U_n^+:=1+p^{n}\BZ_p$.


We want to define a weight map $w=(w_1, w_2):\ \CW(\Gamma)\to \CW\times\CW$ so that if  $\mu^{(\infty)}$ has infinite type $(b,a-b)$, then $w(\mu)=(a,b-a)$.

 Restriction to $ \BZ_p^\times\subset \CO_{\CK_p}^\times$ gives $w_1$, which is clearly analytic. Note that  for any element $\xi\in \CW$, the fiber $w_{1}^{-1}(\xi)$ is a $\CW(\Gamma^-)$-torsor. 

When $p$ splits in $\CK$, one can define $w_{2}$ by restricting to  $\CO_{\CK_p}^\times$ and composing with the injection \[\BZ_p^\times\hookrightarrow \CO_{\CK_p}^\times,\ z\mapsto (1,z^{-1}).\]
Clearly  $w_2$ is actually analytic. 

When $p$ is non-split in $\CK$, we define  $w_2$ on $\CW^{(n)}(\Gamma)$ by
\[\mu\mapsto w_2(\mu): z\mapsto\exp(\kappa_2\log(z))\]
It is only locally analytic since the definition involves the exponential map.
\begin{lem}\label{weight} One has the weight map  \[w=(w_{1},w_{2}):\ \CW^{(n)}(\Gamma)\to \CW\times\CW^{(n)}\]
such that  if $\mu^{(\infty)}$ has infinite type $(b,a-b)$, then $w(\mu)=(a,b-a)$.
\end{lem}


\subsection{Toric period and central $L$-value}\label{toric period}
For any element $\xi\in \CW$, set $\CW^{(n)}(\xi)=w_{1}^{-1}(\xi)\cap \CW^{(n)}(\Gamma)$. When $\xi=\omega^{-1}|\cdot|^k$, denote $\CW^{(n)}(\xi)$ by $\CW^{(n)}(k)$. Fix a connected component of $\CW(k)$ and let $ \Sigma(k)$ (resp. $\Sigma^+(k)$)  be the subset of algebraic Hecke characters in this component (resp. moreover of infinite type $(b,k-b)$ with $b\geq k$).

Let $\sigma$ be a cuspidal automorphic  $\GL_2(\BA)$-representation with  central character $\omega|\cdot|^{-k}$ and lowest weight $k\geq2$.
 For any $\mu\in\Sigma(k)$, the $L$-function $L(s,\sigma\times\mu)$ is self-dual and satisfies a functional equation 
\[L(s,\sigma\times\mu)=\epsilon(s,\sigma\times\mu)L(1-s,\sigma\times\mu),\quad \epsilon(1/2,\sigma\times\mu)=\prod_v\epsilon(1/2,\sigma_{v}\times\mu_v)\in\{\pm1\}\]
 By construction,  \begin{itemize}
     \item $\epsilon_v(1/2,\sigma_v\times\mu_{v})$ is invariant for $v\nmid p\infty$  when $\mu\in\Sigma(k)$,
     \item  $\epsilon_\infty(1/2, \sigma_{\infty}\times \mu_{\infty})=1$ if $\mu\in\Sigma^+(k)$.
 \end{itemize} 
Assume that there exists a unique indefinite quaternion algebra $B$ over $\BQ$ such that $B_p$ is split and $\epsilon(B_v)\mu_{v}(-1)\chi_{\CK}(-1)=\epsilon(1/2,\sigma_{v}\times\mu_{v})$ for  $v\nmid p\infty$. This includes the case 
\begin{itemize}
\item $\epsilon(1/2,\sigma\times\mu)=1$ for some $\mu\in\Sigma^+(k)$ and
    \item $\sigma_p$ is a principal series or $p$ splits in $\CK$. 
\end{itemize}
Let $\pi=\otimes\pi_v$ be the automorphic $B^\times$-representation whose Jacquet-Langlands transfer is $\sigma$.

The starting point of the construction of anticyclotomic $p$-adic $L$-function  is the Waldspurger formula relating the period and central $L$-value for Rankin pair $(\sigma,\mu)$. 
For any $\varphi\in\pi$ and $\mu\in\Sigma^+(k)$, set 
\[P_\mu(\varphi):=\int_{\BA^{\infty,\times}\CK^\times\bs \BA_{\CK}^{\infty,\times} }\varphi(t)\mu(t)d^\times t=\frac{1}{\sharp \BA^{\infty,\times}\CK^\times\bs \BA_{\CK}^{\infty,\times}/V}\sum_{\BA^{\infty,\times}\CK^\times\bs \BA_{\CK}^{\infty,\times}/V}\varphi(t)\mu(t)\]
The Waldspurger formula states that for $\varphi=\otimes\varphi_v\in \pi$ and $\varphi^\vee=\otimes \varphi_v^\vee\in \pi^\vee$,
\[P_\mu(\varphi)P_{\mu^{-1}}(\varphi^\vee)=\frac{\zeta(2)L(1/2,\sigma\times\mu)}{8L(1,\chi_{\CK})^2L(1,\sigma,\ad)}\prod_v\alpha_v(\varphi_v,\varphi_v^\vee;\mu_v)\]
\[\alpha_v(\varphi_v,\varphi_v^\vee;\mu_v)=\frac{L_v(1,\chi_\CK)L_v(1,\sigma,\ad)}{\zeta_v(2)L_v(1/2,\sigma_v\times\mu_v)}\int_{\CK_v^\times/\BQ_v^\times}(\pi(t)\varphi_v,\varphi_v^\vee)\mu_v(t)d^\times t.\]
Here all $L$-functions are complete,
\begin{itemize}
\item $d^\times t$ is the Haar measure on $\BA^{\infty,\times}\CK^\times\bs \BA_{\CK}^{\infty,\times}$ with total volume $1$,
\item  $d_v^\times t$ is the Haar measure on $\BQ_v^\times\backslash\CK_v^\times$ such that  $\prod_v d^{\times}_vt=2L(1,\chi_\CK)dt$ and  $\Vol(\BZ_v^\times\bs \CO_{\CK_v}^\times, d^\times t_v)=1$ for almost all $v$.
\item $(-.-)$ is the Petersson inner product on $\pi\times\pi^\vee$ given by \[(\varphi,\varphi^\vee)=\int_{B^\times\bs \BB^\times/\BA^\times}\varphi(g)\varphi^\vee(g) dg\]
where $dg$ is the Haar measure on $B^\times\bs \BB^\times/\BA^\times$ with total volume $2$,
\item $(-,-)_v$ is an invariant pairing on $\pi_v\times\pi_v^\vee$ such that $\prod_v(-,-)_v=(-,-)$,
\item $V$ is a proper level subgroup fixing $\varphi$ and $\mu$.
\end{itemize}
 Take $J\in B$ such that 
 $Jt^c=tJ$ for all $t\in \CK$ and $J^2\in \BQ^\times$. Note that $\mu(t^c)=\mu^{-1}(t)\omega^{-1}|tt^c|^k$. Thus  \begin{align*}P_\mu(\varphi)&=\int_{\CK^\times\bs \BA_{\CK}^{\infty,\times}/\BA^{\infty,\times}}\varphi(t^c)\mu(t^c)d^\times t=\int_{\CK^\times\bs \BA_{\CK}^{\infty,\times}/\BA^{\infty,\times}}(\varphi\otimes\omega^{-1}|\cdot|^k)(t^c)\mu^{-1}(t)d^\times t\\
 &=\int_{\CK^\times\bs \BA_{\CK}^{\infty,\times}/\BA^{\infty,\times}}(\varphi\otimes\omega^{-1}|\cdot|^k)(tJ)\mu^{-1}(t)d^\times t
=P_{\mu^{-1}}(\pi(J)\varphi\otimes\omega^{-1}|\cdot|^k)
\end{align*}

Fix $\varphi_1=\otimes_v \varphi_{1,v} \in\pi$ (resp. $\varphi_3=\otimes_v \varphi_{3,v} \in\sigma$) of weight $k$ and $\varphi_2=\otimes_v\varphi_{2,v}\in\pi^\vee$ (resp.  $\varphi_4=\otimes_v \varphi_{4,v} \in\sigma^\vee$) of weight $-k$ such that $(\varphi_1, \varphi_2)\neq0$ (resp. $(\varphi_3, \varphi_4)\neq0$).  Assume the $\psi_\BQ$-Whittaker functional (resp. $\bar{\psi}_\BQ$-Whittaker functional) of $\varphi_3$ (resp. $\varphi_4$) is $\otimes_v W_{3,v}$ (resp. $\otimes_v W_{4,v}$) (Here $\psi_\BQ$ is an additive character on $\BQ\bs \BA$). Consider the invariant pairing \[\langle-,-\rangle_v:\ \sigma_v\times\sigma_v^\vee\to\BC,\quad (W_v, W_v^\vee)\mapsto \frac{\int_{\BQ_v^\times} W_v W_v^\vee(y) d^\times y}{L(1,\sigma_v,\ad)L(1,1_v)L(2,1_v)^{-1}}\]
Here we fix the Haar measure $d^\times y$ as in \cite{CST}. For $\varphi_v\in\pi_v$, let \[\tilde{\beta}(\varphi_v, \mu_v;\varphi_{1,v},\varphi_{2,v};W_{3,v},W_{4,v}):=\frac{\alpha_v(\varphi_v,\pi(J)\varphi_v\otimes\omega_v^{-1}|\cdot|_v^k;\mu_v)\langle W_{3,v}, W_{4,v}\rangle_v}{(\varphi_{1,v}, \varphi_{2,v})_v}\]
The following variant of Waldspurger formula is suited for $p$-adic interpolation:
\begin{prop}\label{Wal}For $\varphi=\otimes_v\varphi_v\in\pi$
\[P_\mu^2(\varphi)=\frac{L(1/2,\sigma,\mu)}{4L(1,\chi_{\CK})^2}\frac{(\varphi_1,\varphi_2)}{(\varphi_3,\varphi_4)}\prod_v\tilde{\beta}(\varphi_v,\mu_v;\varphi_{1,v},\varphi_{2,v};W_{3,v},W_{4,v})\]
\end{prop}

One can also consider the toric period integral for Eisenstein series. In this case $B=M_2(\BQ)$ and we shall fix a $\BQ$-basis $\{1,\tau_0\}$ of $\CK$. We shall choose the embedding
\[\CK\to M_2(\BQ),\quad a+b\tau_0\mapsto \begin{pmatrix} a & -b\\b N_{\CK/\BQ}(\tau_0) &
a+b \RTr_{\CK/\BQ}(\tau_0) \end{pmatrix}\] Moreover, assume  $(-1)^{k}=\omega_\infty(-1)$ and  $\omega\neq1$ when $k=2$. Set $\omega_{1-k}=\omega|\cdot|^{1-k}$. View $\omega_{1-k}$ as a character on the  upper Borel $P\subset G$ by the rule 
$\begin{pmatrix}a & b \\ 0 & d\end{pmatrix}\mapsto \omega_{1-k}(a)$. Under these assumptions, the normalized parabolic induction $I_{P(\BR)}^{G(\BR)}(\omega_{1-k})_\infty$ has an irreducible quotient whose minimal weight is $k$. For almost all $v$, take $\varphi_v^\circ\in I_{P(\BQ_v)}^{G(\BQ_v)}(\omega_{1-k})_v$ to be the spherical vector with $\varphi_v^\circ(1)=1$. Define the restricted tensor product $I_{P(\BA)}^{G(\BA)}\omega_{1-k}=\otimes^\prime_v I_{P(\BQ_v)}^{G(\BQ_v)}(\omega_{1-k})_v$ with respect to $\varphi_v^\circ$. For any $\varphi \in I_{P(\BA)}^{G(\BA)}\omega_{1-k}$, let $\varphi_s\in I_{P(\BA)}^{G(\BA)}\delta^{s-1/2}\omega_{1-k}$ be the flat section $\varphi_s(bk)=\varphi(bk)\delta^{s-1/2}(b)$ where $b\in P(\BA)$, $\delta$ is the modulus character
 and $k$ lies in the standard maximal open compact subgroup. 
When $\Re(s)>\frac{k+1}{2}$, the infinite sum $|\det(g)|_\BA^{-1/2}\sum_{\gamma\in P(\BQ)\bs G(\BQ)}\varphi_s(\gamma g)$ converges absolutely and defines the normalized mirabolic Eisenstein series $E(g,\varphi_s)$. It extends  to a meromorphic function on the whole $s$-plane which has no pole at $s=1/2$ under our assumption.  Set $E_\varphi:=E(g,\varphi_s)|_{s=1/2}$. Consider the toric period integral  for $\mu\in\Sigma(k)$:
\[\CP_{\mu}(\varphi)=\int_{\BA^{\infty,\times}\CK^\times\bs \BA_{\CK}^{\infty,\times}}E_\varphi(t)\mu(t)d^\times t\]
For each $v$, consider the local linear form 
\[\lambda_{\mu_v,s}:  I_{P(\BQ_v)}^{G(\BQ_v)}(\omega_{1-k})_v\to \BC,\quad \varphi_v\mapsto \frac{L_v(2s,(\omega_{1-k})_v^{-1})}{L_v(s-1/2,\mu_v)}\int_{\BQ_v^\times\bs \CK_v^\times}\varphi_{v,s}(t)\mu_v(t) |t|_{\CK_v}^{-1/2}d^\times t_v \] 
Note that for almost all $v$, $\lambda_{\mu_v,s}(\varphi_v^\circ)=1${\footnote{ One can see this from
\begin{itemize}
    \item   $L_v(2s,(\omega_{1-k})_v^{-1})=L_v(s-1/2,\mu_v)$ and $\BQ_v^\times \bs \CK_v^\times=\BZ_v^\times\bs \CO_{\CK_v}^\times$ when $v$ is inert.
    \item an explicit calculation using 
the Jacquet section when $v$ is split
\[F_\varphi:\ g\in G(\BQ_v)\mapsto |\det(g_v)|^s \omega_{1-k}(\det(g)) \int_{\BQ_v^\times}\varphi([0,y]g)\omega_{1-k}(y) |y|^{2s} d^\times y\quad \varphi\in \CS(\BQ_v^2)\]
\end{itemize}}}.
\begin{prop}\label{Wal-Eis}Set $\lambda_{\mu_v}=\lambda_{\mu_v,s}|_{s=1/2}$. Then for any  pure tensor $\varphi=\prod_v \varphi_v$,  
\[\CP_{\mu}(\varphi)=\frac{L(0,\mu)}{2L(1,\chi_\CK)L(1,\omega_{1-k}^{-1})}\prod_v\lambda_{\mu_v}(\varphi_v)\] 
\end{prop}
\begin{proof}The embedding $\CK^\times\hookrightarrow G(\BQ)$ induces a bijection  $\BQ^\times\bs \CK^\times\cong P(\BQ)\bs G(\BQ)$. Thus 
\begin{align*}
\CP_{\mu}(E(-,\varphi_s)):&=\int_{\BA^\times\CK^\times\bs\BA_\CK^\times}E(t,\varphi_s)\mu(t)d^\times t=\int_{\BA^\times\bs \BA_\CK^\times}\varphi_s(t)\mu(t)|t|_\CK^{-1/2}d^\times t\\
&=\frac{L(s-1/2, \mu)}{2L(1,\chi_\CK)L(2s,\omega_{1-k}^{-1})}\prod_v\frac{L_v(2s,(\omega_{1-k})_v^{-1})}{L_v(s-1/2,\mu_v)}\int_{\BQ_v^\times \bs \CK_v^\times}\varphi_{v,s}(t)\mu_v(t)|t|_{\CK_v}^{-1/2}d^\times t_v
\end{align*}
Taking evaluation at $s=1/2$, the desired result follows.
\end{proof}

The period integral $P_\mu(\varphi)$ admits an algebraic interpretation.  Let $\tau$ be the standard coordinate on $\BC-\BR$. Let $dz$ be the standard differential on $\BC$ and $d\tau$ be the standard differential on $\BC-\BR$.  Consider the local system \[H:\ B^\times \bs \BC^2\times (\BC-\BR)\times \BB^{\infty,\times}/U\to  B^\times \bs  (\BC-\BR)\times \BB^{\infty,\times}/U=Y_U(\BC)\]
where $\gamma[(a_1,a_2)^t,\tau,g]=[\gamma (a_1,a_2)^t, \gamma\cdot \tau, \gamma^\infty g]$. 
Then by  \cite[Lemmas 2.4.11 \& 2.4.12]{LZZ18}, $\BH=H\otimes \CO_{X_U}$ and $\uomega$ is generated by sections whose value at $\tau$ is $(\tau,1)^t$.  The Kodaira-Spencer isomorphism induces $\KS((2\pi id z)^{\otimes 2})=2\pi i d\tau$. 
For any automorphic form $\varphi$ such that \[\varphi(gk_\theta z)=z^{-m}e^{im\theta}\varphi(g) \quad\ \forall\  k_\theta=\begin{pmatrix}\cos\theta & \sin\theta\\ -\sin\theta & \cos\theta\end{pmatrix}\in\SO_2(\BR),\ z\in\BR_{>0}^\times,\]one can  attach an analytic  section $f_\varphi\in H^0(X_U(\BC),\uomega^m)$ by the rule: 
\[f_\varphi(g\cdot i)(2\pi i dz)^{-m}:=\varphi((g,1))(ci+d)^m,\quad g=\begin{pmatrix}a & b\\ c & d\end{pmatrix}\in \GL_2(\BR)=B_\BR^\times.\]
Let $\delta:=S_H\circ \KS^{-1}\circ\nabla$ be the composition of \begin{itemize}
\item the Hodge splitting $S_H:\ \BH^{ra}\to \uomega^{ra}$ of real analytic sheaves,
\item the Kodaira-Spencer isomorphism $\KS: \uomega^2=\Omega_X$, 
\item and the Gauss-Manin connection $\nabla$ on $\uomega^{k}\subset \Sym^k(\BH)$.
\end{itemize}
As explained in \cite[Section 1.2]{BDP} and \cite[Section 2.4]{LZZ18}, 
 \begin{itemize}
\item  the assignment $\varphi\mapsto f_\varphi$ is $\BB_f^{\times}$-equivariant,
\item $f_{\Delta \varphi\otimes|\cdot|^{-1}}=\delta f_\varphi$ where $\Delta$ is the weight raising operator $-\frac{1}{8\pi i}\begin{pmatrix} 1 & i\\ i & -1\end{pmatrix}$.
 \end{itemize}

In the following, we assume that  either $p$ splits in $\CK$ or $\sigma_p$ is principal series and $p$ is non-split. Take $\varphi=\otimes_v\varphi_v\in \pi$ (resp.  $I_{P(\BA)}^{\GL_2(\BA)} \omega_{1-k}$ in the Eisenstein case)  such that
\begin{itemize}
    \item $\varphi_\infty$ is holomorphic of weight $k$,
  \item $\varphi_p=\Vol(1+p^n\BZ_p)^{-1}\Char_{1+p^n\BZ_p}$ (in the Kirillov model) for $n\gg 0$.
\end{itemize}
When $p$ is non-split in $\CK$, take $g_{\ell,n}$ as in Subsection \ref{NC} with $\ell\gg2n$. Then for $\mu\in\Sigma^+(k)$ (set $\nu=w_2(\mu)$)
\begin{itemize} 
\item $\Delta^{\nu}\varphi_\infty$ is a test vector for $(\pi_\infty,\mu_\infty)$ (resp. $(I_{P(\BR)}^{\GL_2(\BR)}(\omega_{1-k})_\infty,\mu_\infty)$);
\item  $\pi_p(g_{\ell,n})\varphi_p$ (resp. $\varphi_p$) is a test vector for $(\pi_p,\mu_p)$ (and $(I_{P(\BQ_p)}^{\GL_2(\BQ_p)}(\omega_{1-k})_p,\mu_p)$) for $p$ non-split (resp. split) by Appendix \ref{Local test vector}.
\end{itemize}
Let $f_\varphi\in H^0(X_U, \uomega^k)$ be the global section attached to $\varphi$ (resp. $E_\varphi$).

 For each $t\in \BA_{\CK}^{\infty,\times}$, let $P_t=[tg_{\ell,n}]$ (resp. $P_t=[t]$) be the corresponding CM point in $\varprojlim_{U}X_U$ when $p$ is non-split (resp. split). Let $A_{t}$ be the false elliptic curve underlying $P_t$. All these $A_t$ are defined over number fields by CM theory.
For $A=A_1$, take a generator  $\omega$ of $\uomega_{A}$ over the defining field and let $\Omega_\infty\in\BC^\times$ such that $\omega=\Omega_\infty (2\pi i dz)$. Note that by CM theory, for each $A_t$, $\omega_t:=\Omega_\infty (2\pi i dz)$ is a generator of $\uomega_{A_t}$ over the defining field.

For $\mu\in \Sigma^+(k)$, set $\nu=w_2(\mu)$ and 
\[P_\mu^{\alg}(\varphi):= \frac{1}{\sharp \BA^{\infty,\times}\CK^\times\bs \BA_{\CK}^{\infty,\times}/V}\sum_{\BA^{\infty,\times}\CK^\times\bs \BA_{\CK}^{\infty,\times}/V}\delta^\nu f_\varphi(P_t) \omega_t^{-k-2\nu}\mu_\nu(t)\]
Since the Hodge splitting on CM false elliptic curves coincide with the CM splitting, $P_\mu^{\alg}(\varphi)\in\bar{\BQ}$. Recall that $s=2$ (resp. $s=1$) when $p$ is ramified (resp. inert)
\begin{lem}\label{algebracity} For $\mu\in \Sigma^+(k)$,  \[P_\mu^{\alg}(\varphi):=\begin{cases}\frac{p^{(\ell-2n+s-1)\nu}P_\mu(\Delta^\nu\varphi(-g_{\ell,n}))}{\Omega_\infty^{2\nu+k}} & p\ \text{non-split}\\
\frac{P_\mu(\Delta^{\nu}\varphi)}{\Omega_\infty^{2\nu+k}} & p\ \text{split}\end{cases}\]
\end{lem}
\begin{proof}By construction,\begin{align*}
    P_\mu^{\alg}(\varphi)&= \frac{1}{\omega_\infty^{k+2\nu}\sharp \BA^{\infty,\times}\CK^\times\bs \BA_{\CK}^{\infty,\times}/V}\sum_{\BA^{\infty,\times}\CK^\times\bs \BA_{\CK}^{\infty,\times}/V}\delta^\nu f_\varphi(P_t) (2\pi i dz)^{-k-2\nu}\mu_\nu(t)\\
    &=\begin{cases} \frac{P_{\mu_\nu}(\Delta^{\nu}\varphi\otimes|\cdot|^{-\nu}(-g_{\ell,n}))}{\omega_\infty^{k+2\nu}} &\ p\ \text{non-split}\\
\frac{P_{\mu_\nu}(\Delta^{\nu}\varphi\otimes|\cdot|^{-\nu})}{\omega_\infty^{k+2\nu}} &\ p\ \text{split}\end{cases}
\end{align*}
Since $P_{\mu_\nu}(\varphi\otimes|\cdot|^{-\nu})=P_{\mu}(\varphi)$ for any $\varphi\in \pi$, we deduce the stated result.
\end{proof}

\subsection{Anticyclotomic $p$-adic $L$-functions}
Take a large enough coefficient field $L/\BQ_p$ with ring of integers $\CO_L$. 
Take the smooth model over $\CO_L$ of $A_t$, which we still denote by $A_t$. Assume that $\omega$ generates $\uomega_A$ over $\CO_L$. Let $\omega_{p}$ be a generator of the modified sheaf $\Omega_{A}$ which reduces to the marked section module $p\delta_{A}$ and take $\Omega_p\in\BC_p^\times$ such that $\Omega_p\omega_p=\omega$. By CM theory, $\omega_{t,p}:=\Omega_p^{-1}\omega_t$ is a generator of $\Omega_{A_t}$ for any $t\in \BA_{\CK}^{\infty,\times}$. When $p$ splits in $\CK$, all $P_t$ lie in the ordinary locus (i.e.,  $A_t$ is ordinary) and $\Omega_p$ is a $p$-adic unit. When $p$ is non-split in $\CK$, none of $P_t$ lie in the ordinary locus and by Lemma \ref{CMIII}, $v_p(\Omega_p)=\begin{cases} -\frac{1}{p^{\ell-1}(p^2-1)} &\ p\ \text{inert}\\
-\frac{1}{2p^{\ell}(p-1)} &\ p\  \text{ramified}\end{cases}$.

We now define the one-variable and two-variable
anticyclotomic $p$-adic $L$-function.  Take $\varphi=\otimes_v\varphi_v\in\pi$ (resp.  $I_{P(\BA)}^{\GL_2(\BA)} \omega_{1-k}$ in the Eisenstein case)  such that
\begin{itemize}
    \item $\varphi_\infty$ is holomorphic of weight $k$,
    \item $\varphi_p=\Vol(1+p^n\BZ_p)^{-1}\Char_{1+p^n\BZ_p}$ (in the Kirillov model) for $n\gg0$. 
\end{itemize}

The case $p$ is split in $\CK$ is well-understood. Let $\delta_{AS}=S_{ur}\circ \nabla$ be the Atkin-Serre operator on the ordinary locus. Here $S_{ur}$ is the unit root splitting on the ordinary locus (see \cite[Lemma 2.3.1]{LZZ18}). Note that  $S_{ur}$ induces the CM splitting on ordinary (false) elliptic curves.

In the case  $p$ is non-split in $\CK$, we slightly generalize \cite{AI}. 
Let $A_t^\prime=A_t/\CC_{\ell-1}$ and $\lambda_t:\ A_t^\prime\to A_t$ be the dual isogeny. Let $\omega_t^\prime=\lambda_t^*(\omega_{t,p})$, which is a generator of $\Omega_{A^\prime}$.  
Note that by Lemma \ref{CMIII}, $A_t^\prime$ has CM by $\CO_{pd}:=\BZ+pd\CO_{\CK}$ for some positive integer $d$ prime to $p$. View $\CK$ as a subfield of $L$ and set \[\bar{\Omega}_{A_t^\prime}=\{v\in \BH^\sharp_{A_t^\prime}| x\cdot v=\bar{x}v,\ \forall x\in \CO_{pd}\}\] As we can choose $\ell\gg0$, the argument in  \cite[Lemma 5.1]{AI} shows that $\lambda_t^*(\BH_{A_t}^\sharp)\subset\tilde{\BH}_{A_t^\prime}:=\Omega_{A_t^\prime}\oplus \bar{\Omega}_{A_t^\prime}$. For any $p$-adic weight $\kappa$, $\omega^\prime$ induces an isomorphism $v_{\omega_t^\prime}:\ \fm^{\kappa}_{A_t^\prime}\cong\CO_L$. Moreover by Lemma \ref{Canonical Splitting}, the  CM splitting $\tilde{\BH}_{A_t^\prime}\to\Omega_{A_t^\prime}$ induces a splitting $\Psi_t:\ \BW^\kappa(\tilde{\BH}_{A_t^\prime}^\sharp, s_{\HT})\to \fm^{\kappa}_{A_t^\prime}$.

For each $t$, view the CM point $P_t$ as a map from $L$ to $\CX$ and denote the evaluation $v_{\omega_t^\prime}\circ\Psi_t\circ \lambda_t^*\circ P_t^*(\nabla^{\nu}f_\varphi)$ by $\delta_{AS}^\nu f_\varphi(P_t)$. For any $\nu\in\CW^{(n)}$, take any lift $\tilde{\nu}\in \CW$ and consider the value  $\delta_{AS}^{\tilde{\nu}} f_\varphi(P_t)\mu_{\tilde{\nu}}(t)$. Note that for  $\xi$ finite, $\delta_{AS}^\xi f_\varphi$ is given by the twist-by-$\xi$ operation. Thus as  \cite[Lemma 2.1.16]{LZZ18},   our specific choice of $\varphi_p$ implies that the value actually is independent of the lift and we shall denote it again by $\delta_{AS}^{\nu} f_\varphi(P_t)\mu_{\nu}(t)$.

\begin{defn}\label{one-variable}Let $\nu=w_2(\mu)$. Let $\CP_\varphi$ be the function on $\CW^{(n)}(k)$ given by 
\[\mu\mapsto \frac{1}{\sharp \BA^{\infty,\times}\CK^\times\bs \BA_{\CK}^{\infty,\times}/V}\sum_{\BA^{\infty,\times}\CK^\times\bs \BA_{\CK}^{\infty,\times}/V}\delta_{AS}^{\nu}f_\varphi(P_t)\mu_\nu(t)\]
\end{defn}
We now define the two-variable anticyclotomic $p$-adic
   $L$-function. We shall assume $\pi_p$ is non-supercuspidal in the cuspidal case. Take a small enough neighborhood $k\in\CU\subset \CW$. Then one can deform 
$f_\varphi$ into a $p$-adic family $F$ on $\CU$ by the following steps:
\begin{itemize}
\item deform the $p$-stabilized form in $\pi$ into a Coleman family,
\item apply the $p$-depletion procedure and take linear combination of twists by finite order characters.
\end{itemize}
Note that in the Eisenstein case, we can actually take the ordinary stabilization and consider the Hida family. Set $\CW^{(n)}(\Gamma)_\CU=\CW^{(n)}(\Gamma)\times_{\CW}\CU$ for the weight map  $w_1:\ \CW(\Gamma)\to \CW$. Upon shrinking $\CU$, we may assume  the automorphic representation generated by $F_{\kappa}$ for all classical weight $\kappa\in \CU$ has $p$-conductor $<n$.
\begin{defn}\label{two-variable}For $\mu\in \CW^{(n)}(\Gamma)_U$, set $w(\mu)=(\kappa,\nu)$. Let $\CP_F$ be the function on $\CW^{(n)}(\Gamma)_U$ given by 
\[\mu\mapsto  \frac{1}{\sharp \BA^{\infty,\times}\CK^\times\bs \BA_{\CK}^{\infty,\times}/V}\sum_{\BA^{\infty,\times}\CK^\times\bs \BA_{\CK}^{\infty,\times}/V}\delta_{\kappa,AS}^{\nu}F_\kappa(P_t)\mu_\nu(t)\]
\end{defn}
Note that in the cuspidal case, $C(k,\nu):=\tilde{\beta}(\nabla^\nu\varphi_\infty,\nabla^\nu\varphi_\infty,\mu_\infty,\varphi_{1,\infty},\varphi_{2,\infty};W_{3,\infty},W_{4,\infty})$  depends only on $(k,\nu)$. In the Eisenstein case, $C(k,\nu):=\lambda_{\mu_v}(\Delta^\nu\varphi_v)$ depends only on $(k,\nu)$. Adapting \cite[Proposition 5.6  \& 6.4]{AI}, we deduce from Theorem \ref{Wal} and Proposition \ref{Wal-Eis} that
\begin{prop}\label{Int} The function $\CP_{\varphi}$ is locally analytic on $\CW(k)$ and for $\mu\in \Sigma^+(k)$, \[\CP_{\varphi}(\mu)=\Omega_p^{2\nu+k}P_\mu^{\alg}(\varphi).\]
The function $\CP_{F}$ is locally analytic on $\CW^{(n)}(\Gamma)_\CU$ and for classical weight $\kappa\in \CU$,  $\CP_F(\kappa)=\CP_{F_\kappa}$. 

Moreover, in the cuspidal case (for $p$ non-split and split respectively) \[\frac{\CP^2_\varphi(\mu)}{\Omega_p^{4\nu+2k}}=\begin{cases}\frac{p^{(\ell-2n+s-1)2\nu}L(1/2,\sigma,\mu)}{4L(1,\chi_{\CK})^2\Omega_\infty^{4\nu+2k}}\frac{(\varphi_1,\varphi_2)}{(\varphi_3,\varphi_4)}C(k,\nu)\prod_{v<\infty}\tilde{\beta}(\pi_v(g_{\ell,n})\varphi_v,\mu_v;\varphi_{1,v},\varphi_{2,v};W_{3,v},W_{4,v})\\
\frac{L(1/2,\sigma,\mu)}{4L(1,\chi_{\CK})^2\Omega_\infty^{4\nu+2k}}\frac{(\varphi_1,\varphi_2)}{(\varphi_3,\varphi_4)}C(k,\nu)\prod_{v<\infty}\tilde{\beta}(\varphi_v,\mu_v;\varphi_{1,v},\varphi_{2,v};W_{3,v},W_{4,v}) \end{cases}\]
In the Eisenstein case, \[\frac{\CP_\varphi(\mu)}{\Omega_p^{2\nu+k}}=\begin{cases}\frac{p^{(\ell-2n+s-1)\nu}L(0,\mu)}{2L(1,\chi_{\CK})L(1,\omega_{1-k}^{-1})\Omega_\infty^{2\nu+k}}C(k,\nu)\prod_{v<\infty}\lambda_{\mu_v}(\varphi_v(-g_{\ell,n})) & p\ \text{non-split}\\
\frac{L(0,\mu)}{2L(1,\chi_{\CK})L(1,\omega_{1-k}^{-1})\Omega_\infty^{2\nu+k}}C(k,\nu)\prod_{v<\infty}\lambda_{\mu_v}(\varphi_v) & p\ \text{split}\end{cases}\]
Here for $v\neq p$, we understand $g_{\ell,n}=1$.
\end{prop}
Note that on the chosen connected component of $\CW(k)$,
\begin{itemize}
\item  for $v\nmid p\infty$ non-split in $\CK$, $\mu_v\mapsto \tilde{\beta}(\varphi_v,\mu_v;\varphi_{1,v},\varphi_{2,v};W_{3,v},W_{4,v})$ (resp. $\mu_v\mapsto \lambda_{\mu_v}(\varphi_v)$) is  constant by construction;
\item for $v\nmid p\infty$ split in $\CK$, the function $\mu_v\mapsto \tilde{\beta}(\varphi_v, \mu_v;\varphi_{1,v},\varphi_{2,v};W_{3,v},W_{4,v})$ is regular by \cite[Lemma 4.1.8]{LZZ18}. Actually, it is essentially the local epsilon-factor (thus a unit in the Iwasawa algebra)  for  specific $\varphi_v$ (see Appendix \ref{Local test vector II}.) Similar results holds for $\mu_v\mapsto \lambda_{\mu_v}(\varphi_v)$.
\end{itemize}
\begin{remark}When $p$ splits in $\CK$,   one can implement all constructions in the world of $p$-adic modular forms and deduce that $\CP_\varphi$ lies in the Iwasawa algebra from the effect of $\delta_{AS}$ on $q$-expansion/Serre-Tate expansion. For more details, see \cite{How20}. Following \cite{LZZ18} (particularly \cite[Lemma 4.1.9]{LZZ18}),
we define the primitive anticyclotomic $p$-adic $L$-function as \[\CL_{p,\pi}(\mu):=\frac{\CP^2_\varphi(\mu)}{\prod_{v\nmid \infty}\tilde{\beta}(\varphi_v,\mu_v;\varphi_{1,v},\varphi_{2,v};W_{3,v},W_{4,v})}.\]
Note that this $\CL_{p,\pi}(\mu)$ does not depend on $n$ and glues to a bounded function on $\CW(k)$. 
\end{remark}
\begin{remark}
When $p$ is non-split in $\CK$,
\begin{itemize}
    \item the complex period $\Omega_\infty$ and the $p$-adic period $\Omega_p$  depend on $\ell$. Varying $\ell$, one deduces from the continuity of $\CP_\varphi(\mu)$ and Theorem \ref{Test vector} that $\val(\Omega_{p,\ell}/\Omega_{\infty,\ell})-\ell/2$ is constant. We will discuss this interesting fact in details elsewhere. 
    \item if $\sigma_p=\St_p\otimes\xi_p$ is Steinberg and $\xi_p\circ N_{\CK_p/\BQ_p}\mu_p=1$, $\CP_\varphi$ is identically zero in a small neighborhood of $\mu\in \Sigma^+(k)$. We shall discuss  this new exceptional zero phenomenon elsewhere.
\end{itemize}
Note that the local period integral $\mu_p\mapsto \tilde{\beta}(\pi_p(g_{\ell,n})\varphi_p,\mu_p;\varphi_{1,p},\varphi_{2,p};W_{3,p},W_{4,p})$ is not bounded on $\CW(k)$. One may define the primitive anticyclotomic $p$-adic $L$-function as \[\CL_{p,\pi}(\mu):=\frac{\CP^2_\varphi(\mu)}{\prod_{v\nmid p\infty}\tilde{\beta}(\varphi_v,\mu_v;\varphi_{1,v},\varphi_{2,v};W_{3,v},W_{4,v})},\]
which is only locally analytic on $\CW(k)$.

\end{remark}
\begin{remark}We sketch how to define the  primitive two-variable anticyclotomic $p$-adic $L$-function and will report the construction in detail elsewhere. Let $S$ be a finite set of places away from $p\infty$ containing all ramified places of $\pi$. Deform the $p$-stabilized form in $\pi$ (resp. $\pi^\vee$) into a Coleman family $\varPhi_1$ (resp. $\varPhi_2$). Consider  the $B_S^\times$-smooth representations $\Pi_S$ (resp. $\Pi_S^\vee$) generated by $\varPhi_1$ (resp. $\varPhi_2$). By the local-global compatibility in families (see \cite[Section 3 \& 4]{Dis22}), one has $\Pi_S=\otimes_{v\in S}\Pi_v$ (resp. $\Pi_S^\vee=\otimes_{v\in S}\Pi_v^\vee$) for  smooth $B_v^\times$-representations $\Pi_v$ (resp. $\Pi_v^\vee$) of pure local origin. For each $v$, fix a pairing on $\Pi_v\times\Pi_v^\vee$. Note that $F$  is the $p$-depletion (and linear combination of twist-by-finite order characters) of some $\otimes_v\varPhi_v\in \Pi_S$. Assume $\varPhi_{i}=\otimes_v\varPhi_{i,v}$. For each $v\in S$, the main result in \cite{CF24} shows  how to define
\[\beta(\varPhi_v,\mu_v;\varPhi_{1,v},\varPhi_{2,v}):=\frac{\alpha_v(\varPhi_v,\pi(J)\varPhi_v\otimes\omega_v^{-1}|\cdot|_v^k;\mu_v)}{(\varPhi_{1,v},\varPhi_{2,v})}\]  as a meromorphic function on $\CW^{(n)}(\Gamma)_\CU$ for some $n$.  One can define \[\CL_{p,\Pi}:=\frac{\CP_F^2}{\prod_{v\in S}\beta(\varPhi_v, \mu_v;\varPhi_{1,v},\varPhi_{2,v})}\]
 \end{remark}

\section{$p$-adic Waldspurger formulae}\label{Explicit reciprocity law}
Let $\sigma$ be a cuspidal automorphic $\GL_2$-representation with central character $\omega|\cdot|^{-k}$ and lowest weight $2\BN\ni k=r+2\geq2$.  Let $\chi\in\CW(k)$ be of infinity type $(1+j,k-1-j)$, $0\leq j \leq r$ such that $\epsilon(1/2,\sigma\times\chi)=-1$. Let $B$ be the indefinite quaternion algebra  
such that  for each finite place $v$, $\epsilon(B_v)\chi_v(-1)\chi_{\CK}(-1)=\epsilon(1/2,\sigma_{v}\times\chi_{v})$.  Let $\pi=\otimes\pi_v$ be the automorphic $B^\times$-representation whose Jacquet-Langlands transfer is $\sigma$. In this section, assume that  $\sigma_p$ is a {\bf principal series}. This implies that $B_p$ is split and $\sigma_p=\pi_p$.    Fix the embedding $\CK\to B$  and take $g_{\ell,n}\in\GL_2(\BQ_p)$ as in Subsection \ref{NC}.

Previously, we constructed the anticyclotomic $p$-adic $L$-function $\CP_\varphi$ and established its interpolation formula at algebraic character $\mu\in \Sigma^+(k)$. In this section, we shall relate
 $\CP_\varphi(\chi)$  with  the generalized Heegner cycle attached to $(\varphi,\chi)${\footnote{It is highly desirable to establish a family version which states that the big logarithm of the big Heegner class (see Appendix \ref{Heegner cycles in family}) equals the anticyclotomic $p$-adic $L$-function. We shall report this result when $p$ splits in $\CK$ elsewhere. In the case $p$ is non-split, the big $p$-adic logarithm is not established yet.}}. Such formulae are widely studied under different ramification restrictions  (see e.g., \cite{BDP, Bro14,LZZ18, Cas18,HW24, Man24, AI}). Since the semistable models of the relevant Kuga-Sato varieties are not known yet (perhaps even not expected), we shall follow the strategy in \cite{HW24} which works directly with local systems on relevant Shimura curves. Since these curves only have semistable model in general, the  recently developed syntomic formalism with coefficients in \cite{ABSV} is essential\footnote{When $k=2$, one can also use the Vologodsky integration theory in \cite{BZ23}}. 
\subsection{Geometry of Shimura curves and Coleman primitives}
We start with  some facts concerning the rigid geometry of $X_1(m)$. When $B=M_2(\BQ)$, details can be found in \cite[Section 4.4]{BE10}. For general cases, details/ingredients can be found in \cite{Buz97,Car86} and \cite{KM85}. 
 
 When $m\geq1$, the curve $X_1(m)$ has a regular flat model $\CX_1(m)$ over $\BZ_p[\zeta_{p^m}]$. We can and will assume the tame level is small enough, so the irreducible components of the special fiber of $\CX_1(m)$ are all smooth curves of genus at least two. These irreducible components meet at the supersingular points and exactly two among them are isomorphic to the Igusa curve $\IG(m)$.  We denote these two components by $\IG_\infty(m)$ and $\IG_0(m)$. Enlarging $L$ if necessary, $X_1(m)_L$ acquires a semistable model $\pi:\ \CX_{1,m}\to \CX_{1}(m)\times\CO_L$. This semistable model $\CX_{1,m}$ can be obtained by 
 applying sufficiently many blowing-ups and the bi-rational map $\pi$ identifies two components of the closed special fiber $\bar{\CX}_{1,m}$ of $\CX_{1,m}$ with $\IG_\infty(m)\times k_L$ and $\IG_0(m)\times k_L$. 
 
 Let $\red$ be  the mod $p$ reduction map $X_1(m)_L\to \bar{\CX}_{1,m}$. The inverse image of the irreducible components of $\bar{\CX}_{1,m}$ are wide opens and form an admissible cover of  the rigid analytic variety $X_1(m)_L$. Denote by $\CW_\infty$   the admissible open corresponding to $\IG_\infty(m)\times k_L$ and  the inverse image of the smooth locus of $\IG_\infty(m)\times k_L$  by  $\CA_\infty$. $\CA_\infty$ is the underlying affinoid of $\CW_\infty$ and its $L$-points are  in bijection with quadruples $(A,i, \bar{\eta}, \eta_p)$
  where $(A,i)$ is an ordinary false elliptic curve, $\bar{\eta}$ is a $U^p$-orbit of away-from-$p$ level structure and an isomorphism $\eta_p:\ \mu_{p^m}\to e\CC(m)$ between $\mu_{p^m}$ and the canonical subgroup of level $m$. Quotient by the canonical subgroup (of level $1$) gives a characteristic zero lift $\Phi_{\Fr}$ of the Frobenius map in characteristic $p$ to a system of wide open neighborhood of $\CA_\infty$ in $\CW_\infty$. In terms of the Hecke translation, $\Phi_{\Fr}=p\langle p\rangle V_p$
  where $\langle p\rangle$ is the diamond operator.
 

Set  $\BH_r=\Sym^r \BH$ and consider the de Rham cohomology of the complex 
\[0\to \BH_r\xrightarrow{\nabla} \nabla(\BH_r)+\BH_r\otimes \Omega_{X_1(m)}^1\to0\]
In the case $B=M_2(\BQ)$, we shall consider its  parabolic de Rham cohomology.  

For any small enough open compact subgroup $U\subset \BB_f^\times$, and any form $\varphi=\otimes_v\varphi_v\in\pi^U$ such that $\varphi_\infty$ is holomorphic,  denote the differential class in $H^0(X_U, \uomega^r\otimes\Omega^1_{X_U})\subset H^1_{\dR}(X_U,\BH_r)$ attached to $\varphi\otimes|\cdot|$ by $\omega_\varphi$.
When $\varphi_p$ is the new vector,  we obtain $\omega_\varphi\in H^0(X_1(n), \uomega^r\otimes\Omega^1_{X_1(n)})$ for $n\geq \Cond(\pi_p)$. By the relation between $V_p$ and $\Phi_{\Fr}$, the action of $p$-depletion operator $\id-V_p U_p$ on  $\omega_\varphi|_{\CW_\infty}$ is given by $P(\Phi)$, where  $P$ is a polynomial  of degree two when $\pi_p$ is unramified (\cite[Proposition 5.8]{Bro14}) and  degree one when $\pi_p$ is ramified (\cite[Theorem 2.4]{Cas18}). According to the Weil conjecture, $P(T)$ is a Coleman polynomial in the sense of \cite[Definition 5.6]{Bro14}. Now by the theory of Coleman primitives \cite[Theorem 10.1]{Col94},
\begin{lem}\label{Col}When $\varphi_p$ is the new vector,  there exists a locally analytic section $F_{\varphi}$ of $\BH_r$ on $\CW_\infty$ such that $\nabla F_\varphi=[\omega_\varphi]$ and $P(\Phi_{\Fr})F_\varphi$ is rigid analytic on some strict neighborhood $\CA_\infty\subset \CW_\infty^\prime\subset \CW_\infty$. Such section is unique when $r>0$ and unique up to constant when $r=0$. 
\end{lem} 
Now let  $\varphi_p^\prime$ be the $p$-depleted new-vector. Set $\varphi^\prime=\varphi_p^\prime\otimes \otimes_{v\neq p}\varphi_v$ and denote $P(\Phi_{\Fr})F_\varphi$ by $F_{\varphi^\prime}$. When $r=0$, such $F_\varphi^\prime$ is unique only up to constant and we make the following specific choice. Let  $\iota_X:\ X_1(n)\to J_1(n)$ be the quasi-embedding from the Shimura curve $X_1(n)$ to its Jacobian  $J_1(n)$ induced by the Hodge class (\cite[Section 3]{YZZ}). The invariant differential form $\omega_{\varphi^\prime}$ can be seen as an invariant differential form on $J_1(n)$ via $\iota_X$. Let $\log_p$ be the $p$-adic logarithm from $J_1(n)$ to its Lie algebra and set
\[\log_{\varphi^\prime}:\ X_1(n)\to \BC_p,\quad x\mapsto \langle \log_p\iota_X(x),\omega_{\varphi^\prime}\rangle.\]
By \cite{BZ23}, $\log_{\varphi^\prime}$ is the Coleman primitive of $[\omega_{\varphi^\prime}]$.

In the following, take $\varphi=\otimes_v\varphi_v\in \pi$  and $\ell\gg 2n\geq \max\{4,2\Cond \pi_p\}$ such that
\begin{itemize}
    \item $\varphi_\infty$ is holomorphic of weight $k$,
    \item   $\varphi_p=\Vol(1+p^n\BZ_p)^{-1}\Char_{1+p^n\BZ_p}$.
\end{itemize}
Since $\varphi_p$ is the linear combination of $\varphi_p^\prime$ twisted by characters of order $\leq n$, we can obtain a rigid analytic section $F_\varphi$ on $\CX_0(2n,n;\tilde{r})_{c}^{\geq n}$ for $\tilde{r}\gg0$ from $F_{\varphi^\prime}$ by pull-back and twisting by finite order characters.

\subsection{Generalized Heegner class}
For any $t\in \BA_{\CK}^{\infty,\times}$, let $A_t$ be the (false) elliptic curve underlying $P_t:=[tg_{\ell,n}]$ when $p$ is non-split in $\CK$ and $P_t:=[t]$ when $p$ is split in $\CK$. Let  $\Gamma_t$ be the graph of the isogeny $f_t:\ A:=A_1\to A_t$ given by the CM theory. Consider the eigen-decomposition of $\BH_{A_t}$ with respect to the $\CK$-action and take respectively a basis $\{\omega_{A_t}\in\uomega_{A_t}, \eta_{A_t}\}$ such that $\langle \omega_{A_t},\eta_{A_t}\rangle=1$ under the \Poincare pairing and $f_t^*(\omega_{A_t})=\omega_A$. Note that $f_t^*(\eta_{A_t})=|t|_{\CK}^{-1}\eta_A$.

Let $W_r$ be the canonical desingularization of the $r/2$-fold (resp. $r$-fold) fiber product of the universal false  elliptic curve (resp. elliptic curve) over $X_0(2n,n)$ with itself and $C_r=W_r\times A^{r/2}$ (resp. $W_r\times A^r$) when $B\neq M_2(\BQ)$ (resp. $B=M_2(\BQ)$). By \cite[Section 6]{Bro14} and \cite[Section 2]{BDP}, there exists a projector $\epsilon_W\in \mathrm{Corr}(W_r,W_r)$ such that $\epsilon_WH^*_{\dR}(W_r)=H^1_{\dR}(X_0(2n,n), \BH_r)$. The specialization of $\epsilon_W$ at $P_1$ gives a projector $\epsilon_A\in\mathrm{Corr}(A^{r/2},A^{r/2})$ (resp. $\epsilon_A\in\mathrm{Corr}(A^{r},A^{r})$) such that $\epsilon_AH^*(A^{r/2})=\Sym^{r}\BH_A$ (resp. $\epsilon_AH^*(A^{r})=\Sym^{r}\BH_A$). Let  $\epsilon=\epsilon_W\epsilon_A$. The generalized Heegner cycle $\Delta_t:=\epsilon\Gamma_t^{r/2}$ (resp. $\epsilon\Gamma_t^{r}$) in $C_r$ is fibered over $P_t$ and has codimension $r+1$. When $r>0$, denote \[\frac{1}{\sharp \BA^{\infty,\times}\CK^\times\bs \BA_{\CK}^{\infty,\times}/V}\sum_{t\in\BA^{\infty,\times}\CK^\times\bs \BA_{\CK}^{\infty,\times}/V}\chi(t)|t|_\CK^{-1}\Delta_t\]
by $z^\chi$ and when $r=0$, denote its modification by the Hodge class (as in  \cite[Section 6]{Dis22}) by $z^\chi$. 

 We shall consider the image of $z^\chi$ in $H^{2r+2}_{et}
(C_r,L)$ under the \'etale cycle class map.
 When  $B=M_2(\BQ)$ (resp. $B\neq M_2(\BQ)$), let $V$ be the (resp. $e$-component of) the relative \'etale cohomology of the universal (resp. false) elliptic curve on $X_0(2n,n)$ and $V_A$ be the (resp. $e$-component of) the \'etale cohomology of $A$, all with $L$-coefficients. By Lieberman's trick, 
\[ \epsilon H^{2r+2}_{et}(C_r, L)=H^2_{et}(X_0(2n,n), \TSym^r(V))\otimes \TSym^r(V_A)\]
Since $V^\vee:=\Hom_L(V,L)\cong V(1)$, one has $\cl(\Delta_t)\in H^2_{et}(X_0(2n,n), \TSym^r(V^\vee)(1)\otimes \TSym^r(V_A))$ and thus $\cl(z^\chi)\in H^2_{et}(X_0(2n,n), \TSym^r(V^\vee)(1)\otimes \chi_{-1})$, which is cohomologically trivial  by \cite[Lemma 6.2.1]{Dis22}. 
Denote  its image in $H^1(\CK, H^1_{et}(X_0(2n,n)_{\bar{\BQ}}, \TSym^r(V^\vee)(1)\otimes \chi_{-1}))$ again by $z^\chi$.

Denote $\pi\otimes|\cdot|$ by $\pi_1$. The irreducible smooth  $\BB_f^\times$-representation $\pi_1^\infty$ admits a model over $M_\pi$, the Hecke field of $\pi$. Base change this model to $L$ and by abuse of notations, denote the outcome also by $\pi_1^\infty$. Let $V_\pi$ be the $2$-dimensional $G_{\BQ}$-representation such that $L(s-1/2,V_\pi)=L(s+1,\pi)$.   By \cite[Proposition 2.5.2]{Dis22}, for all small enough open compact subgroup $U\subset \BB_f^\times$,
\[\Hom_{G_{\BQ}}(H^1_{et}(X_{U,\bar{\BQ}}, \TSym^r(V)(1)), V_\pi)\cong  \pi_1^{\infty,U}.\]
When $r=0$, let $A_\pi$ be the abelian variety with Tate module $V_\pi$. Then $\pi_1^{\infty,U}=\Hom^\circ(J_U,A_\pi)$. Fix a generator $\omega_\pi\in \Fil^0 D_{\dR}(V_\pi)$. For $\varphi\in\pi^U$, one can decompose $\varphi=\varphi_f\otimes\varphi_\infty$ such that 
 \[(\varphi\otimes|\cdot|)_f^*(\omega_\pi)=\omega_\varphi
\in H^0(X_U,\Omega\otimes\uomega^r)\subset H_{\dR}^1(X_U, \BH_r)\]
 Then by \cite{NN16}, one has 
\[z^{\varphi,\chi}:=(\varphi\otimes|\cdot|)_f(z^\chi)\in H_f^1(\CK^S, V_{\pi}\otimes\chi_{-1}).\]
Here $S$ is a finite set of places containing $p$ and all ramified places of the Galois representation; the superscript $S$ means cohomology of the Galois group $\mathrm{Gal}(\mathcal{K}^{S}/\mathcal{K})$ of the maximal algebraic extension $\CK^S$ of $\mathcal{K}$ unramified outside $S$ and the subscript $f$ means the Bloch--Kato Selmer group.
\subsection{Explicit reciprocity law}
 For any $G_{\CK}$-representation $\rho$ (with $L$-coefficients), let \[D_{\CK_p,\dR}(\rho)= \left(\otimes_{\fp\mid p}(\rho|_{G_{\CK_\fp}}\otimes B_{\dR})^{G_{\CK_\fp}}\right)= D_{\CK_p,\iota,\dR}\oplus D_{\CK_p,\iota^c,\dR},\quad 
D_{\CK_p,\iota,\dR}=D_{\CK_p,\dR}(\rho)\otimes_{\CK_p\otimes L,\iota}L\]
Here  $B_{\dR}$ is Fontaine's de Rham period ring and the filtration $\Fil^\bullet$ on $B_{\dR}$ induces a filtration on $D_{\CK_p,\dR}(\rho)$.  For more details on the de Rham functor, we refer to Subsection \ref{sc:Hodge} below. Since $V_\pi\otimes \chi_{-1}$ is conjugate self-dual,
we obtain an element $\omega_{\pi,\chi}$ in
$\Fil^0D_{\CK_p,\iota,\dR}(V_\pi^\vee(1)\otimes\chi^{-1}_1)$ from $(j!)^{-1}\omega_\pi\otimes \omega_A^j\eta_A^{r-j}$ by  conjugation.  Denote by $\log_p$ the inverse of the isomorphism 
\[D_{\CK_p,\dR}(V_\pi\otimes \chi_{-1})/\Fil^0D_{\CK_p,\dR}(V_\pi\otimes\chi_{-1})\cong H^1_f(\CK_p, V_\pi\otimes \chi_{-1}).\]
Note that \[\Fil^0D_{\CK_p,\dR}(V_\pi^\vee(1)\otimes\chi^{-1}_1)\cong \left(D_{\CK_p,\dR}(V_\pi\otimes \chi_{-1})/\Fil^0D_{\CK_p,\dR}(V_\pi\otimes\chi_{-1})\right)^\vee.\]
\begin{thm}\label{ERC1}For $\chi\in\CW(\Gamma)$ algebraic of infinite type $(1+j,r+1-j)$, one has 
\[\CP_\varphi(\chi)/\Omega_p^{r-2j}=\langle\log_p z^{\varphi,\chi}, \omega_{\pi,\chi}\rangle\]
\end{thm}
\begin{remark}By the definition of $\omega_{\pi,\chi}$, Theorem \ref{ERC1} only deals with the $\iota$-component $\log_p z^{\varphi,\chi}$. To deal with the $\iota^c$-component, one can consider  the new Rankin pair  obtained from $(\sigma,\chi/\CK)$ by complex conjugation and run the whole machinery again. The upshot is the explicit reciprocity law for the new Rankin pair computes the $\iota^c$-component of $\log_pz^{\varphi,\chi}$.  
\end{remark}
Since $H_{\dR}^1(X_U,\BH_r)=D_{\BQ_p,\dR}(H^1_{et}(X_U, \TSym^r V))$ and \[D_{\BQ_p,\dR}(H^1_{et}(X_U, \TSym^r V))\otimes \CK_p=D_{\CK_p,\dR}(H^1_{et}(X_U, \TSym^r V)),\] the above theorem is equivalent to the following theorem by the adjunction formula and the definition of Abel-Jacobi map.
\begin{thm}\label{ERC2}For $\chi\in\CW(\Gamma)$ algebraic of infinite type $(1+j,r+1-j)$
\[\CP_\varphi(\chi)=(j!)^{-1}\Omega_p^{r-2j}\langle AJ_p(z^\chi), \omega_\varphi\wedge\omega_A^j\wedge\eta_A^{r-j}\rangle\]
\end{thm}
 Now we apply the syntomic formalism in \cite{ABSV}\footnote{The case $p=2$ is not excluded as explained by authors in a private communication. The only new ingredient in the $p=2$ case is to obtain the crystalline site  by asking the nilpotence of the divided power of the ideal instead of  just the power of the ideal.} to establish Theorem \ref{ERC2}. We start with the following  crucial proposition{\footnote{When $\Cond\pi_p=0$, we can also deduce it  from   \cite[Lemma 7.2]{AI} and \cite[Proposition 7.1]{Bro14}, \cite[Proposition 3.18]{BDP}.}}
\begin{prop}\label{Syn} For any  $\alpha\in\Sym^r\BH_A$, \[\mathrm{AJ}_p(\Delta_t)(\omega_\varphi\otimes\alpha)=\langle f_t^*F_\varphi(P_t),\alpha\rangle_A \]
Here $\langle-,-\rangle_A$ is the \Poincare pairing on $A$.
\end{prop}
\begin{proof}Since both $\BH$ and $V$ are already defined at the maximal level, $\BH$ and $V$ (resp. $\BH_r$ and $\TSym^rV$) are crystalline attached in the sense of \cite{ABSV}. Moreover the assumption $\pi_p$ is principal series guarantees that the filtered $(\Phi,N)$-module attached to $V_\pi$ is potentially crystalline. Thus we can use the theory in \cite{ABSV} for the bundle $\BH_r$ and $\BH_r\otimes \Sym^r\BH_A$. 

Consider the embedding from CM points to $X_U$. By \cite[Proposition 3.8]{Wal25} (see also the Appendix \ref{Heegner cycles in family}), $z^\chi$ coincides with the image of certain basic vector under the Gysin map.
Thus we can use the  projection formula \cite[Proposition 16.8]{ABSV}  
and  thus the argument in \cite[Theorem 6.4.1]{HW24} still works. Let $s_\omega$ be a local lift of $\omega_\varphi$ at $P_t$ and let $\cl(\Delta_t)$ be the class of $\Delta_t$ in $\Sym^r H_{A_t}\otimes \Sym^r H_A$
\[\mathrm{AJ}_p(\Delta_t)(\omega_\varphi\otimes\alpha)=\langle s_\omega\otimes\alpha, \cl(\Delta_t)\rangle\]
Via the explicit description of Coleman primitive (with the $q$-expansion in the first several lines in \cite[P. 39]{Man24} replaced by the Serre-Tate expansion), the argument in \cite[Lemma 6.2]{HW24} shows that $s_\omega=F_\varphi(P_t)$. As  in \cite[Proposition 3.21]{BDP}, we can deduce the equality
\[\langle F_\varphi(P_t)\otimes\alpha, \cl(\Delta_t)\rangle
=\langle f_t^*F_\varphi(P_t),\alpha\rangle_A\]
\end{proof}
\begin{proof}[Proof of Theorem \ref{ERC2}]By construction and Proposition \ref{Syn}, \[\mathrm{AJ}_p(\Delta_t)(\omega_\varphi\otimes\alpha)=\langle f_t^*F_\varphi(P_t),\alpha\rangle_A\quad \forall \ \alpha\in\Sym^r\BH_A.\]
Take a residue annulus $\CV$ containing  $P_t$ in the strict neighborhood on which $F_\varphi$ is rigid analytic.
As explained in \cite[Proposition 7.8]{AI}, the de Rham sheaf $\uomega|_\CV\subset \BH_\CV$ admits a basis $\{\omega_t^\prime\in\uomega,\eta_t^\prime\}$ such that 
\begin{itemize}
\item $\omega^\prime_t$ specializes to $\omega_{A_t}$ and $\eta_t^\prime$ specializes to $\eta_{A_t}$,
\item $\nabla\omega^\prime=\eta^\prime\otimes db$ for some function $b$ on $\CV$ and $\nabla\eta^\prime=0$.
\end{itemize}
Write $F_\varphi|_{\CV}=\sum_{j=0}^r(-1)^jg_j(\omega^\prime)^{r-j}(\eta^\prime)^j$. Then as \cite[Lemma 3.22]{BDP},
\[\mathrm{AJ}_p(\Delta_t)(\omega_\varphi\otimes\omega_A^{r-j}\eta_A^j)=|t|_{\CK}^{-j}g_j(P_t).\]

Let $F_{\varphi,r}$ be $(-1)^r$ times the image of $F_\varphi$ along the quotient map $\BH_r\to\uomega^{-r}$. Write $\nabla^{r-j}F_{\varphi,r}|_{\CV}=\sum_{k=0}^{r-j}h_k (\omega^\prime)^{r-2j-k}(\eta^\prime)^k$.
Then by the argument in \cite[Proposition 7.8]{AI}, $h_0=\partial^{r-j}g_r=\frac{r!}{j!}g_j$. By construction,  $\nabla^{r-j}F_{\varphi,r}(P_t)=h_0(P_t)$. On the other hand, by the proof of \cite[Proposition 7.6(ii)]{AI} together with the Serre-Tate expansion result in \cite{Fan},
\[\nabla^{r-j}F_{\varphi,r}=r!\nabla^{-1-j}\varphi\quad  \forall\ 0\leq j\leq r\]
as a rigid analytic section of $\BW_{r-2j}$ on some strict neighborhood. Now the stated results follow from the definition of $\CP_\varphi$.
\end{proof}

For completeness, let us record the Kronecker limit formula when $k=2$. Choose $\varphi=\otimes\varphi_v\in I_{P(\BA)}^{\GL_2(\BA)}(\omega|\cdot|^{-1})$,
such that $\varphi_\infty$ is holomorphic of weight $2$ and $\varphi_p=\Vol(1+p^n\BZ_p)^{-1}\Char_{1+p^n\BZ_p}$ with $n\gg0$. Let $u_\varphi\in H^1(Y, \CO_Y^\times)$ be the modular unit such that $d\log u_\varphi=E_\varphi\otimes|\cdot|$. Take an algebraic Hecke character  $\chi\in \Sigma^+(2)$ of infinite type $(-1,-1)$.   Now we can state the following generalization of \cite[Proposition 7.9]{AI}:
\begin{thm}\label{weight two Eisenstein}With notations as above, one has \[\CP_\varphi(\chi)=\frac{1}{\sharp \BA^{\infty,\times}\CK^\times\bs \BA_{\CK}^{\infty,\times}/V}\sum_{t\in\BA^{\infty,\times}\CK^\times\bs \BA_{\CK}^{\infty,\times}/V}\chi(t)|t|_{\CK}^{-1}\log_p(u_\varphi(P_t))\]
\end{thm}
\begin{proof}The discussion on $p$-adic logarithm  in \cite[Section 7.2.1]{AI} applies in our setting. Thus the argument in \cite[Proposition 7.9]{AI} carries over verbatim.
\end{proof}

\section{Local aspect of the $\pm$-Iwasawa theory}
In the remainder of the article, we develop a new $\pm$-Iwasawa theory and combine it with the $p$-adic Waldspurger formula to establish the $p$-converse theorem for self-dual Hecke characters. In this section, we focus on the local aspects of this $\pm$-theory. From now on, we assume $p$ is non-split in $\mathcal{K}$ except in Section \ref{split case}.

\subsection{Some $p$-adic Hodge Theory}\label{sc:Hodge}
Let $K$ be a finite extension of $\mathbb{Q}_p$ and let $V$ be a continuous representation of $G_K$ on a $\mathbb{Q}_p$-vector space of dimension $d$. Let $\mathbb{B}_{\mathrm{dR}}$ be the de Rham period ring of Fontaine and let $D_{K,\mathrm{dR}}(V):=(V\otimes B_{\mathrm{dR}})^{G_K}$ with a filtration $\mathrm{Fil}^\bullet$. We call $V$ a de Rham representation if $\mathrm{dim}_{K} D_{K,\mathrm{dR}}(V)=d$. Let $\mathbb{B}_{\mathrm{cris}}\subset \mathbb{B}_{\mathrm{dR}}$ be the crystalline period ring of Fontaine with the Frobenius action $\varphi$. Denote $D_{K, \mathrm{cris}}(V):=(V\otimes \mathbb{B}_{\mathrm{cris}})^{G_K}$. Let
$$\mathrm{exp}=\exp_{K,V}: \frac{D_{K,\mathrm{dR}}(V)}{\mathrm{Fil}^0D_{K,\mathrm{dR}}(V)+D_{K,\mathrm{cris}}(V)^{\varphi=1}}\hookrightarrow H^1(K, V)$$
be the Bloch--Kato exponential map, which is the boundary map in the cohomology exact sequence of
$$0\rightarrow V\rightarrow  V\otimes(\mathbb{B}^{\varphi=1}_{\mathrm{cris}}\oplus \mathbb{B}^+_{\mathrm{dR}})\rightarrow V\otimes\mathbb{B}_{\mathrm{dR}}\rightarrow 0$$
obtained from the fundamental exact sequence of $p$-adic Hodge theory
$$0\rightarrow \mathbb{Q}_p\rightarrow \mathbb{B}^{\varphi=1}_{\mathrm{cris}}\oplus\mathbb{B}^+_{\mathrm{dR}}\rightarrow \mathbb{B}_{\mathrm{dR}}\rightarrow 0.$$
We also denote the dual exponential map
$$\exp^*=\exp^*_{K,V}: H^1(K, V^*(1))\rightarrow \mathrm{Fil}^0 (D_{K,\mathrm{dR}}(V^*(1)))$$
as the dual of the exponential map for $V$ under the Tate local duality.
We define two spaces
$$H^1_e=\mathrm{ker}\{H^1(K, V)\rightarrow H^1(K, \mathbb{B}^{\varphi=1}_{\mathrm{cris}}\otimes V)\}$$
and
$$H^1_f=\mathrm{ker}\{H^1(K, V)\rightarrow H^1(K, \mathbb{B}_{\mathrm{cris}}\otimes V)\}.$$
In other words, the $H^1_e$ is the image of the exponential map.

In practice we often assume $V$ is an $L$-vector space and the Galois action is $L$-linear. Thus all these spaces above are also $L$-vector spaces.

We make the convention that the Hodge-Tate weight of the cyclotomic character is $+1$. The following lemma is useful.
\begin{lemma}
Suppose $V$ is a $2$-dimensional de Rham $L$-representation of $G_{\mathbb{Q}_p}$ and Hodge-Tate weight $(a_1,a_2)$ for $a_1\leq 0$ and $a_2> 0$ and such that $V^{G_K}=0$. Then $H^1_f(\mathbb{Q}_p, V)$ has dimension $1$. If in addition that $D_{\mathrm{cris}}(V)^{\varphi=1}=0$, then $H^1_e(\mathbb{Q}_p, V)$ also has dimension $1$.
\end{lemma}
\begin{proof}
This is clear from the dimension formula, say in \cite[Proposition 2.8, 2.21]{Bellaiche}.
\end{proof}
\begin{remark}
By identifying $H^1(\mathcal{K}_p,V)$ with $H^1(\mathbb{Q}_p, \mathrm{Ind}^{\mathcal{K}_p}_{\mathbb{Q}_p}V)$ and $D_{\CK_p,\dR}(V)\cong D_{\BQ_p,\dR}( \mathrm{Ind}^{\mathcal{K}_p}_{\mathbb{Q}_p}V)$ using Shapiro lemma, we see that for the local $G_{\mathcal{K}_p}$-representation $V$ associated to an $L$-valued  Hecke character of $\mathbb{A}^\times_\mathcal{K}$ of Archimedean type $(a_1,a_2)$ with $a_1< 0$ and $a_2\geq 0$ (so Hodge-Tate weight is $(-a_1, -a_2)$) with $a_1+a_2=-1$,  $D_{\CK_p,\mathrm{cris}}(V)^{\varphi=1}=0$ by considering the motivic weight. Then the $\exp$ map (and thus also its inverse map $\log$) and the $\exp^*$ map are both isomorphisms between $1$-dimensional $L$-vector spaces.
\end{remark}

\subsection{Vanishing Orders of CM $L$-functions}
 We recall some results on the generic vanishing orders of CM $L$-functions, which will be used  repeatedly later. A Hecke character $\psi_0$ of $\mathcal{K}^\times\backslash\mathbb{A}^\times_\mathcal{K}$ is called {\bf self-dual} if it has Archimedean type $(-n-1,n)$ for some natural number $n$ and its restriction to $\mathbb{A}^\times_\mathbb{Q}$ is $\chi_{\CK}|\cdot|^{-1}$ where $\chi_{\CK}$ is the quadratic character for $\mathcal{K}/\mathbb{Q}$. The following proposition is proved by Greenberg \cite{Greenberg1}.
\begin{proposition}\label{Greenbergnonvanishing}
Let $M$ be a positive integer. Then for all but finitely many self-dual characters $\psi_0$ whose conductor has norm bounded by  $M$ (but allowing general  $n$), we have either $L(\psi_0,1)\not=0$ or the global root number for $\psi_0$ is $-1$.
\end{proposition}
The next proposition is proved by Jia in \cite{Jia} generalizing an earlier result of Rohrlich by removing the class number $1$ assumption.
\begin{proposition}\label{Jianonvanishing}
Let $P$ be a finite set of rational primes. Then for all but finitely many self-dual characters $\psi_0$ of Archimedean type $(-1,0)$ whose conductors are supported in $P$, we have
\begin{itemize}
\item If the global root number for $\psi_0$ is $+1$, then $L(\psi_0,1)\not=0$.
\item If the global root number for $\psi_0$ is $-1$,  the vanishing order of $L(\psi_0,s)$ at $s=1$ is $1$.
\end{itemize}
\end{proposition}

\subsection{$\pm$-theory}
Recall that $\Gamma_\CK$ and $\Gamma_\CK^-$ is the Galois group of the $\BZ_p^2$ and anticyclotomic $\BZ_p$ extension of $\CK$ respectively. Let $\Psi_\mathcal{K}$ be the natural character $G_\mathcal{K}\rightarrow \Gamma_\mathcal{K}\hookrightarrow\CO_L[[\Gamma_\mathcal{K}]]$
and $\Psi_\mathcal{K}^-$ be the character $G_\mathcal{K}\rightarrow \Gamma^-_\CK\hookrightarrow\mathcal{O}_L[[\Gamma^-_\CK]]$.
For a Hecke character $\chi$ of $\mathcal{K}^\times\backslash\mathbb{A}^\times_\mathcal{K}$, we write $\boldsymbol{\chi}$ to be $\chi\Psi_\CK^-$.  We often write $\phi_0$ for the origin point of $\mathrm{Spec}\Lambda^-$. 

Consider $\Lambda^{-,\prime}=\mathcal{O}_L[[p^{-m}X]]$ and   $\Lambda^{-,\prime}_L=\mathcal{O}_L[[p^{-m}X]]\otimes L$ for some $m\gg0$. One way to think of $\Lambda^{-,\prime}$ is to compare it with Tate algebras: we have $L\langle p^{-m}X\rangle \subset \mathcal{O}_L[[p^{-m}X]]\otimes L\subset L\langle p^{-m-1}X\rangle$.
A $\BC_p$-point $\phi$ of $\mathrm{Spec}\mathcal{O}_L[[p^{-m}X]]$ sends $X$ to $\phi(X)$ with $|\phi(X)|_p<p^{-m}$.
By abuse of notations, we also write $\Psi_\CK^-$ $\boldsymbol{\chi}$ for the corresponding characters with coefficient ring changed to  $\Lambda^{-,\prime}$ or $\Lambda^{-,\prime}_L$. We will indicate the precise ring when necessary.

A $\BC_p$-point $\phi\in\mathrm{Spec}\mathcal{O}_L[[p^{-m}X]]$ is called {\bf  an arithmetic point} if there is an even integer $\kappa$ such that the specialization of $\Psi^-_\mathcal{K}$ to $\phi$ is the $p$-avatar of a Hecke character of $\mathcal{K}^\times\backslash \mathbb{A}^\times_{\mathcal{K}}$ of Archimedean type $(-\kappa, \kappa)$. Sometimes we write $\kappa_\phi$ for the $\kappa$ corresponding to this $\phi$.

\begin{definition}
Let $\mathfrak{X}_1$ be the set of arithmetic points $\phi$  such that the corresponding  $\kappa_\phi$ is non-negative and  divisible by $2(p^2-1)$. Let $\mathfrak{X}_2$ be the set of arithmetic points $\phi$  such that the corresponding $\kappa_\phi\leq -2$ and   divisible by $2(p^2-1)$.
\end{definition}
The reason to impose the condition that $\kappa_\phi$ is divisible by $2(p^2-1)$ is due to Lemma \ref{punchged}. It is the analogue of the phenomenon that powers of the cyclotomic character live in the same component in the $p$-adic family only when the weights are congruent modulo $(p-1)$.

We shall also need to consider the $2$-variable weight space $\mathcal{O}_L[[p^{-m}X,p^{-m'}T]]\otimes L$ where $\mathcal{O}_L[[p^{-m'}T]]$ is the weight space parameterizing the Coleman family $\mathcal{F}$ (in the case we consider the slope of $f$ is $\frac{1}{2}$, the weight is $2$, and the corresponding automorphic representation is either ramified or has distinct Satake parameters at $p$, thus the corresponding point on the eigencurve is \'etale over the generic fiber of the weight space). An arithmetic point $\phi$ of $\mathrm{Spec}\mathcal{O}_L[[p^{-m}X,p^{-m'}T]]\otimes L$ sends $(1+X)$ to $(1+p)^{\kappa_\phi}$ and sends $(1+T)$ to $(1+p)^{k_\phi-2}$ where $\kappa_\phi$ and $k_\phi\geq 2$ are integers. We define $2$ sets of arithmetic point as follows. To keep the notations same as \cite[Theorem 5.13]{BDP} and \cite[Theorem 7.1]{AI} we let $j_\phi=\kappa_\phi+\frac{k_\phi-2}{2}$.
\begin{definition}
We define the set $\mathfrak{X}_3$ to be the arithmetic points $\phi$ such that $-\frac{k_\phi-2}{2}\leq \kappa_\phi\leq \frac{k_\phi-2}{2}$ and  $\kappa_\phi$ is divisible by $2(p^2-1)$
. Let $\mathfrak{X}_4$ be the arithmetic points $\phi$ such that $\kappa_\phi<-\frac{k_\phi-2}{2}$ and $\kappa_\phi$ is divisible by $2(p^2-1)$. Let $\mathfrak{X}_5$ be the arithmetic points $\phi$ such that $\kappa_\phi>\frac{k_\phi-2}{2}$ and $\kappa_\phi$ is divisible by $2(p^2-1)$. By abuse of notations, we also use $\mathfrak{X}_3$, $\mathfrak{X}_4$ and $\mathfrak{X}_5$ for the corresponding sets of arithmetic points for a fixed $f$ (compared to a Coleman family $\mathcal{F}$).
\end{definition}
Let $\psi$ be a self-dual Hecke character on $\CK^\times\bs\BA_\CK^\times$ with infinite type $(-1,0)$.
\begin{lemma}\label{punchged}
Suppose $m>1$ and $\kappa_\phi$ is divisible by $2(p^2-1)$. At the prime $p$, the $p$-local component of the Hecke character corresponding to $\boldsymbol{\psi}_\phi$ is the same as the $p$-local component of $\psi$ (recall we assumed $p$ is non-split in $\mathcal{K}$).
\end{lemma}
\begin{proof}
If it were not the case, then it is easy to show by considering the image of $\mathcal{K}^\times_p\subset \mathbb{A}^\times_\CK$ under the reciprocity map and the formula for the $p$-adic avatar of a Hecke character, that the image of the specialization of $\Psi^-_\mathcal{K}$ to $\phi$ must contain a primitive $p$-th root of unity or $-1$ times an element $y\in\mathcal{O}^\times_{\mathcal{K}_p}$ with image $1$ in its residue field. Here we need the assumption on $\kappa_\phi$ to ensure that $\kappa_\phi$ kills the multiplicative group of the residue field of $\mathcal{O}_{\mathcal{K}_p}$, and that $\sqrt{-D}/\sqrt{-D}^c=-1$. Here $-D$ is the fundamental discriminant of $\CK$. This contradicts that this character takes values in a small neighborhood of $1$ since $m>1$.
\end{proof}

\begin{lemma}\label{root number}
Suppose $m>1$. For each $\phi\in\mathfrak{X}_1$ the character $\boldsymbol{\psi}_\phi$ has same root number as $\psi$; for each $\phi\in\mathfrak{X}_2$ the root number of the character $\boldsymbol{\psi}_\phi$ is $-1$ times the root number of $\psi$.
\end{lemma}
\begin{proof}
We consider the local root numbers of the specializations. The Archimedean root numbers for Hecke characters $\boldsymbol{\psi}_\phi$ with $\phi$ in $\mathfrak{X}_1$ are the same as that for $\psi$, while the Archimedean root numbers for Hecke characters $\boldsymbol{\psi}_\phi$ with $\phi$ in $\mathfrak{X}_2$ become that for $\psi$ multiplied by $-1$ (say using the formula in \cite[Lecture 2, 5.1]{Rohrlich1}, which implies the local root number is $i^{-|a-b|}$ if the Archimedean type is $(a,b)$). Note that the specialization of $\Psi^-_\mathcal{K}$ at $\phi$ is unramified at all primes $v\nmid p\infty$.  If $v$ is non-split in $\mathcal{K}$, then the image of the uniformizer $\varpi_v$ in $\Gamma^-_\CK$ under the reciprocity map is the identity, since its square is $1$ and $\Gamma^-_\CK$ is torsion-free. If $v=w\bar{w}$ is split in $\mathcal{K}$, then the local root number for $\boldsymbol{\psi}_\phi$ at $v$ is $\boldsymbol{\psi}_\phi((-1)_w)$ where $(-1)_w$ is the \idele element at $w$ taking value $-1$. Again, the image of $(-1)_w$ in $\Gamma^-_\CK$ under the reciprocity map is the identity, since the square of $(-1)_w$ is $1$ and $\Gamma^-_\CK$ is torsion-free. Together with Lemma \ref{punchged}, we see the conclusion of the lemma by taking the product of the root number over all places.
\end{proof}

\subsection{Explicit Reciprocity Law for Elliptic Units}
We recall some results about elliptic units following Kato in \cite[Chapter 15]{Kato}. We take an elliptic curve $E$ with CM by $\mathcal{O}_\mathcal{K}$. Let  $\psi_E$ be the associated Hecke character, which has  infinity type $(-1,0)$. Let $\mathfrak{f}$ be an ideal of $\mathcal{O}_\mathcal{K}$ contained in the conductor of $\psi_E$. Let $\psi$ be a Hecke character of $\mathcal{K}^\times$ of infinity type $(-r,0)$ with conductor dividing $\mathfrak{f}$. We fix a number field $M$ containing the values of $\psi$ and assume $L$  contains $M$. Following \emph{loc.cit.} let $V_M(\psi)=H^1(E(\mathbb{C}), \mathbb{Q})^{\otimes r}\otimes_\mathcal{K} M$. The \'etale realization  $V_L(\psi)$ of $\psi$ is $H^1_{\mathrm{et}}(E_{\bar{\mathbb{Q}}}, \mathbb{Q}_p)^{\otimes r}\otimes_\mathcal{K} L$ equipped with the Galois action given by $\psi$ (i.e., twisted by the character $\psi\cdot\psi_E^{-r}$ which is of finite order and conductor dividing $\mathfrak{f}$). We define the $1$-dimensional $M$-vector space $S(\psi)$ to be $\mathrm{coLie}(E)^{\otimes r}\otimes M$. We also define a map
$$\mathrm{per}_\psi: S(\psi)\rightarrow V_\mathbb{C}(\psi):=V_M(\psi)\otimes_M\mathbb{C}$$
as the map induced by the period map
$$\mathrm{coLie}(E)\rightarrow H^1(E(\mathbb{C}),\mathbb{C}).$$
Let $g_\psi$ be the CM form associated to $\psi$ as in \cite[15.10]{Kato}. Let $S(g_\psi)$ be the $M$-vector space spanned by the newform $g_\psi$. We define
$$V^\sim_L(\psi):=V_L(\psi)\oplus \tilde{c}\cdot V_L(\psi)$$
where $\tilde{c}\in\mathrm{Gal}(\bar{\mathbb{Q}}/\mathbb{Q})$ is a lift of $c\in\Gal(\CK/\BQ)$ and the action of $\mathrm{Gal}(\bar{\mathbb{Q}}/\mathbb{Q})$ on $\tilde{c}\cdot V_L(\psi)$ is the Galois action on $V_L(\psi)$ composed with the conjugation by $c$.

As in \cite[15.11]{Kato}, we fix an isomorphism of $1$-dimensional $M$-vector spaces
$$S(\psi)\simeq S(g_\psi),$$
which determines a unique isomorphism
$$V^{\sim}_L(\psi)\simeq V_L(g_\psi)$$
as in \cite[Lemma 5.11 (1)]{Kato} (the $V_L(g_\psi)$ is the Galois representation associated to $g_\psi$).

More generally, for Hecke characters $\psi$ of general Archimedean type $(-r_1,r_2)$ with $r_i\geq 0$, then $\psi\cdot |\cdot|^{-r_2}$ is of Archimedean type $(-r_1-r_2,0)$. We can define the $V(\psi)$ and $S(\psi)$ to be those of $\psi\cdot |\cdot|^{-r_2}$ as above, twisted by $(\zeta_{p^n})_n^{\otimes (-r_2)}$ (the $-r_2$-th power of the Tate motive).
Let $\mathcal{K}_{p^\infty\mathfrak{f}}$ be the ray class field of $\mathcal{K}$ of conductor $p^\infty\mathfrak{f}$. Let $S$ be a finite set of primes containing the primes dividing $p\mathfrak{f}$. We define for any $\mathbb{Z}_p$ representation $T$ of $G_\mathcal{K}$ the Iwasawa cohomology
$$H^1_{\mathrm{Iw}}(\mathcal{K}^S_{p^\infty\mathfrak{f}}, T):=\varprojlim_{\mathcal{K}'} H^1(\mathcal{K}^{\prime,S}, T)$$
for $\mathcal{K}'$ running through all finite sub-extensions of $\mathcal{K}_{p^\infty\mathfrak{f}}$. Here the superscript $S$ means cohomology of the Galois group $\mathrm{Gal}(\mathcal{K}^{\prime,S}/\mathcal{K}^\prime)$ where $\mathcal{K}^{\prime,S}$ is the maximal algebraic extension of $\mathcal{K}^\prime$ unramified outside $S$.

In \cite[Section 15.6]{Kato}, the family of elliptic units $z_{p^\infty\mathfrak{f}}$ is defined as an element in $H^1_{\mathrm{Iw}}(\mathcal{K}^S_{p^\infty\mathfrak{f}}, \mathbb{Z}_p(1))$. (In fact in the terminology of \cite[Section 5.1]{JK} one needs to multiply it by any element in the ideal $\mathcal{J}_\Lambda$ (in the notation of \emph{loc.cit.}) to make it in $H^1_{\mathrm{Iw}}(\mathcal{K}^S_{p^\infty\mathfrak{f}}, \mathbb{Z}_p(1))$, but this $\mathcal{J}_\Lambda$ is the unit ideal in our weight ring $\Lambda^{-}_L$). We have the following slight generalization of \cite[Proposition 15.9]{Kato}.

From now on we suppose the set of prime divisors of $\mathfrak{f}$ is contained in the union of $\{p\}$ and the set of prime divisors of the conductor of $\psi$. This will guarantee the involved $L$-function is primitive. 
\begin{proposition}
Let $r\geq 1$ and $\psi$ is a Hecke character of $\mathcal{K}$ of type $(-r_1,r_2)$ with $r_i\geq 0$. Let $\gamma\in V_M(\psi)$. Consider the map
$$H^1_{\mathrm{Iw}}(\mathcal{K}^S_{p^\infty\mathfrak{f}}, \mathbb{Z}_p(1))\rightarrow H^1_{\mathrm{Iw}}(\mathcal{K}^S_{p^\infty\mathfrak{f}},\mathbb{Z}_p(1))\otimes V_L(\psi)\rightarrow H^1_{\mathrm{Iw}}(\mathcal{K}^S_{p^\infty\mathfrak{f}}, V_L(\psi)(1))$$
where the first map is multiplying by $\gamma$ and the second is the natural identification map (note that $\psi$ factorizes through $\mathcal{K}_{p^\infty\mathfrak{f}}$).
Let $z_{p^\infty\mathfrak{f},\psi,\gamma}$ be the image of $z_{p^\infty\mathfrak{f}}$ under the above map.

Then for each character $\chi$ of $\mathrm{Gal}(\mathcal{K}'/\mathcal{K})$ where $\mathcal{K}'$ is a finite extension of $\mathcal{K}$ contained in $\mathcal{K}(p^\infty\mathfrak{f})$, the image of $z_{p^\infty\mathfrak{f},\psi,\gamma}\in H^1_{\mathrm{Iw}}(\mathcal{K}^S_{p^\infty\mathfrak{f}}, V_L(\psi)(1))$ in $H^1(\mathcal{K}', V_L(\psi)(1))$ has image under $\exp^*$ map   in $S(\psi)\otimes \mathcal{K}'$. Further composed with the map
$$\sum_{g\in\mathrm{Gal}(\mathcal{K}'/\mathcal{K})}\chi(g)\mathrm{per}_\psi\circ g: S(\psi)\otimes \mathcal{K}'\rightarrow V(\psi),$$
the image is $L_{p\mathfrak{f}}(\bar{\psi},\chi,r_1-r_2)\cdot\gamma$.
\end{proposition}
\begin{remark}
A more precise way to state the result is the image under the $\exp^*$ map is $\frac{L(\bar{\psi},\chi,r_1-r_2)}{\Omega^{r_1+r_2}_{\omega_E}}\cdot \omega_E^{r_1+r_2}$
where $\Omega^{r_1+r_2}_{\omega_{E}}$ is the ratio between $\omega^{r_1+r_2}_E$ over $\gamma$ under the deRham isomorphism over $\mathbb{C}$, the motive for $\psi$ is viewed as the Serre tensor of $E$ by a finite order character of $\mathbb{A}^\times_\mathcal{K}/\mathcal{K}^\times$, and $\omega_E$ is viewed as a differential for the motive for $\psi$ over the splitting field of this character.
\end{remark}

\begin{proof}\footnote{We thank D. Loeffler for showing us the following argument.}
If the infinity type of $\psi$ is $(-r,0)$, then this is directly proved in \cite[Theorem 15.9]{Kato}. For the general infinity type case, we use the identification along the cyclotomic family of $z_{p^\infty\mathfrak{f},\psi}$ with the zeta element for the CM modular form $g_\psi$ as in \cite[(15.16.1)]{Kato} (under the above identification of $S(\psi)$ with $S(g_\psi)$ and $V^{\sim}_L(\psi)$ with $V_L(g_\psi)$). The proposition then follows from the explicit reciprocity law for zeta elements by \cite[Theorem 12.5]{Kato}.
\end{proof}
Let $T(\psi)\subset V_L(\psi)$ be the $\mathcal{O}_L$-submodule corresponding to $\mathcal{O}_L\subset L$.
\begin{definition}
Fix a basis $\gamma\in T(\psi)$. Define $z_{p^\infty,\psi}\in H^1_{\mathrm{Iw}}(\mathcal{K}^S_{p^\infty}, T(\psi)):=\varprojlim_{\mathcal{K}'} H^1(\mathcal{K}^{\prime,S}, T(\psi))$ ($\mathcal{K}'$ runs through all finite sub-extensions of $\mathcal{K}_{p^\infty}$) as the image of $z_{p^\infty\mathfrak{f},\psi,\gamma}$.
\end{definition}

Let $\psi_0$ be a self-dual Hecke character of $\mathcal{K}^\times\backslash\mathbb{A}^\times_\mathcal{K}$ with Archimedean type $(-1,0)$.  We first show that the family of elliptic unit $z_{p^\infty,\psi_0}$ is nonzero along the anticyclotomic line.

\begin{lemma}\label{freeness}
There is no point $\phi\in\mathrm{Spec}(\Lambda^-_L)$ such that $\psi_0=\Psi_{\CK,\phi}$ (the specialization of $\Psi_\CK$ to $\phi$) as characters of $G_{\mathcal{K}_p}$. Therefore we have that $H^1(\mathcal{K}_p, \boldsymbol{\psi}_0)$ and $H^1(\mathcal{K}^S,\boldsymbol{\psi}_0)$ are both free over $\Lambda^-_L$. Moreover for each $\phi\in\mathrm{Spec}(\Lambda^-_L)$, $H^1(\mathcal{K}_p, \boldsymbol{\psi}_{0,\phi})$ has rank $2$. Thus $H^1(\mathcal{K}_p, \boldsymbol{\psi}_0)$ is free of rank $2$ over $\Lambda^-_L$.
\end{lemma}
\begin{proof}
If it were the case that $\psi=\Psi_{\CK,\phi}$, since $\Psi_{\CK,\phi}=\Psi^{-c}_{\CK,\phi}$, the Hodge-Tate weights of $\Psi_{\CK,\phi}$ should add up to $0$, contradicting that the Hodge-Tate weight for $\psi_0$ is $(1,0)$. For the second statement, note that $\Lambda^-_L$ is a principal ideal domain, it is enough to show that the modules are torsion-free, which in turn follows from the first statement (by considering the Galois cohomology long exact sequence associated to the short exact sequence of multiplying by a nonzero prime element $\varpi\in \Lambda^{-}_L$ on $\boldsymbol{\psi}_0$).  The third statement follows from the first statement and the Euler-Poincar\'e characteristic formula (say \cite[Theorem 2.17]{DDT}).
\end{proof}

\begin{proposition}\label{nontriviality}
The $z_{p^\infty,\psi_0}$ is a nonzero element in $H^1(\mathcal{K}^S, \boldsymbol{\psi}_0)$.
\end{proposition}
\begin{proof}
If the global root number of $\psi_0$ is $+1$, this is a direct consequence of the explicit reciprocity law for elliptic units we recalled above, together with Greenberg's result Proposition \ref{Greenbergnonvanishing}  on non-vanishing of $L$-values.

If the global root number is $-1$, we can show the non-triviality of the elliptic unit family by switching the roles played by $\iota$ and $ \iota^c$ and looking at the image under $\exp^*$ at arithmetic points in $\mathfrak{X}_2$ (where the global root numbers of the specializations are $+1$), and applying Greenberg's result again as above.
\end{proof}
\subsection{Iwasawa theory for elliptic units}
Now we work with the coefficient ring $\Lambda^-_L$. Let $\psi_0$ be a self-dual character of Archimedean type $(-1,0)$. We define various Iwasawa modules below.
Throughout, $*$ always denotes the Pontryagin dual.
\begin{definition}\label{defineSelmer}
Let $$H^1_{\mathrm{str}}(\mathcal{K}^S, \boldsymbol{\psi}_0):=\mathrm{ker}\{H^1(\mathcal{K}^S, \boldsymbol{\psi}_0)\rightarrow H^1(\mathcal{K}_p, \boldsymbol{\psi}_0)\}.$$
 The strict Selmer groups
 \[\mathrm{Sel}^{\mathrm{str}}(\mathcal{K}^S, \boldsymbol{\psi}_0^*):=\mathrm{ker}\{H^1(\mathcal{K}^S, \boldsymbol{\psi}_0^*)\rightarrow \prod_{v\nmid p} H^1(\mathcal{K}_v, \boldsymbol{\psi}_0^*)\times H^1(\mathcal{K}_p, \boldsymbol{\psi}_0^*)\}\]
and the relaxed Selmer groups
\[\mathrm{Sel}^{\mathrm{rel}}(\mathcal{K}^S, \boldsymbol{\psi}_0^*):=\mathrm{ker}\{H^1(\mathcal{K}^S, \boldsymbol{\psi}_0^*)\rightarrow \prod_{v\nmid p} H^1(\mathcal{K}_v, \boldsymbol{\psi}_0^*)\}.\]
Put
$$X^{\mathrm{str}}_{\psi_0}=\mathrm{Sel}^{\mathrm{str}}(\mathcal{K}^S, \boldsymbol{\psi}_0^*)^*$$
and define $X^{\mathrm{rel}}$ similarly.
\end{definition}
We have the following control theorem for strict Selmer groups.
\begin{lemma}\label{strictcontrol}
The specialization map
$$X^{\mathrm{str}}_{\psi_0}\otimes_{\phi_0} L\rightarrow \mathrm{Sel}^{\mathrm{str}}(\mathcal{K},L/\mathcal{O}_L(\psi_0))^*\otimes L$$
is an isomorphism.
\end{lemma}
\begin{proof}
We consider the short exact sequence
$$0\rightarrow \Lambda^*(\boldsymbol{\psi}_0)[X]\rightarrow \Lambda^*(\boldsymbol{\psi}_0)\rightarrow \Lambda^*(\boldsymbol{\psi}_0)\rightarrow 0.$$
Note that $H^1(\mathcal{K}_v, L/\mathcal{O}_L(\psi_0))=0$ for all $v\nmid p$ (say using \cite[Theorem 2.17 (d)]{DDT}).  We just need to study the kernels of
$$H^1(\mathcal{K}, L/\mathcal{O}_L(\psi_0))\rightarrow H^1(\mathcal{K}, \Lambda^*(\boldsymbol{\psi}_0))[X]$$ and
$$H^1(\mathcal{K}_p, L/\mathcal{O}_L(\psi_0))\rightarrow H^1(\mathcal{K}_p, \Lambda^*(\boldsymbol{\psi}_0))[X].$$
These are the cokernels of
$$\times X: H^0(\mathcal{K},\Lambda^*(\boldsymbol{\psi}_0))\otimes_{\mathcal{O}_L}L\rightarrow H^0(\mathcal{K},\Lambda^*(\boldsymbol{\psi}_0))\otimes_{\mathcal{O}_L}L$$
and
$$\times X: H^0(\mathcal{K}_p,\Lambda^*(\boldsymbol{\psi}_0))\otimes_{\mathcal{O}_L}L\rightarrow H^0(\mathcal{K}_p,\Lambda^*(\boldsymbol{\psi}_0))\otimes_{\mathcal{O}_L}L.$$
Taking the Pontryagin dual, we see that they are respectively the dual of the kernels of
$$\times X: \Lambda(\boldsymbol{\psi}_0)_{G_\mathcal{K}}\rightarrow \Lambda(\boldsymbol{\psi}_0)_{G_\mathcal{K}}$$
and
$$\times X: \Lambda(\boldsymbol{\psi}_0)_{G_{\mathcal{K}_p}}\rightarrow \Lambda(\boldsymbol{\psi}_0)_{G_{\mathcal{K}_p}}.$$
The above two kernels are both killed by $X$ and any element $\boldsymbol{\psi}_0(g)-1$ for every $g\in G_{\mathcal{K}_p}$. Clearly $\psi_0|_{G_{\mathcal{K}_p}}$ is not the trivial character. Thus these elements generate the unit ideal and the kernels are both $0$. We are done.
\end{proof}
\begin{lemma}\label{global}
The module $X^{\mathrm{str}}_{\psi_0}$ is torsion over $\Lambda^{-}_L$ and the strict Selmer group $H^1_{\mathrm{str}}(\mathcal{K}^S, \boldsymbol{\psi}_0)$ has rank zero over $\Lambda^{-}_L$.
\end{lemma}
\begin{proof}
The first statement follows from the generic non-triviality of the elliptic unit Euler system over $\Lambda^-$, Proposition \ref{nontriviality}, and the standard Euler system argument (e.g., \cite[Section 2.2]{Rubin3}) bounding the strict Selmer group by the index of the elliptic units. For the second statement, by Galois cohomology long exact sequence, it is enough to show that at some arithmetic point $\phi$ the $H^1_{\mathrm{str}}(\mathcal{K}^S,\boldsymbol{\psi}_{0,\phi})$ has rank $0$. Note that for each $v\nmid p$, $H^1(\mathcal{K}_v, \boldsymbol{\psi}_{0,\phi})$ has $0$ rank for generic $\phi$. The result again follows from the Euler system argument and non-triviality of the elliptic unit Euler system along the anticyclotomic line.
\end{proof}
\begin{corollary}\label{4.3}
The relaxed Selmer group $H^1(\mathcal{K}^S, \boldsymbol{\psi}_0)$ has rank one over $\Lambda^{-}_L$.
\end{corollary}
\begin{proof}
The rank is at least $1$ since we have seen the family of elliptic unit is generically non-trivial along the anticyclotomic line.  We look at the Poitou-Tate exact sequence
\begin{align*}
&0\rightarrow H^1_{\mathrm{str}}(\mathcal{K}^S,\boldsymbol{\psi}_0)\rightarrow H^1(\mathcal{K}^S,\boldsymbol{\psi}_0)\rightarrow H^1(\mathcal{K}_p,\boldsymbol{\psi}_0)&\\&\rightarrow H^1(\mathcal{K}^S, \boldsymbol{\psi}_0^\vee(1)\otimes\Lambda_L^{-,*})^*\rightarrow H^1_{\mathrm{str}}(\mathcal{K}^S, \boldsymbol{\psi}_0^\vee(1)\otimes\Lambda_L^{-,*})^*\rightarrow 0.&
\end{align*}
From Lemma \ref{global} the first and last terms have rank $0$. From the non-triviality of the elliptic units we know the second and fourth terms have ranks at least one. If $H^1(\mathcal{K}^S, \psi_0\otimes\Lambda^{-}_L)$ has rank at least $2$, then $H^1(\mathcal{K}^S, \boldsymbol{\psi}_0^\vee(1)\otimes\Lambda^{-,*}_L)^*$ also  has rank at least $2$. Then by Lemma \ref{freeness} the third term has rank $2$. Thus we get a contradiction.
\end{proof} 

We use $\mathrm{char}$ to denote characteristic ideal.
\begin{theorem}\label{RMCC}
We have
$$\mathrm{char}_{\Lambda^{-}_L}(X^{\mathrm{str}}_{\psi_0})=\mathrm{char}_{\Lambda^{-}_L}(\frac{H^1_{\mathrm{Iw}}(\mathcal{K}^-_\infty, \boldsymbol{\psi}_0)}{\Lambda^{-}_Lz_{p^\infty,\psi_0}}).$$
\end{theorem}
\begin{proof}
By Proposition \ref{nontriviality} the family of zeta elements along the anticyclotomic line is not identically zero. The two-variable Iwasawa main conjecture over the $\mathbb{Z}_p^2$ extension $\mathcal{K}_\infty/\mathcal{K}$ is proved in \cite[Theorem 5.2]{JK}, which generalizes earlier results of Rubin (\cite[Theorem 2 (ii)]{Rubin1}, \cite[Theorem 4.1 (ii)]{Rubin2}) to allow the order of the character divisible by $p$. In \emph{loc.cit.} the result is formulated in terms of Det functor (see \cite[1.3]{JK}) and fundamental line. Our one-variable result follows from the non-vanishing of the elliptic unit Euler system along the anticyclotomic line, and the fact that fundamental line commutes with arbitrary base change. To relate this formulation to the formulation here, we use the similar argument as in \cite[Proposition 7.1]{Kobayashi}. More precisely,
\begin{itemize}
\item Let $T$ be the $p$-adic representation of $G_{\mathcal{K}}$ corresponding to $\psi_0$. Global duality says that $\Sha^2(\mathcal{K}_n^S, \mathrm{Hom}(T,\mathbb{Q}_p/\mathbb{Z}_p(1)))$ is the Pontryagin dual of $\Sha^1(\mathcal{K}_n, S, T)$
where \[\Sha^i(\mathcal{K}_n^S, -)=\mathrm{ker}\{H^i(\mathcal{K}^S_n,-)\rightarrow \prod_{v\in S} H^i(\mathcal{K}_{n,v},-)\}.\]
\item Let $j$ be the natural morphism $\mathrm{Spec} \mathcal{K}_n\rightarrow\mathrm{Spec}\mathcal{O}_{\mathcal{K}_n}[1/p]$. The localization long exact sequence (see \emph{loc.cit.})
\begin{align*}
&0\rightarrow H^1(\mathcal{O}_{\mathcal{K}_n}[1/p], j_*T)\rightarrow H^1(\mathcal{O}_{\mathcal{K}_n}[1/S], T)\rightarrow \oplus_v H^0(k(v), H^1(I_{n,v}, T))\\
&\rightarrow H^2(\mathcal{O}_{\mathcal{K}_n}[1/p], j_*T)
\rightarrow H^2(\mathcal{O}_{\mathcal{K}_n}[1/S], T)\rightarrow \oplus_v H^2(K_{n,v}, T)&
\end{align*}
for intermediate fields $\mathcal{K}_n$ between $\mathcal{K}$ and $\mathcal{K}^-_\infty$, where $v$ runs over all primes outside $p$ dividing $\Sigma$ and $k(v)$ is the corresponding residue field. When comparing the two formulations, the error term comes from the term $H^0(k(v), H^1(I_{n,v}, T))$.
Note that each prime $q\not=p$ split in $\mathcal{K}$ is finitely decomposed in $\mathcal{K}^-_\infty$, and $\varprojlim_n H^0(k(v), H^1(I_{n,v}, T))=0$. Each prime $q\not=p$ non-split in $\mathcal{K}$ splits completely in $\mathcal{K}^-_\infty$, and thus $\varprojlim_n H^0(k(v), H^1(I_{n,v}, T))$ is killed by a fixed (independent of $n$) power of $p$. In each case the above error term vanishes after inverting $p$.
\end{itemize}
\end{proof}

\begin{corollary}\label{1111}
 Suppose $L(\psi_0,s)$ has vanishing order $1$ at $s=1$. Then the specialization $z_{p^\infty,\psi_0}$ at $\phi_0$ has nonzero restriction to $H^1_f(\mathcal{K}_p, \psi_0)$.
\end{corollary}
\begin{proof}
The Gross--Zagier and Kolyvagin theorem implies that the rank of the Bloch-Kato Selmer group for $\psi_0$ is $1$ and is generated by the Mordell-Weil group of the associated $\mathrm{GL}_2$-type abelian variety. Thus by our assumption and Lemma \ref{AV} the rank of its strict Selmer group is $0$. This implies by Lemma \ref{strictcontrol} that $\phi_0$ is not in the support of the characteristic ideal of $X^{\mathrm{str},S}_{\psi_0}\otimes L$ as a module over $\Lambda^{-}_L$. By Rubin's Iwasawa main conjecture, this implies $\phi_0$ is not in the support of the characteristic ideal of $\frac{H^1(\mathcal{K}^S, \boldsymbol{\psi}_0)}{\Lambda^{-}_L\cdot z_{p^\infty,\psi_0}}$, which in turn implies $\phi_0(z_{p^\infty,\psi_0})$ is nonzero. This is an element in the rank $1$ Bloch-Kato Selmer group for $\psi_0$ since it lies in the kernel of the $\exp^*$ map. So by our assumption again, the localization image of $\phi_0(z_{p^\infty,\psi_0})$ in $H^1_f(\mathcal{K}_p,\psi_0)$ is also nonzero.
\end{proof}

\subsection{Main constructions}
 Let $\psi$ be a self-dual character of Archimedean type $(-1,0)$.
 \subsubsection{$p$-adic $L$-functions}\label{padicL}
Let $\chi_3$ be a Hecke character of $\mathcal{K}^\times\backslash \mathbb{A}^\times_\mathcal{K}$ of Archimedean type $(0,0)$, such that $\chi_1:=\psi\chi^{-1}_3$ is unramified at $p$ (and of Archimedean type $(-1,0)$). We require that  $\chi_3$ is unramified outside $p$ and another odd inert prime $\ell_0$. Let $\ell_1$ be an odd inert prime different from $p$ and $\ell_0$ such that $\psi$ and $\chi_3$ are unramified at $\ell_1$. We choose a Hecke character $\chi_{\mathrm{aux}}$ of finite order, unramified outside $\ell_1$ such that $\chi_1\chi_3^c(\chi_{\mathrm{aux}}\chi_{\mathrm{aux}}^{-c})$ has global sign $+1$ and that $L(\chi_1\chi_3^c(\chi_{\mathrm{aux}}\chi_{\mathrm{aux}}^{-c}),1)\not=0$. This is clearly possible: we first take a character of conductor and order a power of $\ell_1$ and use the argument in \cite[Page 247]{Greenberg} for $\ell_1$ to meet the requirement on root number, and then apply  Proposition \ref{Jianonvanishing} to ensure non-vanishing of the central $L$-value. Note that
$\chi_1\cdot\chi_{\mathrm{aux}}\cdot\chi_3\cdot\chi^{-1}_{\mathrm{aux}}=\psi$ and $\chi_1\cdot\chi_{\mathrm{aux}}\cdot(\chi_3\cdot\chi_{\mathrm{aux}}^{-1})^c=\chi_1\chi^c_3(\chi_{\mathrm{aux}}
\cdot\chi_{\mathrm{aux}}^{-c})$. We denote $(\chi_1\chi^c_3(\chi_{\mathrm{aux}}
\cdot\chi_{\mathrm{aux}}^{-c}))^c$ as $\psi'$. Let $\xi:=\chi_1\chi_{\mathrm{aux}}$ and $f:=f_\xi$ be the CM form corresponding to $\xi$. Let $\chi:=\chi_3\chi_{\mathrm{aux}}^{-1}$. We will consider the $p$-adic Waldspurger formula for the pair $(f,\chi/\CK)$.

\begin{definition}\label{pLf}
We write $\mathcal{L}^\circ_{\psi}$ for the $p$-adic $L$-function $\CP_\varphi$ in the Eisenstein case which interpolates the algebraic part of the $L$-values $L(\mathcal{K},\boldsymbol{\psi}_\phi)$ for $\phi$ in $\mathfrak{X}_1$ (using the Eisenstein case). Note that the origin point $\phi_0$ is an interpolation point -- this corresponds to $k=1$, $j=0$ in \cite[2.3.1]{AI}. We write $\mathcal{L}_{f,\chi}$ for the anticyclotomic $p$-adic $L$-function $\CP_\varphi$ for the pair $(f,\chi/\CK)$ (with interpolation points in $\mathfrak{X}_4$). Here we allow certain flexibility of the  test vector $\varphi$.
\end{definition}
Note that these $p$-adic $L$-functions live in some ring $\Lambda^{-,\prime}_L\cong \CO_L[[p^{-m}X]]\otimes L$ for $m\gg0$. 
 
\subsubsection{Local definitions over small discs}
\begin{lemma}\label{choicecharacter}
We have
\begin{itemize}
\item There is a finite order anticyclotomic character $\chi'$ whose order and conductor are both a power of an odd prime $\ell_1$ inert in $\mathcal{K}$ which is prime to $p$ and all primes where $\psi$ is ramified, such that $L(\psi\chi',1)\not=0$.
\item There is a finite order anticyclotomic character $\chi''$ whose order and conductor are both a power of an odd prime $\ell_1$ inert in $\mathcal{K}$ which is prime to $p$ and all primes where $\psi$ is ramified, such that $L(\psi\chi'',s)$ has vanishing order $1$ at $s=1$.
\end{itemize}
\end{lemma}
\begin{proof}
We take $\chi'$ and $\chi''$ to be finite order characters of the anticyclotomic $\mathbb{Z}_{\ell_1}$-Galois group of $\mathcal{K}$. Note that the local root number of $\psi\chi'$ at $\ell_1$ is determined by the conductor of $\chi'$ at $\ell_1$ (whether being an odd or even power of $\ell_1$, see e.g., \cite[Proposition 3.4]{LX}). Moreover, multiplying by $\chi'$ does not change the local root numbers at primes outside $\ell_1$. We may first manage to take $\chi'$ to make the global root number correct, and then use Proposition \ref{Jianonvanishing} to ensure the vanishing orders of the corresponding $L$-functions are $0$ or $1$ respectively.
\end{proof}
\begin{remark}
The restrictions of $\chi'$ and $\chi''$ to $G_{\mathcal{K}_p}$ are the trivial characters, since they are unramified at $p$, anticyclotomic, and of odd order, and thus $\psi|_{G_{\mathcal{K}_p}}=\psi\chi'|_{G_{\mathcal{K}_p}}=\psi\chi''|_{G_{\mathcal{K}_p}}$.
\end{remark}

Take $\chi^\prime$ as in Lemma \ref{choicecharacter}  and assume 
the image of $z_{p^\infty,\psi\chi'}$ in $H^1(\mathcal{K}_p, \boldsymbol{\psi}\chi')$ do not take $0$-value inside the disc of radius $p^{-m}$ around the origin. We use $\Lambda_L^{-,\prime}$ as the coefficient ring for family of characters.

Choose another auxiliary anticyclotomic character $\chi''$ as in Lemma \ref{choicecharacter}.  The map from the rank $1$ space $H^1_f(\mathcal{K}, \psi\chi'')$ to $H^1_f(\mathcal{K}_p,\psi\chi'')$ is nonzero by  Lemma \ref{AV} below, which is proved in \cite[Theorem 2.8]{BSW} using the $p$-adic analytic subgroup theorem.  \begin{lemma}\label{AV}
Suppose $f$ is a weight $2$ modular form with trivial character whose associated $\mathrm{GL}_2$-type abelian variety is $A_f$. Here $A_f$ has dimension $[F:\mathbb{Q}]$ where $F$ is the Hecke field of $f$ (which is totally real). For any prime $\mathfrak{p}$ above $p$,  let $V_{f,\mathfrak{p}}$ be the $\mathfrak{p}$-adic Tate module of $A_f$, which is a rank $2$ Galois representation of $G_{\mathbb{Q}_p}$ over $F_{\mathfrak{p}}$. Suppose the vanishing order of $L(f,s)$ at $s=1$ is $1$. Then for each such prime $\mathfrak{p}$ above $p$, the map from the rank $1$ space $H^1_f(\mathbb{Q}, V_{f,\mathfrak{p}})$ to $H^1(\mathbb{Q}_p, V_{f,\mathfrak{p}})$ is injective.
\end{lemma}
 Note that $z_{p^\infty,\psi\chi'}$ (resp. $z_{p^\infty,\psi\chi''}$) specializes to elements in $H^1_f(\mathcal{K}, -)$ at arithmetic points in $\mathfrak{X}_2$ (resp. $\mathfrak{X}_1$) by Lemma \ref{root number} and the explicit reciprocity law for elliptic units, switching the roles played by $\iota$ and $\iota^c$.
\begin{definition}\label{mainlocaltheory}
Let $H^1_-(\mathcal{K}_p, \boldsymbol{\psi})\subset H^1(\mathcal{K}_p,\boldsymbol{\psi})$ and $H^1_+(\mathcal{K}_p,\boldsymbol{\psi})$ be the $\Lambda^{-,\prime}_L$-module generated by the image of $z_{p^\infty,\psi\chi'}$ and $z_{p^\infty,\psi\chi''}$ respectively. We sometimes just write them as $H^1_-$ and $H^1_+$ for short. Clearly $H^1_-$ (resp. $H_+^1$)  specializes to $H^1_f$ at all points  $\phi\in\mathfrak{X}_2$ (resp. $\fX_1$). 
\end{definition}

By Corollary \ref{1111}, we see that the specialization of $z_{p^\infty,\psi\chi''}$ to $\phi_0$ has nonzero image in $H^1_f(\mathcal{K}_p,\psi)$. Thus the specializations of $H_\pm^1$ to $\phi_0$ span the $2$-dimensional vector space $H^1(\mathcal{K}_p,\psi)$. Indeed one of them is the $1$-dimensional vector space $H^1(\mathcal{K}_p,\psi)/H^1_f(\mathcal{K}_p,\psi)$ (from the explicit reciprocity law and $L(\psi\chi',1)\not=0$), and the other one is $H^1_f(\mathcal{K}_p, \psi)$. Recall that $H^1(\mathcal{K}_p,\boldsymbol{\psi})$ is free of rank $2$ over $\Lambda^{-,\prime}_L$. We have the following   analogue of Rubin's conjecture in our context, which is much easier to see.
\begin{proposition}
By possibly enlarging the $m$, we can ensure that over $\Lambda^{-,\prime}_L$,
\begin{equation}\label{RMC}
H^1(\mathcal{K}_p,\boldsymbol{\psi})=H^1_+(\mathcal{K}_p,\boldsymbol{\psi})\oplus H^1_-(\mathcal{K}_p,\boldsymbol{\psi}).
\end{equation}
\end{proposition}
\begin{proof}
Indeed we take a basis of $H^1(\mathcal{K}_p,\boldsymbol{\psi})$ over $\Lambda^{-,\prime}_L$ and express $z_{p^\infty,\psi\chi''}$ and $z_{p^\infty,\psi\chi'}$ in terms of this basis. The determinant of this matrix takes nonzero value at $\phi_0$. Therefore as an analytic function, it is an invertible element in a small disc around $\phi_0$. Thus by linear algebra, $z_{p^\infty,\psi\chi''}$ and $z_{p^\infty,\psi\chi'}$ also form a basis of $H^1(\mathcal{K}_p,\boldsymbol{\psi})$.
\end{proof}

\begin{lemma}\label{characterize1}
The $H^1_\pm$ are characterized by all classes specializing to $H^1_f$ for all $\phi$ in $\mathfrak{X}_1$ and $\mathfrak{X}_2$ respectively.
\end{lemma}
\begin{proof}
Suppose $a_+\in H^1_+$ and $a_-\in H^1_-$ are such that $a_++a_-$ specializes in $H^1_f$ for all $\phi\in\mathfrak{X}_1$, then $a_-=0$ by Proposition \ref{Greenbergnonvanishing} (since $\boldsymbol{\psi}_\phi\chi'$ has global root number $+1$) and the explicit reciprocity law for elliptic units, which implies the image of $z_{p^\infty,\psi\chi'}$ has nonzero image in $H^1(\mathcal{K}_p, \boldsymbol{\psi}_\phi\chi')/H^1_f(\mathcal{K}_p, \boldsymbol{\psi}_\phi\chi')$ for these $\phi$. Thus if we write $a_-=f_1z_{p^\infty,\psi\chi'}$ with $f_1\in\Lambda^{-,\prime}_L$, then $f_1$ has to take $0$ value at a $p$-adically dense set. Thus $f_1$ is identically zero. If $a_++a_-$ specializes to an element in $H^1_f$ for all $\phi\in\mathfrak{X}_2$, then $a_+=0$ by the same argument using Proposition \ref{Greenbergnonvanishing} and the explicit reciprocity law for elliptic units, switching the roles played by $\iota$ and $\iota^c$. (Note that the global root numbers for $\boldsymbol{\psi}_\phi\chi''$ are $+1$.)
\end{proof}
\subsubsection{Local definitions over localizations of Iwasawa algebra}
Now we develop the local $\pm$-theory over the localization of the anticyclotomic Iwasawa algebra at an appropriate polynomial taking nonzero value at the origin. In this subsubsection, we assume $p\not=2$ and $\zeta_p\in L$.  Recall $\Lambda^-_L=\mathcal{O}_L[[X]]\otimes L$.

Following Pollack \cite{Pollack}, we define that
$$\log^+_p(X)=\frac{1}{p}\prod_n(\frac{\Phi_{2n}(1+X)}{p}),\quad \log^-_p(X)=\frac{1}{p}\prod_n(\frac{\Phi_{2n-1}(1+X)}{p})$$
where $\Phi_m$ is the  $p^m$-th cyclotomic polynomial.
For $1\leq l\leq p-1$, let
$$f_l(X)=\frac{1}{1-\zeta^l_p}((1+X)-\zeta^l_p).$$
Let $f_a$ (resp. $f_b$) be the product of $f_l$ with $l\in[1,p-1]$ running over quadratic residues (resp. non-residues) modulo $p$. Suppose $\mathcal{K}=\mathbb{Q}(\sqrt{-D})$ where $-D$ is the fundamental discriminant. Let
$$f_+=\prod_{2|n} f_a((1+X)^{p^n}-1)\prod_{2\nmid n}f_b((1+X)^{p^n}-1)$$
and
$$f_-=\prod_{2\nmid n} f_a((1+X)^{p^n}-1)\prod_{2| n}f_b((1+X)^{p^n}-1)$$
for $p$ ramified in $\mathcal{K}$ with $D'=\frac{D}{p}$ being a quadratic non-residue modulo $p$, and let
$$f_+=\prod_n f_a((1+X)^{p^n}-1),\quad f_-=\prod_n f_b((1+X)^{p^n}-1)$$
for $p$ ramified in $\mathcal{K}$ such that $D'$ is a quadratic residue modulo $p$.
If $p$ is inert in $\mathcal{K}$, set
$$f_+=\log^+_p, f_-=\log^-_p.$$
These are analogues of Pollack's half logarithms.
Let $\mathfrak{X}_\pm\subset \mathrm{Spec}\Lambda^-_L$ be the zeros of $f_\pm$.

The following lemma is a straightforward application of the local root number formulas. For convenience of  readers we copy the formula in \cite[Proposition 2]{RohrlichRN} here for the ramified case.
\begin{proposition}
Suppose $p$ is ramified in $\mathcal{K}$ and $\chi_p$ is the local component at $p$ of a self-dual Hecke character of $\mathcal{K}^\times\backslash\mathbb{A}^\times_\mathcal{K}$. Let $\varpi$ be a uniformizer of $\mathcal{K}^\times_p$. Then the local root number for $\chi_p$ is $\left(\frac{2}{p}\right)$ if $f=1$, and is $\left(\frac{-2l_{\chi_p}}{p}\right)\chi(\varpi^{f-1})i^\delta$ if $f>1$, where $\left(\cdot\right)$ is the Legendre symbol, $f$ is the exponent of the conductor of $\chi_p$, $l_{\chi_p}$ is an invertible residue class modulo $p$ such that
$$\chi_p(1+\varpi^{f-1})=\exp\left(\frac{2\pi il_{\chi_p}}{p}\right),$$ $\delta=0$ if $p\equiv 1(\mathrm{mod}\ 4)$, and $\delta=1$ if $p\equiv 3(\mathrm{mod}\ 4)$.
\end{proposition}

\begin{lemma}\label{rn}
For one of $\mathfrak{X}_\pm$, for all but finitely many points $\phi$ in it, we have the root numbers $r(\boldsymbol{\psi}_{\phi,p})=r(\psi_p)$, and $r(\boldsymbol{\psi}_\phi)=r(\psi)$. For the other one of $\mathfrak{X}_\pm$, for all but finitely many points $\phi$ in it, we have $r(\boldsymbol{\psi}_{\phi,p})=-r(\psi_p)$ and $r(\boldsymbol{\psi}_\phi)=-r(\psi)$.
\end{lemma}
\begin{proof}
Note that the specialization of $\Psi^-_\mathcal{K}$ at $\phi$ is unramified at all primes outside $p$. Moreover the Archimedean types of these characters are always $(-1,0)$, so the Archimedean local root number is unchanged. If $v\nmid p$ is non-split in $\mathcal{K}$, then the image of the uniformizer $\varpi_v$ in $\Gamma^-_\CK$ under the reciprocity map is the identity, since its square is $1$ and $\Gamma^-_\CK$ is torsion-free. If $v=w\bar{w}$ is split in $\mathcal{K}$, then the local root number for $\boldsymbol{\psi}_\phi$ at $v$ is $\boldsymbol{\psi}_\phi((-1)_w)$ where $(-1)_w$ is the \idele element at $w$ taking value $-1$. Again, the image of $(-1)_w$ in $\Gamma^-_\CK$ under the reciprocity map is the identity, since the square of $(-1)_w$ is $1$ and $\Gamma^-_\CK$ is torsion-free. Thus the local root numbers at these primes are unchanged throughout the family. Now the lemma follows from the above proposition at ramified primes $p$, taking the uniformizer there to be $\sqrt{-D}$. If $p$ is inert then the local formula is easier and essentially determined by the level of the conductor, e.g., given in \cite[Proposition 3.4]{LX}.
\end{proof}
\begin{definition}
We define $\mathfrak{X}'_{\pm}$  as the subset of $\mathfrak{X}_\pm$ respectively, of points satisfying the formulas for root numbers in Lemma \ref{rn}.
\end{definition}

Let $\chi'$ and $\chi''$ be as in Lemma \ref{choicecharacter}. From the above lemma we may find a basis $(v_4,v_5)$ of $H^1(\mathcal{K}_p, \psi\otimes \Lambda^-_L)$ over $\Lambda^-_L$. Let $P_{\psi,\pm}\in \Lambda^-_L$ be the determinant of $(z_{\psi\chi'}, z_{\psi\chi''})$ with respect to this basis. Note that this does depend on the choices of $\chi'$ and $\chi''$. From our choice we know that $\phi_0(P_{\psi,\pm})\not=0$. The $(z_{\psi\chi'}, z_{\psi\chi''})$ form a basis of $H^1(\mathcal{K}_p,\psi\otimes\Lambda^-_{L,P_{\psi,\pm}})$ over  $\Lambda^-_{L,P_{\psi,\pm}}$, the localization at $P_{\psi,\pm}$ of $\Lambda_L^{-}$.  We define $H^1_-(\mathcal{K}_p,\psi\otimes\Lambda^-_{L,P_{\psi,\pm}})$ and $H^1_+(\mathcal{K}_p, \psi\otimes\Lambda^-_{L,P_{\psi,\pm}})$ as the submodules generated by $z_{\psi\chi'}$ and $z_{\psi\chi''}$ respectively. The following proposition is an analogue of Rubin's conjecture and can be proved in the same way as (\ref{RMC}).

\begin{proposition}
We have
\begin{equation}\label{rubincj2}
H^1(\mathcal{K}_p, \psi\otimes\Lambda^-_{L,P_{\psi,\pm}})=H^1_-(\mathcal{K}_p, \psi\otimes\Lambda^-_{L,P_{\psi,\pm}})\oplus H^1_+(\mathcal{K}_p, \psi\otimes\Lambda^-_{L,P_{\psi,\pm}}).
\end{equation}
\end{proposition}

We assume $\mathfrak{X}_+$ is the set in the above Lemma \ref{rn} (the $\mathfrak{X}_-$ case is completely analogous).

\begin{lemma}\label{characterize2}
The module $H^1_+(\mathcal{K}_p,\psi\otimes\Lambda^-_{L,P_{\psi,\pm}})$ can be characterized as sections whose specializations to points in $\mathfrak{X}'_+$ where $P_{\psi,\pm}$ does not take $0$ value live in $H^1_f$, and the module $H^1_-(\mathcal{K}_p,\psi\otimes\Lambda^-_{L,P_{\psi,\pm}})$ can be characterized as sections whose specializations to points in $\mathfrak{X}'_-$ where $P_{\psi,\pm}$ does not take $0$ value live in $H^1_f$.
\end{lemma}
\begin{proof}
This follows easily from the definition, the explicit reciprocity law for elliptic units, Proposition \ref{Jianonvanishing} and Lemma \ref{rn}.
\end{proof}

\section{Perrin-Riou's main conjecture and $p$-converse}
Let $\psi$ be a self-dual Hecke character of Archimedean type $(-1,0)$ such that the global root number is $-1$.
To study Iwasawa theory for $p$-adic $L$-functions and Perrin-Riou's main conjecture we need to work with the ring $\Lambda^{-,\prime}_L$. Here $\Lambda^{-,\prime}:=\mathcal{O}_L[[p^{-m}X]]$ and $\Lambda^{-,\prime}_L:=\mathcal{O}_L[[p^{-m}X]]\otimes L$ for some $m\gg0$.

\subsection{Iwasawa theory for $p$-adic $L$-functions}

Let $H^1_\pm(\mathcal{K}_p, (\Lambda^{-,\prime}(\boldsymbol{\psi}))^*)$ be the orthogonal complement of
$H^1_\pm(\mathcal{K}_p, \boldsymbol{\psi})$
under the local Tate pairing. We define various Iwasawa modules over $\Lambda^{-,\prime}$.
\begin{definition}\label{defineSelmer1}
Let
$$H^1_\pm(\mathcal{K}^S, \boldsymbol{\psi}):=\mathrm{ker}\{H^1(\mathcal{K}^S, \boldsymbol{\psi})\rightarrow \frac{H^1(\mathcal{K}_p, \boldsymbol{\psi})}{H^1_\pm(\mathcal{K}_p, \boldsymbol{\psi})}\}.$$
Let
\[\mathrm{Sel}^\pm(\mathcal{K}^S, \boldsymbol{\psi}^*):=\mathrm{ker}\{H^1(\mathcal{K}^S, \boldsymbol{\psi}^*)\rightarrow \prod_{v\nmid p} H^1(\mathcal{K}_v, \boldsymbol{\psi}^*)\times \frac{H^1(\mathcal{K}_p, \boldsymbol{\psi}^*)}{H^1_\pm(\mathcal{K}_p, \boldsymbol{\psi}^*)}\}.\]
Put (again $*$ means Pontryagin dual)
$$X^{\pm}_\psi=\mathrm{Sel}^{\pm}(\mathcal{K}^S, \boldsymbol{\psi}^*)^*.$$
We also define the strict and relaxed Selmer modules as Definition \ref{defineSelmer} except replacing the coefficient ring by $\Lambda^{-,\prime}_L$.
\end{definition}
Composing with the conjugation isomorphism $\circ: \gamma\mapsto \gamma^{-1}$ on the weight space, the anticyclotomic $p$-adic $L$-function in the Eisenstein case  gives  another $p$-adic $L$-function which we denote as $\mathcal{L}^{\bullet}_{\psi}$ interpolating the algebraic part of special $L$-values $L(\boldsymbol{\psi}_\phi,1)=L(\boldsymbol{\psi}^c_\phi,1)$ for $\phi$ in $\mathfrak{X}_2$.  Note that the image of $\mathfrak{X}_2$ under $\circ$ is inside $\mathfrak{X}_1$ and the global root number for these $\boldsymbol{\psi}_\phi$ are $+1$. The following question arises naturally: is there a relation between $\phi_0(\mathcal{L}^{\bullet}_{\psi})$ and the $p$-adic logarithm $\log_p$ of the specialization at $\phi_0$ of $z_{p^\infty,\psi}$ (due to root number reason, it does lie in $H^1_f$)?

We start with the following proposition about factorization of $p$-adic $L$-functions. Recall that $f$ is the eigenform associated to the Hecke character $\xi$, see Subsubsection \ref{padicL}.
\begin{proposition}\label{factorization}
 There is a nonzero constant $C_{\xi,\chi}$ and a unit element $\mathcal{U}_{\xi,\chi}$ in $\Lambda_L^-$ such that
$$\mathcal{L}^2_{f,\chi}=C_{\xi,\chi}\mathcal{U}_{\xi,\chi}\mathcal{L}^{\bullet}_{\xi\chi}
\mathcal{L}^\circ_{\xi\chi^c}.$$
Here the superscript $\circ$ means the conjugation $\circ:\gamma\mapsto \gamma^{-1}$ on the weight space.
\end{proposition}
\begin{proof}
The proposition follows from comparing the interpolation formulas for the various $p$-adic $L$-functions. Indeed note that we take $\phi\in\mathfrak{X}_4$. Then the $L$-function on the left hand side splits up into $2$ $L$-functions for the characters $\xi\boldsymbol{\chi}_\phi$ and $\xi\boldsymbol{\chi}^c_\phi$. These are precisely the $L$-values interpolated by $\mathcal{L}^{\bullet}_{\xi\chi}$ and $\mathcal{L}^\circ_{\xi\chi^c}$.   It is easy to see that we can take the same twisting element $g_{\ell,n}$ for all of them, and thus the $p$-adic periods $\Omega_p$ and Archimedean periods $\Omega_\infty$ involved are the same. We consider the ratio $\frac{\mathcal{L}^2_{f,\chi}}{\mathcal{L}^{\bullet}_{\xi\chi}\mathcal{L}^\circ_{\xi\chi^c}}$ and use our observations right after Proposition \ref{Int} on the local toric integrals to finish the proof.
\end{proof}
\begin{lemma}
The $\mathcal{L}^{\bullet}_\psi$ is not identically zero.
\end{lemma}
\begin{proof}This follows from Proposition \ref{Greenbergnonvanishing} and the interpolation formula for the $p$-adic $L$-function.
\end{proof}

The Iwasawa main conjecture for $\mathcal{L}^{\bullet}_{\psi}$ below is a corollary of Theorem \ref{RMCC}.
\begin{corollary}\label{AIMC}
Assumptions are as in Theorem \ref{RMCC}. Then by possibly shrinking the weight space, we have the Iwasawa main conjecture
$$\mathrm{char}_{\Lambda^{-,\prime}_L}(X^-_\psi)=(\mathcal{L}^{\bullet}_{\psi}).$$
\end{corollary}

\begin{proof}
As in the proof of Lemma \ref{choicecharacter}, we can find a finite order anticyclotomic CM character $\eta$ such that $\eta$ restricts to the trivial character of $G_{\mathcal{K}_p}$, and the vanishing order of $L(\psi\eta, s)$ at $s=1$ is $1$, the vanishing order of $L(\xi^c\chi\eta,s)=L(\psi'\eta, s)$ is $0$ at $s=1$, and that the localization map $H^1_f(\mathcal{K}, \psi\eta)\rightarrow H^1_f(\mathcal{K}_p,\psi\eta)$ is injective. We first prove the following.
\begin{lemma}\label{auxnonzero}
The specialization of $\mathcal{L}^{\bullet}_{\psi\eta}$ at $\phi_0$ is nonzero.
\end{lemma}
\begin{proof}
By our assumption on the vanishing orders of $L$-functions of $\psi\eta$ and $\psi'\eta$, the image of the Heegner point associated to $(f,\chi\eta)$ in $H^1_f(\mathcal{K},\psi\eta\oplus\psi'\eta)$ is nonzero. But $H^1_f(\mathcal{K}, \psi'\eta)$ is trivial since $L(\psi'\eta,1)\not=0$, so the image in $H^1_f(\mathcal{K},\psi\eta)$ is nonzero. Therefore by Lemma \ref{AV} its image in $H^1_f(\mathcal{K}_p,\psi\eta)$ is also nonzero. We see moreover that $H^1_f(\mathcal{K}_p, V_f\otimes\chi\eta)\simeq H^1_f(\mathcal{K}_p, \psi\eta\oplus\psi'\eta)$ is mapped isomorphically to
\begin{equation}\label{deRham}
\mathrm{Fil}^0D_{\mathrm{dR},\mathcal{K}_p}(\psi\eta\oplus\psi'\eta)\simeq\mathcal{K}_p\otimes_{\mathbb{Q}_p}
 L\simeq L\oplus L
\end{equation}
under the $p$-adic logarithm map, where the $\mathcal{K}_p$ acts on the first component by identity and acts by complex conjugation on the second.
By the $p$-adic Waldspurger formula and Proposition \ref{factorization} on the factorization of the $p$-adic $L$-functions, the first component is given by
$\phi_0(\mathcal{L}^{\bullet}_{\psi\eta})  L(\xi^c\cdot\chi\eta,1)\omega_f^\vee$.

Note that the definition of this $p$-adic $L$-function depends on our chosen  embedding $\iota:\bar{\mathcal{K}}\hookrightarrow \mathbb{C}_p$. In the following, we write $\mathcal{L}^\iota_{f,\chi\eta}$  to indicate this dependence. Consider another $p$-adic $L$-function $\mathcal{L}^{\iota^c}_{f,\chi\eta}$ with respect to the $\tilde{c}$-conjugation of $(f,\chi/\CK)$ via the embedding $\iota^c$ (the $p$-adic object is the same as that of $(f,\chi/\CK)$ via the embedding $\iota$). Its interpolation formulas are similar to the $\iota$ case except the interpolation points are for $\phi\in\mathfrak{X}_5$ (while the interpolation points are $ \phi\in \mathfrak{X}_4$ for $\mathcal{L}^\iota_{f,\chi}$).

By Proposition \ref{factorization} and Lemma \ref{root number}, we see that  $\mathcal{L}^{\iota^c}_{f,\chi\eta}$ splits into $2$ CM $p$-adic $L$-functions whose global root numbers are both $-1$ (while the global root numbers for the $2$ factors of $\mathcal{L}^\iota_{f,\chi\eta}$ are both $+1$). Thus $\mathcal{L}^{\iota^c}_{f,\chi\eta}$ is identically zero. By the $p$-adic Waldspurger formula, the second component of the image of the Heegner point for $(f,\chi\eta)$ in (\ref{deRham}) is $0$. Consequently, the first component must be nonzero. This proves $\phi_0(\mathcal{L}^{\bullet}_{\psi\eta})\not=0$.
\end{proof}
\begin{remark}
In fact the above picture is seen more clearly in the simple case when $p$ splits in $\mathcal{K}$, where the $2$ components in (\ref{deRham}) correspond to the $2$ primes $v_0$ and $\bar{v}_0$ above $p$. The second component in (\ref{deRham}) of the image of the Heegner point associated to $(f,\chi\eta)$ is its image in $H^1_f(\mathcal{K}_p, \psi'\eta)$ which must be $0$ since the rank of the Selmer group of $\psi'\eta$ is $0$.
\end{remark}
We return to the proof of the corollary. Consider $w:= z_{p^\infty,\psi\eta}/\mathcal{L}^{\bullet}_{\psi\eta}$. As before, by shrinking the weight space, we may avoid the zero locus of $\mathcal{L}^{\bullet}_{\psi\eta}$ and assume that $w$ is a basis of $H^1_+(\mathcal{K}_p,\boldsymbol{\psi})$ over $\Lambda^{-,\prime}_L$ (note that $\psi$ and $\psi\eta$ restrict to the same character of $\mathcal{K}_p$). Both $z_{p^\infty,\psi}$ and $w$ lie in the $1$-dimensional space $H^1_+(\mathcal{K}_p,\boldsymbol{\psi})$ over $\Lambda^{-,\prime}_L$. Hence we may write $z_{p^\infty,\psi}=f_2\cdot w$ for some $f_2\in\Lambda^{-,\prime}_L$. By comparing the explicit reciprocity law for elliptic units for $\psi$ and $\psi\eta$ as recalled above at the dense set of arithmetic points in $\mathfrak{X}_2$, we see that as elements in $ H^1_+(\mathcal{K}_p,\boldsymbol{\psi})$,
$$z_{p^\infty,\psi}= \mathcal{B}\cdot\mathcal{L}^{\bullet}_{\psi}\cdot w$$
where $\mathcal{B}$ is a nonzero constant times a unit in $\Lambda^{-,\prime}$ coming from local factors in the $p$-adic $L$-functions at primes outside $p$ (see the remarks after Proposition \ref{Int}, noting that the local data at $p$ are the same for $\psi$ and $\psi\eta$ and thus the local factors at $p$ and the $p$-adic and Archimedean periods are the same).
Now the corollary follows from Theorem \ref{RMCC}, (\ref{RMC})
over $\Lambda^{-,\prime}_L$, and the Poitou-Tate exact sequence
$$0 \rightarrow H^1(\mathcal{K}^S, \boldsymbol{\psi})\rightarrow \frac{H^1(\mathcal{K}_p, \boldsymbol{\psi})}{H^1_-(\mathcal{K}_p, \boldsymbol{\psi})}\rightarrow X^-_{\psi}\rightarrow X^{\mathrm{str}}_{\psi}\rightarrow 0.$$
\end{proof}
\begin{remark}
When $p$ is split in $\mathcal{K}$, one has a good understanding of a local regulator map as in \cite[Theorem 4.15]{LZ1}, which is not available in the non-split case here. Thus here instead we use the local basis associated to an auxiliary character $\psi\eta$ to replace the role played by the regulator map in the argument. One may compare with the argument in Section \ref{split case}.
\end{remark}

\subsection{Selmer groups and global duality argument}
In this subsection, we prove an analogue of Perrin-Riou's conjecture for Heegner family and deduce the $p$-converse theorem (similar to \cite{CW}).

We have seen from Corollary \ref{4.3} and Lemma \ref{freeness} that $H^1_{\mathrm{str}}(\mathcal{K}^S, \boldsymbol{\psi})=0$, and $H^1(\mathcal{K}^S, \boldsymbol{\psi})$ is free of rank one over $\Lambda^{-,\prime}_L$ (note that this ring is a discrete valuation ring). From Lemma \ref{localHeegner} below we see that $H^1(\mathcal{K}^S,\boldsymbol{\psi})$ is mapped to $H^1_+(\mathcal{K}_p, \boldsymbol{\psi})$, and
\begin{equation}\label{rel}
H^1_+(\mathcal{K}^S, \boldsymbol{\psi})=H^1(\mathcal{K}^S, \boldsymbol{\psi}).
\end{equation} Thus also $H^1_-(\mathcal{K}^S, \boldsymbol{\psi})=0$.
We have by the Poitou-Tate exact sequence the following:
\begin{equation}\label{(1)}
0\rightarrow H^1_+(\mathcal{K}^S, \boldsymbol{\psi})\rightarrow H^1_+(\mathcal{K}_p, \boldsymbol{\psi})\rightarrow X^{\mathrm{rel}}_{\psi}\rightarrow X^+_\psi\rightarrow 0\end{equation}
and
\begin{equation}\label{(2)}
0\rightarrow H^1(\mathcal{K}^S, \boldsymbol{\psi})\rightarrow \frac{H^1(\mathcal{K}_p, \boldsymbol{\psi})}{H^1_-(\mathcal{K}_p, \boldsymbol{\psi})}\rightarrow X^-_{\psi}\rightarrow X^{\mathrm{str}}_{\psi}\rightarrow 0.
\end{equation}
Let $P$ be the height one prime $(X)$. We focus on the lengths of various Iwasawa modules at $P$ since this is related to the $p$-converse theorem.

The argument in \cite[Lemma 6.7]{CW} shows that for any height one prime $P$ of $\Lambda^{-,\prime}_L$, we have
\begin{equation}\label{(3)}\mathrm{leng}_P(X^{\mathrm{rel}}_{\psi,\mathrm{tor}})=\mathrm{leng}_P(X^{\mathrm{str}}_\psi).
\end{equation}
Here the subscript $\mathrm{tor}$ means the $\Lambda^{-,\prime}_L$-torsion part. Consider the $\Lambda^{-,\prime}_L$-localization map
\begin{equation}\label{locind}
H^1(\mathcal{K}^S,\boldsymbol{\psi})\rightarrow H^1_+(\mathcal{K}_p, \boldsymbol{\psi}).
\end{equation}
Write $\mathrm{locind}_P$ for the length at $P$ of the cokernel of map (\ref{locind}).
Using (\ref{(1)}), (\ref{(2)}) and (\ref{(3)}) we can see that
\begin{equation}\label{exactsequence}
\mathrm{leng}_PX^+_{\psi,\mathrm{tor}}+2\mathrm{locind}_P=\mathrm{leng}_PX^-_{\psi}.
\end{equation}
(Note that from Lemma \ref{global}, Corollary \ref{4.3} and Poitou-Tate exact sequence in the proof there, the rank of $X^{\mathrm{rel}}_\psi$ over $\Lambda^{-,\prime}_L$ is $1$, and the rank of $X^+_\psi$ is also $1$.)

On the other hand we have proved the main conjecture that
\begin{equation}\label{imc}
\mathrm{char}_{\Lambda^{-,\prime}_L}(X^-_{\psi})=(\mathcal{L}^{\bullet}_{\psi}).
\end{equation}
\begin{definition}
We define a virtual Heegner family to be an element $\kappa^{\mathrm{virt}}=\kappa^{\mathrm{virt}}_{f,\boldsymbol{\chi}}$ in $H^1_+(\mathcal{K}^S,\boldsymbol{\psi})$ whose image in $H^1(\mathcal{K}_p,\boldsymbol{\psi})$ satisfies that for the height one prime $P=(X)$,
\begin{equation}\label{vhf}
2\mathrm{leng}_P(\frac{H^1_+(\mathcal{K}_p,\boldsymbol{\psi})}
{(\Lambda^{-,\prime}_L)\kappa^{\mathrm{virt}}})= \mathrm{ord}_P(\mathcal{L}^{\bullet}_{\xi\chi}\cdot\mathcal{L}^\circ_{\xi\chi^c}),
\end{equation}
and the specialization of  $\kappa^{\mathrm{virt}}$ to $\phi_0$ is the Heegner point $\kappa_{f,\chi}$ in Appendix \ref{Heegner cycles in family} associated to the pair $(f,\chi/\CK)$.
\end{definition}
The existence of a virtual Heegner family will be proved in Proposition \ref{isvirtualHeegner}.

Recall that $\xi\chi=\psi$, and $L(\xi\chi^c, 1)\not=0$ by our choice.
\begin{theorem}\label{main proposition}
If the Selmer rank of $\psi$ is $1$, then the specialization $\kappa_{f,\chi}$ of the virtual Heegner family to $\phi_0$ is non-torsion, and thus the vanishing order of $L(\psi, s)$ at $s=1$ is exactly $1$.
\end{theorem}
\begin{proof}
For the height one prime $P:=(X)$,  $$\mathrm{ord}_P(\mathrm{char}(\frac{H^1_+(\mathcal{K}_p,\boldsymbol{\psi})}{(\Lambda^{-,\prime}_L)
\kappa^{\mathrm{virt}}}))
=\mathrm{locind}_P+\mathrm{ord}_P(\mathrm{char}(\frac{H^1_+(\mathcal{K}^S,\boldsymbol{\psi})}
{(\Lambda^{-,\prime}_L)\kappa^{\mathrm{virt}}})),$$
so by (\ref{vhf}), (\ref{imc}) and (\ref{exactsequence}),
\begin{align}\label{PRCJ}
&2(\mathrm{leng}_P(\frac{H^1_+(\mathcal{K}^S,\boldsymbol{\psi})}
{(\Lambda^{-,\prime}_L)\kappa^{\mathrm{virt}}}))&&=
\mathrm{leng}_P(X^{+}_{\psi,\mathrm{tor}})+\mathrm{ord}_P
(\mathcal{L}^\circ_{\xi\chi^c})&\\&&&
=\mathrm{leng}_P(X^+_{\psi,\mathrm{tor}})
.&
\end{align}
Here for the last identity, we use  $L(\xi\chi^c,1)\not=0$. This equation (\ref{PRCJ}) can  be viewed as an analogue of Perrin-Riou's main conjecture for Heegner family in our setting at $P$.

We note that the specialization of the Selmer condition $H^1_+(\mathcal{K}_p, \boldsymbol{\psi})$ to $\phi_0$ is nothing but the $H^1_f(\mathcal{K}_p,\psi)$. Thus by self duality of $H^1_f(\mathcal{K}_p,\psi)$ under the  local Tate duality, we have $$H^1_+(\mathcal{K}_p, \Lambda^{-,\prime}_L(\boldsymbol{\psi})^*)^{\gamma^-=1}=H^1_f(\mathcal{K}_p,\psi).$$ Now under the assumption that the Selmer rank of $\psi$ is $1$, we see that $X$ does not divide $\mathrm{char}_{\Lambda^{-,\prime}_L}(X^{+}_{\psi,\mathrm{tor}})$. Thus the specialization $\kappa^{\mathrm{virt}}$ of the virtual Heegner family to $\phi_0$ must be non-torsion.

But the specialization of the virtual Heegner family to $\phi_0$ is by definition the Heegner point associated to the pair $(f,\chi/\CK)$. It being non-torsion, together with the Gross--Zagier and Kolyvagin theorem imply the last sentence of the theorem.
\end{proof}

\subsection{Construction of the Heegner family}\label{VHC}
In this subsection, we write $\kappa_{\mathcal{F},\chi}$ and $\kappa_{f,\chi}$ for the family of Heegner cycles associated to the pair $(\mathcal{F},\chi/\CK)$ or $(f,\chi/\CK)$ respectively in Appendix \ref{Heegner cycles in family}. Note that as remarked in the appendix, the test vectors at $p$ are different from the test vectors for the explicit reciprocity law, and the resulting $p$-adic Waldspurgers differ by a nonzero constant $C$  in Remark \ref{constantC} depending only on the local data at $p$.
\subsubsection{Over Small Discs}
Now we construct the virtual Heegner family. In Appendix \ref{Heegner cycles in family}, we have a Coleman family $\mathcal{F}$ over a domain, whose specialization at the arithmetic point $\phi'_0$ is $f$.
\begin{lemma}\label{localHeegner}
The image of $H^1(\mathcal{K}^S, \boldsymbol{\psi})$ in $H^1(\mathcal{K}_p,\boldsymbol{\psi})$ lies in $H^1_+(\mathcal{K}_p,\boldsymbol{\psi})$.
\end{lemma}
\begin{proof}
By the explicit reciprocity law, $z_\psi$ lies in $H^1_+(\mathcal{K}_p,\boldsymbol{\psi})$. Recall we have shown that $H^1(\mathcal{K}^S,\boldsymbol{\psi})$ is a free module of rank $1$. We take $v$ as a basis. Then $z_\psi$ is $v$ multiplied by a nonzero element of $\Lambda^{-,\prime}_L$. Thus $v$ itself is also in $H^1_+(\mathcal{K}_p,\boldsymbol{\psi})$ since $H^1(\mathcal{K}_p,\boldsymbol{\psi})=H^1_+(\mathcal{K}_p,\boldsymbol{\psi})\oplus H^1_-(\mathcal{K}_p,\boldsymbol{\psi})$.
\end{proof}
\begin{remark}\label{RKHeegner}
Below we will consider the image of the Heegner family $\kappa_{\mathcal{F},\boldsymbol{\chi}}$ associated to the pair $(\mathcal{F},\boldsymbol{\chi})$ in $H^1(\mathcal{K}_p,\mathcal{V}\otimes\boldsymbol{\chi})$ constructed in Appendix \ref{Heegner cycles in family} where $\mathcal{V}$ is the dual Galois representation associated to the Coleman family $\mathcal{F}$. In particular we will compare different families of characters $\boldsymbol{\chi}$ and $\boldsymbol{\chi}\boldsymbol{\eta}$ (where $\boldsymbol{\eta}$ is some auxiliary family of characters) restricting to the same character on $G_{\mathcal{K}_p}$, where we actually choose different test vectors $\mathcal{F}_1$ and $\mathcal{F}_2$ (at primes not dividing $p$) in the Coleman family. The convention in Appendix B.3 explicitly means the following: at an arithmetic point $\phi$, under the identification of $\mathcal{V}_\phi$ with $\mathcal{V}_{\mathcal{F}^\vee_{1,\phi}}$ and $\mathcal{V}_{\mathcal{F}^\vee_{2,\phi}}$ there, and thus of $H^1_f(\mathcal{K}_p, \mathcal{V}_\phi\otimes\boldsymbol{\chi}_\phi)$ with $H^1_f(\mathcal{K}_p, \mathcal{V}_{\mathcal{F}^\vee_{1,\phi}}\otimes\boldsymbol{\chi}_\phi)$ and $H^1_f(\mathcal{K}_p, \mathcal{V}_{\mathcal{F}^\vee_{2,\phi}}\otimes(\boldsymbol{\chi}\boldsymbol{\eta})_\phi)$, the maps $(\log_p(-), \mathcal{F}^\vee_{1,\phi})$ (taking $p$-adic logarithm and pairing with $\mathcal{F}^\vee_{1,\phi}$) and $(\log_p(-),\mathcal{F}^\vee_{2,\phi})$ coincide. This is easily seen from the construction and the functoriality of $p$-adic logarithm map.
\end{remark}

\begin{proposition}\label{isvirtualHeegner}
After possibly enlarging $m$ and $m'$, the specialization of $\kappa_{\mathcal{F},\boldsymbol{\chi}}$ to $f$ is a virtual Heegner family.
\end{proposition}
\begin{proof}
We show that the image in $H^1(\mathcal{K}_p,\boldsymbol{\psi})$ of the specialization $\kappa_{f,\boldsymbol{\chi}}$ of the Heegner family $\kappa_{\mathcal{F},\boldsymbol{\chi}}$ to $f$ does give us a virtual Heegner family.

Consider
$$\mathrm{Fil}^0D_{\mathrm{dR},\mathcal{K}_p}(V_{\mathcal{F}_\phi}\otimes\boldsymbol{\chi}_\phi)\simeq \mathcal{K}_p\otimes_{\mathbb{Q}_p} L\simeq L\oplus L,$$
where $V_{\mathcal{F}_\phi}$ is the Galois representation associated to $\mathcal{F}_\phi$ twisted by $(-\frac{k_\phi-2}{2})$, $\mathcal{K}_p$ acts on the first factor by identity, and on the second by complex conjugation.  Recall that the definition of the $p$-adic $L$-function depends on the choice of an isomorphism $\iota: \bar{\mathbb{Q}}\hookrightarrow \mathbb{C}_p$. Thus from now on in this proof we write $\mathcal{L}_{\mathcal{F},\chi}^\iota$ to indicate this dependence.

We consider the variable $p$-adic $L$-function $\mathcal{L}^{\iota^c}_{\mathcal{F},\chi}$ which specializes to $\CL_{f,\chi}^{\iota^c}$. We claim that $\mathcal{L}^{\iota^c}_{\mathcal{F},\chi}$ is divisible by the variable $T$. Indeed we specialize $\mathcal{F}$ to $f$ and the Rankin-Selberg $L$-value splits up into two CM $L$-values and the claim follows from the global root number consideration as in the last paragraph of the proof of Lemma \ref{auxnonzero}.

Recall $\psi'=\xi^c\chi$. We can apply Proposition \ref{Jianonvanishing} to find a finite order anticyclotomic character $\eta'$ of conductor and order a power of an odd prime $q'$ inert in $\mathcal{K}$ where $\xi$ and $\chi$ are both unramified, so that $L(\psi\eta',1)\not=0$ and $L(\psi'\eta^{\prime}, s)$ has vanishing order $1$ at $s=1$.  This time we have $\mathcal{L}^\iota_{\mathcal{F},\chi\eta^\prime}$ is divisible by the variable $T$. Keep in mind that the restriction of $\eta'$ to $G_{\mathcal{K}_p}$ is trivial, and the quaternion algebra associated to $(f,\chi)$ and $(f,\chi\eta')$ are the same from our construction.

Using again Proposition \ref{Jianonvanishing} and the argument in our main local definition, we can find a finite order anticyclotomic character $\eta$ of conductor and order a power of $q'$ , such that $L(\psi\eta,s)$ has vanishing order $1$ at $s=1$ and $L(\psi'\eta,1)\not=0$. Again from the construction we see the quaternion algebra associated to the pairs $(f,\chi\eta)$ and $(f,\chi)$ are the same. Then as above we have $\mathcal{L}^{\iota^c}_{\mathcal{F},\chi\eta}$ is divisible by the variable $T$.

Using Lemma \ref{auxnonzero} and Proposition \ref{factorization}, we see that $\phi_0(\mathcal{L}^\iota_{\mathcal{F},\chi\eta})\not=0$ and $\phi_0(\mathcal{L}^{\iota^c}_{\mathcal{F},\chi\eta'})\not=0$. Thus we can (by possibly enlarging $m$ and $m'$) ensure that $\mathcal{L}^{\iota^c}_{\mathcal{F},\chi\eta'}$  and $\mathcal{L}^{\iota}_{\mathcal{F},\chi\eta}$ are invertible elements in $\mathcal{O}_L[[p^{-m'}T, p^{-m}X]]\otimes L$.
Define $A$ and  $B$ in $\mathcal{O}_L[[p^{-m'}T, p^{-m}X]]\otimes L$ such that
$$A=\frac{\mathcal{L}_{\mathcal{F},\chi\eta}^{\iota^c}}
{\mathcal{L}^{\iota^c}_{\mathcal{F},\chi\eta'}},
B=\frac{\mathcal{L}^\iota_{\mathcal{F},\chi\eta'}}{\mathcal{L}^{\iota}_{\mathcal{F},\chi\eta}}.$$

Note that the $A$ and $B$ are all divisible by the variable $T$. By further enlarging $m'$, we may ensure that the determinant of the block matrix
$$\begin{pmatrix}1&A\\B&1\end{pmatrix}$$
is an invertible element in $\mathcal{O}_L[[p^{-m'}T, p^{-m}X]]\otimes L$. Thus its inverse matrix has all entries in $\mathcal{O}_L[[p^{-m'}T, p^{-m}X]]\otimes L$. We consider the expression
$$(C_1, C_2)\begin{pmatrix}1&A\\B&1\end{pmatrix}^{-1}
(\kappa_{\mathcal{F},\boldsymbol{\chi\eta}}, \kappa_{\mathcal{F},\boldsymbol{\chi\eta'}})^t$$
where the $C_i\in\mathcal{O}_L[[p^{-m'}T, p^{-m}X]]\otimes L$ are such that
$$C_1=\frac{\mathcal{L}^\iota_{\mathcal{F},\chi}}{\mathcal{L}^{\iota}_{\mathcal{F},\chi\eta}},
C_2=\frac{\mathcal{L}^{\iota^c}_{\mathcal{F},\chi}}{\mathcal{L}^{\iota^c}_{\mathcal{F},\chi\eta'}}$$
and the superscript $t$ means transpose. Note that the constant $C$ appearing in the $p$-adic Waldspurger formula in Remark \ref{constantC} only depends on the local data at $p$, which must be the same for the pairs $(\mathcal{F},\chi)$, $(\mathcal{F},\chi\eta)$ and $(\mathcal{F},\chi\eta')$ at every arithmetic point $\phi$. Then by considering the images at arithmetic points in $\phi\in\mathfrak{X}_3$ using the $p$-adic Waldspurger formula we proved and linear algebra, we see that this expression coincides with $\kappa_{\mathcal{F},\boldsymbol{\chi}}$ as local Galois cohomology classes at $p$ since they have the same images in $\mathrm{Fil}^0 D_{\mathrm{dR},\mathcal{K}_p}(V_{\mathcal{F}_\phi}\otimes\boldsymbol{\chi}_\phi)\simeq L\oplus L$ under the $p$-adic $\log$ map. Indeed by our $p$-adic Waldspurger formula for the pairs $(\mathcal{F},\boldsymbol{\chi}_{\eta})$ and $(\mathcal{F},\boldsymbol{\chi}_{\eta'})$ at an arithmetic point $\phi\in\mathfrak{X}_3$, we have
$$\begin{pmatrix}1&A\\B&1\end{pmatrix}^{-1}(\kappa_{\mathcal{F},\boldsymbol{\chi\eta}}, \kappa_{\mathcal{F},\boldsymbol{\chi\eta'}})^t$$
has image in the first component $L$ being $(C^\iota\cdot\mathcal{L}^\iota_{\mathcal{F},\boldsymbol{\chi\eta}}(\omega_{\mathcal{F}_\phi}\wedge\omega_A^{j_\phi}\wedge
\eta_A^{r_\phi-j_\phi})^\vee, 0)^t$ where $r_\phi$ and $j_\phi$ are the $r$ and $j$ in our $p$-adic Waldspurger formula at $\phi$, $\vee$ means dual, and the $C^\iota$ is the constant $C$ in Remark \ref{constantC} which depends only on the local data at $p$ (note that $\chi$, $\chi\eta$ and $\chi\eta'$ restrict to the same local character at $p$ and  that we are justified to simply write $\omega_{\mathcal{F}_\phi}$ here, thanks to Remark \ref{RKHeegner}); it (similarly) has image in the second component $L$ being
$(0, C^{\iota^c}\cdot\mathcal{L}^{\iota^c}_{\mathcal{F},\boldsymbol{\chi\eta'}}
(\omega_{\mathcal{F}_\phi}\wedge\omega_A^{j_\phi}\wedge\eta_A^{r_\phi-j_\phi})^\vee)^t$. Multiplying by $(C_1, C_2)$ we get the desired identity at $\phi$.
So by density of arithmetic points $\phi$ in $\mathfrak{X}_3$, we have locally at $p$,
$$\kappa_{\mathcal{F},\boldsymbol{\chi}}=(C_1, C_2)\begin{pmatrix}1&A\\B&1\end{pmatrix}^{-1}
(\kappa_{\mathcal{F},\boldsymbol{\chi\eta}}, \kappa_{\mathcal{F},\boldsymbol{\chi\eta'}})^t$$
Note that the two components in $L\oplus L$ are given by specializations at $\phi$ of the $\iota$ and $\iota^c$ $p$-adic $L$-functions $\mathcal{L}^\iota_{\mathcal{F},\chi}$ and $\mathcal{L}^{\iota^c}_{\mathcal{F},\chi}$ respectively (and similarly for the objects involving $\chi\eta$ and $\chi\eta'$).

Now we specialize $\mathcal{F}$ to $f$. Then as $C_2$ is divisible by $T$, the specialization of $C_2$ is identically $0$. The specializations of $C_1$ is divisible by $X^t$ where $t=\mathrm{ord}_X \mathcal{L}^\iota_{f,\chi}$. We consider the specialization with $T\mapsto 0$ and take
$$\kappa'_{f,\boldsymbol{\chi}}:=(\frac{1}{X^t}C_1, 0)\begin{pmatrix}1&A\\B&1\end{pmatrix}^{-1}
(\kappa_{\mathcal{F},\boldsymbol{\chi\eta}}, \kappa_{\mathcal{F},\boldsymbol{\chi\eta'}})^t.$$
Then clearly
\begin{equation}\label{vhfrelation}
X^t\kappa'_{f,\boldsymbol{\chi}}=\kappa_{f,\boldsymbol{\chi}}
\end{equation}
in $H^1(\mathcal{K}_p, V_f\otimes\boldsymbol{\chi})$. Moreover it is easy to check that the specialization of $\kappa'_{f,\boldsymbol{\chi}}$ to $X\mapsto 0$ has image equal to $(C^\iota\frac{\mathcal{L}^\iota_{f,\chi}}{X^t}(0)\omega_f^\vee, 0)\in \mathrm{Fil}^0 D_{\mathrm{dR},\mathcal{K}_p}(V_f\otimes\chi)\simeq L\oplus L$, which is not zero. However since
$$\kappa'_{f,\boldsymbol{\chi}}=(\frac{1}{X^t}C_1)\kappa_{f,\boldsymbol{\chi\eta}}$$
whose specialization to $X\mapsto 0$ has its image being $0$ in $H^1_f(\mathcal{K}_p, \psi'\eta)$ (recall $V_f\otimes \chi\simeq \psi\oplus\psi'$ as representations of $G_\mathcal{K}$) since the Selmer group of $\psi'\eta$ has rank $0$. We see that the image of $\kappa'_{f,\chi}$ in $H^1_f(\mathcal{K}_p,\psi)$ must be nonzero. Thus combining with Lemma \ref{localHeegner}, the image of $\kappa'_{f,\boldsymbol{\chi}}$ is a basis of $H^1_+(\mathcal{K}_p, \boldsymbol{\psi})$ upon further enlarging $m$. Therefore we have shown that $\kappa_{f,\boldsymbol{\chi}}$ is indeed a virtual Heegner family. Indeed, we have seen that $\mathrm{ord}_{(X)}\mathcal{L}^{\bullet}_{\xi\chi}=t$ and $\mathrm{ord}_{(X)}\mathcal{L}^\circ_{\xi,\chi^c}=0$, thus the property follows from (\ref{vhfrelation}) and Proposition \ref{factorization}.

\end{proof}

\subsubsection{Over Localizations of Iwasawa Algebras}
Let $\mathcal{A}$ be the ring of rigid analytic functions on the open unit disc. Following Pollack \cite{Pollack}, for two elements $F$ and $G$ in $\mathcal{A}$, we say $F$ is $O(G)$ if $\mathrm{sup}_{|z|<r}|F(z)|_p$ is $O(\mathrm{sup}_{|z|<r}|G(z)|_p)$ as $r\rightarrow 1^-$.
We write $F\sim G$ if $F$ is $O(G)$ and $G$ is $O(F)$. Pollack proved that $\log^+_p\sim\log^-_p$. If $F$ is $O(\log^+_p(X))$ then we say $F\in\mathcal{H}_{\frac{1}{2}}$. In the following, assume $p\not=2$.

The following lemma can be proved in the same way as \cite[Sections 4.1, 4.2, 5.1]{Pollack}
\begin{lemma}\label{division}
The  $f_\pm$ are convergent in $\mathcal{A}$, and  $f_\pm\sim \log_p^\pm$. Moreover if $G$ is some element in $\mathcal{H}_{\frac{1}{2}}$ such that $G=hf_+$ (or $G=hf_-$) for some $h\in\mathcal{A}$. Then $h\in\mathcal{O}_L[[X]]\otimes L$.
\end{lemma}

Let $\chi$ and $\chi'$ be as in the previous subsection.
Note that by the norm relation for Heegner points associated to the stabilization $(f_\alpha,\chi)$, we obtain as in \cite[Theorem 6.8.4]{LLZ} a Heegner point family in $H^1(\mathcal{K}^S, \boldsymbol{\psi})\otimes_{\Lambda^-_L}\mathcal{H}_{\frac{1}{2}}$, since the slope of $f$ has $p$-adic valuation $\frac{1}{2}$. By the same argument as in Corollary \ref{4.3} and Lemma \ref{freeness}, we can show that $H^1(\mathcal{K}^S,\boldsymbol{\psi})$ is free of rank $1$. We write $v_1$ for a generator of the free rank-$1$ $\Lambda^-_L$-module $H^1(\mathcal{K}^S,\boldsymbol{\psi})$. Then for any $\phi\in\mathfrak{X}'_-$, $\boldsymbol{\psi}_\phi$ has global root number $+1$. Thus for all but finitely many such $\phi$ we have $L(\boldsymbol{\psi}_\phi,1)\not=0$ by Proposition \ref{Jianonvanishing}. Thus $H^1_f(\mathcal{K}^S, \boldsymbol{\psi}_\phi)=0$ for them. Therefore by Lemma \ref{division}, there are nonzero elements $P_{f,\chi}(X), g_{f,\chi,v_1}(X)\in \Lambda^-_L$ such that $\phi_0(P_{f,\chi})\not=0$, $P_{f,\chi}(X)$ divides $f_-(X)$ and
\begin{equation}\label{bdd Heegner}
P_{f,\chi}(X)\kappa_{f,\boldsymbol{\chi}}^\psi=f_-(X)\cdot g_{f,\chi,v_1}(X)v_1
\end{equation}
where $\kappa_{f,\boldsymbol{\chi}}^\psi$ is the projection to $\psi$-component of $\kappa_{f,\boldsymbol{\chi}}$.
Indeed we just take $P_{f,\chi}$ to be the polynomial such that $P_{f,\chi}(X)$ takes $0$ values at all points $\phi$ in $\mathfrak{X}_-$ such that $L(\boldsymbol{\psi}_\phi,1)=0$. Then the specializations of $P_{f,\chi}(X)\kappa^\psi_{f,\boldsymbol{\chi}}$ to all points in $\mathfrak{X}_-$ are $0$ (since they all specialize to elements in $H^1_f$). We write $P_{f,\chi}(X)\kappa^\psi_{f,\boldsymbol{\chi}}=P_2(X)\cdot v_1$ for $P_2(X)\in\mathcal{H}_{\frac{1}{2}}$. Then by Lemma \ref{division} and noting that $P_2(X)$ is divisible by $f_-(X)$, we get (\ref{bdd Heegner}).
\begin{definition}
We define the bounded Heegner class $\kappa^-_{f,\chi}:=P_{f,\chi}(X)^{-1}g_{f,\chi,v_1}\cdot v_1\in H^1(\mathcal{K}^S,\psi\otimes\Lambda^-_{L, P_{f,\chi}})$. This is clearly independent of the choice of $v_1$.
\end{definition}

\begin{lemma}
If $\psi_3$ is a self-dual character whose global root number is $-1$, then the image of $H^1(\mathcal{K}^S, \psi_3\otimes\Lambda^-_L)$ in $H^1(\mathcal{K}_p,\psi_3\otimes\Lambda^-_{L, P_{\psi_3,\pm}})$ belongs to $H^1_+(\mathcal{K}_p,\psi_3\otimes\Lambda^-_{L, P_{\psi_3,\pm}})$.
If $\psi_4$ is a self-dual character whose global root number is $+1$, then the image of $H^1(\mathcal{K}^S, \psi_4\otimes\Lambda^-_{L,P_{\psi,\pm}})$ in $H^1(\mathcal{K}_p,\psi_4\otimes\Lambda^-_{L, P_{\psi_4,\pm}})$ belongs to $H^1_-(\mathcal{K}_p,\psi_4\otimes\Lambda^-_{L, P_{\psi_4,\pm}})$.
\end{lemma}
\begin{proof}
Suppose first the global root number is $-1$. By the explicit reciprocity law at points in $\mathfrak{X}'_+$ for $\psi_3$, we have that $z_{\psi_3}\in H^1_+$. Let $v_1$ be a generator of the free $\Lambda^-_L$-module $H^1(\mathcal{K}^S,\psi_3\otimes\Lambda^-_L)$. Then $z_{\psi_3}=P_3(X)v_1$ for some $0\not=P_3(X)\in\Lambda^-_L$ (the nonzero follows from Proposition \ref{Jianonvanishing}). Thus $v_1$ must also be in $H^1_+$. The second statement is proved completely similarly.
\end{proof}
\begin{remark}\label{remark1}
In fact for each $\phi\in\mathfrak{X}_+\cup\mathfrak{X}_-$, we have the global root number for the pair $(f,\chi\chi_\phi)$ is $-1$. Indeed the proof of Lemma \ref{rn} implies the local root numbers at all primes outside $p$ are independent of $\phi$. Moreover since the local representation $\pi_{f,p}$ is a principal series, by the theorem of Saito-Tunnell (see \cite[Theorem 1.3]{YZZ}), the local root number for the pair $(f,\boldsymbol{\chi}_\phi)$ at $p$  is also independent of $\phi$.
\end{remark}
Let $v_2$ be a basis of $H^1_+(\mathcal{K}_p, \psi\otimes\Lambda^-_{L, P_{\psi,\pm}})$. Then by the above lemma we see that the image of the bounded Heegner class $\kappa^-_{f,\boldsymbol{\chi}}$ lies in $H^1_+(\mathcal{K}_p, \psi\otimes\Lambda^-_{L,P_{\psi,\pm}P_{f,\chi}})$, which we write as $g_{f,\chi,v_2}\cdot v_2$ for $g_{f,\chi,v_2}\in\Lambda^-_{L, P_{\psi,\pm}P_{f,\chi}}$. Choose $\eta$ as in Proposition \ref{isvirtualHeegner} and  similarly define $P_{f,\chi\eta}$ and $g_{f,\chi\eta,v_2}$. We claim that neither $g_{f,\chi,v_2}$ nor $g_{f,\chi\eta,v_2}$ is identically $0$. Indeed, we look at arithmetic points $\phi$ in $\mathfrak{X}'_+$, for all but finitely many such $\phi$ we have the vanishing order of $L(\boldsymbol{\psi}_\phi,s)$ at $s=1$ is $1$, and $L(\boldsymbol{\psi}'_\phi,s)\not=0$ (by Remark \ref{remark1} these $\boldsymbol{\psi}'_\phi$ have global root numbers $+1$), again by Proposition \ref{Jianonvanishing}. Note also that $f_-$ takes nonzero values at points in $\mathfrak{X}_+$. We have seen that for those $\phi$ the image of the Heegner point in $H^1_f(\mathcal{K}_p, \boldsymbol{\psi}_\phi\oplus\boldsymbol{\psi}'_\phi)$ is nonzero, by Lemma \ref{AV}. But $H^1_f(\mathcal{K}^S, \boldsymbol{\psi}'_\phi)=0$ by our choice. Thus the image of the Heegner point in $H^1_f(\mathcal{K}_p,\boldsymbol{\psi}_\phi)$ is nonzero. This shows the claim for $g_{f,\chi,v_2}$. The claim for $g_{f,\chi\eta,v_2}$ is proved in the same way.

Let $\psi_3=\psi=\xi\chi$, $\psi_4=\xi\chi^c=\psi^{\prime,c}$, $\psi'_3=\xi\chi\eta$ and $\psi'_4=\xi\chi^c\eta^c$. Note that the restrictions of $\psi_3$ and $\psi'_3$ to $G_{\mathcal{K}_p}$ are the same, and the restrictions of $\psi_4$ and $\psi'_4$ to $G_{\mathcal{K}_p}$ are the same. Without loss of generality we assume $\mathfrak{X}_+$ is the set in Lemma \ref{rn} for $\psi_4$ (the $\mathfrak{X}_-$ case is treated similarly). Let $z_{\psi_3}=f_{\psi_3}\cdot v_2$ and $z_{\psi'_3}=f_{\psi'_3}\cdot v_2$ for $f_{\psi_3}$ and $f_{\psi'_3}$ in $\Lambda^-_{L, P_{\psi,\pm}}$. Let $v_3$ be a basis of $H^1_-(\mathcal{K}_p, \psi_4\otimes\Lambda^-_{L, P_{\psi_4,\pm}})$ where $P_{\psi_4,\pm}$ is defined similarly as $P_{\psi,\pm}$ with $\psi$ replaced by $\psi_4$. Let $z_{\psi_4}=f_{\psi_4}\cdot v_3$ and $z_{\psi'_4}=f_{\psi'_4}\cdot v_3$ for $f_{\psi_4}$ and $f_{\psi'_4}$ in $\Lambda^-_{L, P_{\psi_4,\pm}}$. These $f_{\psi_3}$, $f_{\psi'_3}$, $f_{\psi_4}$, $f_{\psi'_4}$ are not identically $0$, by the explicit reciprocity law at arithmetic points in $\mathfrak{X}_+\cup\mathfrak{X}_-$ of global root numbers $+1$ and Proposition \ref{Jianonvanishing}. Note that these depend on the choices of $v_2$ and $v_3$ but we omit the dependence in the notation for simplicity since  below we only consider certain ratios, making these choices irrelevant.
\begin{proposition}\label{bddHeegnerfamily}
We have 
\begin{equation}
(\frac{g_{f,\chi,v_2}}{g_{f,\chi\eta,v_2}})^2=\frac{f_{\psi_3}f^\circ_{\psi_4}}{f_{\psi'_3}f^\circ_{\psi'_4}}\cdot
\prod_{v\nmid p}\frac{C_{f,\chi,v}}{C_{f,\chi\eta,v}}
\end{equation}
as elements in $\mathrm{Frac}(\Lambda^-_L)$,
where each $C_{f,\chi,v}$ and $C_{f,\chi\eta,v}$ are elements in $\Lambda_L^{-,\times}$ interpolating the local toric integrals at $v$ in the $p$-adic Waldspurger formula for $(f,\chi)$ and $(f,\chi\eta)$ respectively. Here the superscript $\circ$ means the automorphism $\Lambda^-_L\rightarrow \Lambda^-_L$ induced by $\Gamma^-\rightarrow \Gamma^-: g\mapsto g^{-1}$.
\end{proposition}
\begin{proof}
Let $\mathcal{S}$ be the subset of $\mathfrak{X}'_+\cup\mathfrak{X}'_-$ excluding the zeros of $P_{\psi,\pm}$, $P_{\psi',\pm}$, and the denominators and numerators of the reduced form of $g_{f,\chi,v_2}$, $g_{f,\chi\eta,v_2}$, $P_{f,\chi}$, $P_{f,\chi\eta}$, $f_{\psi_3}$, $f_{\psi'_3}$, $f_{\psi_4}$ and $f_{\psi'_4}$ in  $\Frac(\Lambda^-_L)$. Write $\mathcal{S}'$ for these excluded points, which is a finite set. Write $\mathcal{S}^\pm=\mathfrak{X}'_\pm\cap \mathcal{S}$. For simplicity, we write $\chi_\phi$ for the specialization of $\Psi^-_\mathcal{K}$ to  an arithmetic point $\phi$. For $\phi\in\mathcal{S}^+$, let $\mathfrak{X}_{1,\phi}$ and $\mathfrak{X}_{2,\phi}$  be the points in $\mathrm{Spec}\Lambda^-_L$ such that  $\Psi^-_\mathcal{K}$ specializes to  $\chi_\phi\chi_{\phi'}$ with $\phi'\in\mathfrak{X}_1$ and $\phi'\in\mathfrak{X}_2$ respectively.

Fix $\phi\in \CS^+$ and consider the specializations of $\mathcal{L}_{f,\chi}$ and $\mathcal{L}_{f,\chi\eta}$ at arithmetic points $\phi''$ in $\mathfrak{X}_{2,\phi}\backslash \mathcal{S}'$. From the interpolation formula of the anticyclotomic $p$-adic $L$-functions, we see that
$$\phi''(\prod_v\frac{C_{f,\chi,v}}{C_{f,\chi\eta,v}})\phi''(\frac{f_{\psi_3}f^\circ_{\psi_4}}
{f_{\psi'_3}f^\circ_{\psi'_4}})=
\phi'(\frac{\mathcal{L}^2_{f,\chi\chi_\phi}}{\mathcal{L}^2_{f,\chi\eta\chi_\phi}})$$
if $\phi''\in \mathfrak{X}_{2,\phi}\backslash \mathcal{S}'$ satisfies that $\chi_{\phi''}=\chi_\phi\chi_{\phi'}$. Note that the local data for the numerators and denominators at $p$ are all the same, thus the local factors at $p$,  the $p$-adic periods  and the Archimedean periods cancel out.
Thus
$$\phi(\prod_v\frac{C_{f,\chi,v}}{C_{f,\chi\eta,v}})
\phi(\frac{f_{\psi_3}f^\circ_{\psi_4}}{f_{\psi'_3}f^\circ_{\psi'_4}})=
\phi_0(\frac{\mathcal{L}^2_{f,\chi\chi_\phi}}{\mathcal{L}^2_{f,\chi\eta\chi_\phi}}).$$

Note that $\psi\chi_{\phi}$ has Hodge-Tate weights $(1,0)$, and $\psi'\chi_{\phi}$ has Hodge-Tate weights $(0,1)$. Consider the action of $\mathcal{O}_{\CK,p}^\times$ on the Tate module of $\psi\chi_{\phi}$ given by composing the $G_{\CK}$-action with the Artin map. Then by the complex multiplication theory and Shimura's reciprocity law, this $\mathcal{O}_{\CK,p}^\times$-action is given by   $\iota$. Consequently,   $\CK_p$ acts on the tangent space of the Tate module via $\iota$. See for example \cite[3.11]{Milne}. Note that $H^1_f(\mathcal{K}_p, V_f\otimes\chi\chi_{\phi})=H^1_f(\mathcal{K}_p, \psi\chi_\phi\oplus\psi'\chi_\phi)$   is mapped isomorphically to the tangent space $\CK_p\otimes L\cong L\oplus L$ of $\psi\chi_\phi\oplus\psi'\chi_\phi$
via the $p$-adic logarithm map. Let $\mathcal{K}_p$ act on the first $L$ via $\iota$ and on the second $L$ via $\iota^c$. Then the $\psi\chi_\phi$-part corresponds to the first $L$, and the $\psi'\chi_\phi$-part corresponds to the second $L$.

Recall that the $p$-adic Waldspurger formulae involving $\mathcal{L}_{f,\chi\chi_\phi}$ and $\mathcal{L}_{f,\chi\eta\chi_\phi}$ deal with the first $L$, i.e. the $\psi\chi_\phi$-part. 
Thus $$\phi_0(\frac{\mathcal{L}_{f,\chi\chi_\phi}}{\mathcal{L}_{f,\chi\eta\chi_\phi}})
=\phi(\frac{g_{f,\chi,v_2}}{g_{f,\chi\eta,v_2}}).$$  
So
$$\phi(\prod_v\frac{C_{f,\chi,v}}{C_{f,\chi\eta,v}})
\phi(\frac{f_{\psi_3}f^\circ_{\psi_4}}{f_{\psi'_3}f^\circ_{\psi'_4}})=
\phi(\frac{g_{f,\chi,v_2}}{g_{f,\chi\eta,v_2}})^2.$$
By density of all such $\phi\in \CS^+$, we deduce the proposition.
\end{proof}
Note that by our choice for $\chi\eta$,  $\phi_0(g_{f,\chi\eta,v_2})$, $\phi_0(f_{\psi'_3})$ (from Corollary \ref{1111}), $\phi_0(f^\circ_{\psi_4})$ and $\phi_0(f^\circ_{\psi'_4})$ are all nonzero. The image of $\kappa^-_{f,\boldsymbol{\chi}}$ in $H^1(\mathcal{K}_p,\psi\otimes\Lambda^-_{L,P_{\psi,\pm}})$ is $\frac{g_{f,\chi,v_2}}{g_{f,\chi\eta,v_2}}$ times the image of $\kappa^-_{f,\boldsymbol{\chi\eta}}$, and $\phi_0(\kappa_{f,\boldsymbol{\chi\eta}})$ has nonzero image in $H^1_f(\mathcal{K}_p, \psi)$. This enables one to prove the $p$-converse theorem in the same way as in the previous section, using the bounded family $\kappa^-_{f,\boldsymbol{\chi}}$. Indeed the $-$ main conjecture for $\psi_3$ states that over the localization of the Iwasawa algebra, the characteristic ideal of the $-$-Selmer groups of $\psi_3$ is given by $f_{\psi_3}$ and can be deduced from the Iwasawa main conjecture for elliptic units as in Corollary \ref{imc}. Then we can deduce the $p$-converse theorem as in Theorem \ref{main proposition}, using Proposition \ref{bddHeegnerfamily} in the place of the property satisfied by the virtual Heegner family. We omit the details.
\begin{remark}
One can also construct the bounded Heegner class for the $\psi'$-component, and express its local image using the anticyclotomic  $p$-adic $L$-function for the embedding $\iota^c$ by looking at the other half of the arithmetic points $\phi$. If we write this class as $\kappa^+_{f,\boldsymbol{\chi}}$, then we have the decomposition $\kappa_{f,\boldsymbol{\chi}}=f_-\kappa^-_{f,\boldsymbol{\chi}}+f_+\kappa^+_{f,\boldsymbol{\chi}}$.
\end{remark}

\section{Split Case}\label{split case}
We now treat the case when $p$ is split in $\mathcal{K}$. This case is much easier than the non-split case -- the Heegner family can be constructed over the anticyclotomic Iwasawa algebra and we do not need to work with Coleman families since there is a Zariski dense set of arithmetic points of weight $2$. The argument here is just a simplified version of the argument in \cite{WAN} (we only care about $p$-converse theorem here). In this section, we set $\Lambda^-=\CO_L[[X]] \hat{\otimes}_{\mathbb{Z}_p}\mathbb{Z}^{\mathrm{ur}}_p$
and 
$\Lambda^{-,\prime}:=\mathcal{O}_L[[p^{-m}X]]\hat{\otimes}_{\mathbb{Z}_p}\mathbb{Z}^{\mathrm{ur}}_p$.  We add the subscript $L$ for the $p$-inverted version. Here $\mathbb{Z}_p^{\mathrm{ur}}$ is the integer ring of the maximal unramified complete extension of $\mathbb{Q}_p$.

Let $p=v_0\bar{v}_0$ in $\mathcal{K}$. Then the $\pm$-local cohomology groups for the family $\boldsymbol{\psi}$ are nothing but the summands in
$$H^1(\mathcal{K}_p, \boldsymbol{\psi})\simeq H^1(\mathcal{K}_{v_0},\boldsymbol{\psi})\oplus H^1(\mathcal{K}_{\bar{v}_0}, \boldsymbol{\psi}).$$
Analogous to the non-split case, we write the first and second summand above as $H^1_+$ and $H^1_-$. Choose the pair of CM characters $\xi$ and $\chi$ as in the non-split case, requiring  $\xi$ to be  unramified at $p$. Then the form $f=f_\xi$ is good ordinary at $p$. Here the associated Galois representation
$$T_f|_{G_\mathcal{K}}\simeq \xi\oplus \xi^c$$
and the action of $c\in\mathrm{Gal}(\mathcal{K}/\mathbb{Q})$ is given by swapping the summands.  Then by \cite[Theorem 4.15]{LZ1} (applied by taking the $T$ there to be $\xi(-1)$ which has Hodge-Tate weight $0$, and restrict the $2$-variable map in \emph{loc.cit.} to our $1$-variable family), there is a big regulator map
$$\mathrm{Log}_{v_0}: H^1(\mathcal{K}_{v_0}, \xi\otimes\boldsymbol{\chi})\hat{\otimes}\mathbb{Z}^{\mathrm{ur}}_p[1/p]\rightarrow \Lambda^-_L.$$
 In \cite{LZ1} the result is stated for $p>2$. The case when $p=2$ is covered in \cite{Venjakob}\footnote{We thank Otmar Venjakob for writing up \cite{Venjakob}.}.  By our $p$-adic Waldspurger formula and the formula for the constant $C$ in Remark \ref{constantC} (note that in the notation there we have $\chi_1\chi_2\xi_1\xi_2=|\cdot|^{-1}$ as representations of $\mathbb{Q}^\times_p$), we have up to a unit in $\Lambda^-_L$,
$$\mathrm{Log}_{v_0}(\kappa_{f,\boldsymbol{\chi}})=\mathcal{L}_{f,\chi}.$$
This is similar to \cite[Theorem 2.8.8]{Castella}, replacing Proposition 2.3.2 of \emph{op.cit.} with the $p$-adic Waldspurger we proved here.

 Note that by \cite[Theorem 4.15]{LZ1}, the specialization to $\phi_0$ of $\mathrm{Log}_{v_0}$ is a nonzero multiple of the $p$-adic $\log$ map. By considering the motivic weight, we see that $\psi_{v_0}$ is neither the trivial character nor the cyclotomic character. Thus the specialization of $\mathrm{Log}_{v_0}$ to $\phi_0$ is an isomorphism. Thus by possibly enlarging $m$, we can ensure that $\mathrm{Log}_{v_0}$ becomes an isomorphism after extending the coefficients to $\Lambda_L^{-,\prime}$.
As in the non-split case $\mathcal{L}_{f,\chi}$ is equal to $\mathcal{L}_{\psi}^{\bullet}\cdot\mathcal{L}^\circ_{\psi'}$ up to an element in $\Lambda_L^{-,\times}$. Moreover by \cite[Theorem 4.15]{LZ1} and \cite{Venjakob} again, we have big regulator maps $$\mathrm{Exp}^*_{\bar{v}_0}:\ H^1(\mathcal{K}_{\bar{v}_0}, \boldsymbol{\psi}')\hat{\otimes}_{\mathbb{Z}_p}\mathbb{Z}^{\mathrm{ur}}_p[1/p]\rightarrow \Lambda_L^-$$
$$\mathrm{Exp}^*_{v_0}:\ H^1(\mathcal{K}_{v_0}, \boldsymbol{\psi})\hat{\otimes}_{\mathbb{Z}_p}\mathbb{Z}^{\mathrm{ur}}_p[1/p]\rightarrow \Lambda_L^-,$$ such that by the explicit reciprocity law for elliptic units, up to an element in $\Lambda_L^{-,\times}$
$$\mathrm{Exp}^*_{v_0}(z_{\psi,p^\infty})=\mathcal{L}^{\bullet}_\psi,\quad \mathrm{Exp}^*_{\bar{v}_0}(z_{\psi',p^\infty})=\mathcal{L}_{\psi'}.$$
 Moreover  by \cite[Theorem 4.15]{LZ1}, the specialization of $\mathrm{Exp}^*_{\bar{v}_0}$ to $\phi_0$ is a nonzero multiple of the $p$-adic $\exp^*$ map, and the specialization of $\mathrm{Exp}^*_{v_0}$ to $\phi_0$ is a nonzero multiple of the $p$-adic $\log$ map. Thus enlarging $m$ if necessary, these regulator maps become isomorphisms when extending coefficients to $\Lambda^{-,\prime}_L\hat{\otimes}_{\BZ_p}\BZ_p^{\ur}$. Now the main theorem uses the same argument as in the proof of Theorem \ref{main proposition}: the $-$-main conjecture for $\mathcal{L}^{\bullet}_\psi$ and the $-$-main conjecture for $\mathcal{L}_{\psi'}$ over $\Lambda^{-,\prime}_L$ follow from the corresponding main conjecture for elliptic units as in the proof of Corollary \ref{AIMC}, using that $\mathrm{Exp}^*_{\bar{v}_0}$ and $\mathrm{Exp}^*_{v_0}$ are isomorphisms over $\Lambda^{-,\prime}_L$. Perrin-Riou's conjecture, namely 
\begin{equation}
    2\mathrm{leng}_P(\frac{H^1(\mathcal{K}^S,\boldsymbol{\psi})}{\Lambda^{-,\prime}_L\kappa_{f,\boldsymbol{\chi}}})=\mathrm{leng}_P X^+_{\psi,\mathrm{tor}}
\end{equation}
(recall $P$ is the prime corresponding to the origin point) follows from the same  argument as  in the proof of Theorem \ref{main proposition}. The $p$-converse theorem follows from Perrin-Riou's conjecture in the same way as the last two paragraphs of the proof of Theorem \ref{main proposition}.

\appendix

\section{Local test vector and local period integral}\label{Local test vector}
In this section, we provide some useful results on local test vectors and local period integrals for Rankin pairs. In the following, \begin{itemize}
\item Let $F$ be a $p$-adic field with ring of integers $\CO$ and uniformizer $\varpi$.
\item Fix a non-trivial additive character $\psi_F: \CO\bs F\to\BC^\times$ of level zero and let $d^\times y$ be the  Haar measure on $F^\times$ with $\Vol(\CO^\times)=1$.

\item  Let $\pi$ be an infinite dimensional irreducible smooth admissible complex $\GL_2(F)$-representation with central character $\omega$. 

\item Let $N\subset P\subset \GL_2(F)$ be the upper unipotent subgroup and upper Borel subgroup respectively. For any $m\in\BN$, let $U_0(m):=\begin{pmatrix}\CO^\times & \CO\\ \varpi^m\CO & \CO^\times\end{pmatrix}$.

\item  Identify $\pi$ (resp. $\pi^\vee$) with its  $\psi_F$-Kirillov (resp. $\bar{\psi}_F$-Kirillov model).
\item Take $n\geq\max\{1,\Cond \pi\}$ and set $\varphi=\Vol(1+\varpi^n\CO)^{-1}\Char_{1+\varpi^n\CO}\in \pi$. 
\end{itemize}
Note that $\varphi$ is fixed by \[\{\begin{pmatrix}a & b\\ c & d\end{pmatrix}\in\GL_2(\CO)| a-1\equiv d-1\equiv 0\mod \varpi^n, \varpi^{2n}\mid c\}\]
and in the $\bar{\psi}_F$-Kirillov model, $\varphi=\Vol(1+\varpi^n\CO)^{-1}\Char_{-1+\varpi^n\CO}$.

 Let $K/F$ be an \'etale quadratic  extension with an embedding $ K\hookrightarrow M_2(F)$ and associated quadratic character $\chi_K$. Let $\chi:\ K^\times\to\BC^\times$ be a smooth character such that $\chi|_{F^\times}\omega=1$. 
 By the result of Saito-Tunnell, $m(\pi,\chi):=\dim\Hom_{K^\times}(\pi,\chi^{-1})\leq1$.  When $m(\pi,\chi)=1$, an element $v\in\pi$ is called a test vector for the pair $(\pi,\chi)$ (this depends on the embedding) if $L(v)\neq0$ for some (hence any)  $L\neq0\in \Hom_{K^\times}(\pi|_{K^\times},\chi^{-1})$.    Note that
\begin{itemize}
    \item when $K=F\oplus F$, $m(\pi,\chi)=1$,
    \item when $K/F$ is a field extension and $\pi$ is non-supercuspidal, the only case $m(\pi,\chi)=0$ is $\pi=\St_F\otimes\xi$
with $\xi\circ N_{K/F}\cdot\chi=1$.
\end{itemize}  
Throughout, we shall assume $m(\pi,\chi)=1$ and the pairing  \[ (-,-):\ \pi\times\pi^\vee\to\BC,\quad (W,W^\vee)\mapsto\int_{F^\times}W(y)W^\vee(y)d^\times y\]
is well-defined. Particularly, this holds when $\pi$ is unitary.   We shall consider the canonical local period 
\[\alpha(\varphi,\pi(J)\varphi\otimes\omega^{-1};\chi)=\frac{L(1,\chi_K)L(1,\pi,\ad)}{\zeta(2)L(1/2,\pi,\chi)}\int_{K^\times/F^\times}(\pi(t)\varphi,\pi(J)\varphi\otimes\omega^{-1})\chi(t)d^\times t.\]
\subsection{The case $K/F$ is a field extension}
In this subsection, assume $K/F$ is a field extension. Take $\ell\geq 2n$ and  let 
$\varphi_{\ell}=\pi(g_{\ell,n})\varphi$ where $g_{\ell,n}=\begin{pmatrix} \varpi^{2n-\ell+1-s} & \varpi^{-\ell+n+1-s}\\ 0 & 1\end{pmatrix}$. Here $s=1$ (resp. $s=2$) when $K/F$ is relatively inert (resp. ramified).

\begin{prop}\label{Test vector} Assume that  $\pi$ is non-supercuspidal and $m(\pi,\chi)=1$. 
 Then $\varphi_{\ell}$  is a test vector for $(\pi,\chi)$ for all $\ell\gg2n$.
\end{prop}

We start by describing $\chi$-eigenvectors for $\pi$ non-supercuspidal.  Let $W_\chi\neq0\in \pi^\vee$ such that $\pi^\vee(t)W_\chi=\chi(t)W_\chi$ for all $t\in K^\times$. 
\begin{lem}\label{eigen}Assume $\pi$ is non-supercuspidal. Then the support of $W_\chi$ is not bounded. More precisely, there exists $C\gg0$ and non-zero constants $c_1$,$c_2$ (resp. $c$) such that when $\val(y)>C$
\begin{itemize}
\item $W_\chi(y)=c_1 \xi_1(y)|y|^{1/2} + c_2 \xi_2(y)|y|^{1/2}$  when  $\pi^\vee=I(\xi_1,\xi_2)$ with $\xi_1\neq\xi_2$,
\item $W_\chi(y)=c_1\xi(y)|y|^{1/2} + c_2\xi(y)|y|^{1/2}\val_\varpi(y)$  when $\pi^\vee=I(\xi_1,\xi_2)$ with $\xi_1=\xi_2=\xi$,
\item $W_\chi(y)=c\xi(y)|y|$  when $\pi^\vee=\St_F\otimes\xi$.
\end{itemize}
 Moreover $W\mapsto (W,W_\chi)$ belongs to $\Hom_{K^\times}(\pi|_{K^\times},\chi^{-1})$.
\end{lem}
\begin{proof}For the moreover part, note that $\pi|_{K^\times}=\oplus\theta^{-1}$ where the summation runs over all $\theta$ such that $m(\pi,\theta)=1$. Set $W_{\theta}\in \pi$ be a non-zero $\theta^{-1}$-eigenvector. Then $W=\sum c_\theta W_\theta$ for some $c_\theta\in\BC$ and the moreover part follows.

 Consider the $P$-subrepresentation $\pi^\vee(N)=\langle \pi^\vee(n)v-v| n\in N,\ v\in \pi^\vee\rangle\subset \pi^\vee$, which just consists of all Schwartz functions on $F^\times$. 
If $W_\chi\in \pi^\vee(N)$, the $\GL_2(F)=PK^\times$-translation of $W_\chi$ all belong to $\pi^\vee(N)$, which forces $\pi^\vee(N)=\pi^\vee$. This contradicts with the assumption $\pi^\vee$ is non-supercuspidal. Same arguments also shows the $P$-submodule generated by $\pi^\vee(N)$ and $W_\chi$ must be $\pi^\vee$. Thus the image of $W_\chi$ in the Jacquet module is a generator. Based on these observations, the remaining part follows from \cite[Theorem 4.7.2 \& 4.7.3]{Bum97}.
\end{proof}

\begin{proof}[Proof of Proposition \ref{Test vector}] Note that 
\begin{align*}
(\varphi_\ell,W_\chi)&=\int_{F^\times}\varphi(\begin{pmatrix} y\varpi^{2n-\ell+1-s} & y\varpi^{-\ell+n+1-s}\\ 0 & 1\end{pmatrix})W_\chi(y)d^\times y\\&=\int_{F^\times}\psi_F(y\varpi^{-\ell+n+1-s})\varphi(y\varpi^{2n-\ell+1-s})W_\chi(y)d^\times y\\
&=\int_{F^\times}\psi_F(y\varpi^{-n})\varphi(y)W_\chi(y\varpi^{\ell+s-1-2n})d^\times y\\
&=\int_{1+\varpi^n\CO}\psi_F(y\varpi^{-n})W_\chi(y\varpi^{\ell+s-1-2n})d^\times y=\psi_F(\varpi^{-n})W_\chi(\varpi^{\ell+s-1-2n})
\end{align*}
where the last equality follows from Lemma \ref{eigen} and the assumption $\ell\gg 2n$. Again by Lemma \ref{eigen}, we deduce the stated result.
\end{proof}
\subsection{The case $K=F\oplus F$}\label{Local test vector II}
In this subsection, we consider the case  $K=F\oplus F$. Take the embedding \[K\hookrightarrow M_2(F),\quad (a,b)\mapsto S^{-1}\begin{pmatrix} b & 0\\ 0 & a\end{pmatrix} S, \quad S=\begin{pmatrix}1 & 1\\ 0 & 1\end{pmatrix}.\]
Assume $J=S^{-1}wS$ with $w=\begin{pmatrix} 0 & 1\\ -1 & 0\end{pmatrix}$ and $\chi=(\chi_1,\chi_2)$. Then
\begin{align*}
    \alpha(W,\pi(J)W\otimes\omega^{-1};\chi)
    =&\frac{L(1,\chi_K)L(1,\pi,\ad)}{\zeta(2)L(1/2,\pi,\chi)} \int_{F^\times}\pi(S)W(y)\chi_2(y)d^\times y\\
   &\times \int_{F^\times}\pi(SJ)W(y)\chi_2^{-1}\omega^{-1}(y)d^\times y
    \end{align*}
 For $W=\varphi$,  $\int_{F^\times}\varphi(y)\chi_2(y)d^\times y=1$.  By the functional equation, 
 \begin{align*}
   \int_{F^\times}\pi(SJ)\varphi(y)\chi_2^{-1}\omega^{-1}(y)d^\times y&=\gamma(1/2,\pi,\chi_2,\bar{\psi}_F)\int_{F^\times}\varphi(y)\chi_2(y)d^\times y\\
   &=\chi_1(-1)\gamma(1/2,\pi,\chi_2,\psi_F)
 \end{align*}
Thus $\varphi$ is  a test vector and 
 \[\alpha(\varphi,\pi(J)\varphi\otimes\omega^{-1};\chi)=\chi_1(-1)\frac{L(1,\chi_K)L(1,\pi,\ad)}{\zeta(2)L(1/2,\pi,\chi)}\epsilon(1/2,\pi,\chi_2,\psi_F)\]
A similar computation shows that if we take $W$ such that $W\otimes\chi_2$ is new in $\pi\otimes\chi_2$, then
  \[\alpha(W,\pi(J)W\otimes\omega^{-1};\chi)=\chi_1(-1)\zeta(2)^{-1}L(1,\chi_K)L(1,\pi,\ad)\epsilon(1/2,\pi,\chi_2,\psi_F)\]

\section{Generalized Heegner cycles in Coleman families}\label{Heegner cycles in family}
In this section, we adapt the results on generalized Heegner cycles in Coleman families from \cite{JLZ} and \cite{Wal25} to our setting. We shall also borrow ideas from the exposition in \cite[Section 6]{Dis22}. Fix a finite order Hecke character $\omega$ of $\BQ$ with $p$-conductor $n_0$.

Let $\sigma$ be a cuspidal automorphic $\GL(2)$-representation of
weight $r+2\geq 2$ and central character $\omega|\cdot|^{-r}$.  In other words, $\sigma=\sigma(f)$ is the cuspidal automorphic representation generated by some cuspidal newform $f$ of weight $r+2$ such that $L(s,\sigma)=L(s+1/2,f)$. 

Let $\chi$ be an algebraic Hecke character of
$\CK$ of infinity type $(r-j,j)$ with $0\leq j\leq r$ and $\chi|_{\BA^\times}\omega|\cdot|^{-r}=1${\footnote{The Rankin pair $(\sigma,\chi)$ in this Appendix differs from that in Section \ref{Explicit reciprocity law} by a twist of $|\cdot|$.}}.  Let $V_\sigma$ be the two-dimensional Galois representation attached to $\sigma$, i.e. $L(s-1/2,V_\sigma)=L(s,\sigma)$. We shall call $(\sigma,\chi)$ or $(V_\sigma,\chi)$  a {\bf balanced} Rankin pair.

 Fix  a small enough tame level $V^p\subset \CK_{\BA}^{p\infty,\times}$.  For $m\geq n_0$, let $V_m=V^pV(m)$ where $V(m)=1+2p^{n_0}\BZ_p+2p^m\CO_{\CK_p}\subset \CO_{\CK_p}^\times$. Let $F_m/\CK$  be the abelian extension with Galois group $\Gal(F_m/\CK)\cong \CK^\times\bs \BA_{\CK}^{\infty,\times}/V_m$.   Let $F_\infty=\cup_mF_m$ and $\tilde{\Gamma}^-=\Gal(F_\infty/\CK)$. We shall view $\chi$ as a character of $\tilde{\Gamma}^-$ via the geometrically normalized Artin map.
Let $\CW(\tilde{\Gamma}^-)$ be the anticyclotomic weight space. Let $\boldsymbol{\kappa}^-$ be the universal character.

Assume $\sigma_p$ is non-supercuspidal and fix a refinement $(\sigma,\varphi_p)$, i.e., $\varphi_p\in \sigma_p^{N(\BZ_p)}$ is $U_p$-eigen with non-zero eigenvalue $\alpha$,  that is {\bf noble} in the sense of~\cite{Han17}. The refinement $(\sigma,\varphi_p)$ determines a point  in the eigencurve (of $\GL(2)$). Take a disc  $D$ in the eigencurve containing this point and let $\CV$ be the   rank-$2$ $G_\BQ$-representation on $D$ which specializes to $V_\sigma$. Let $\tilde{D}\subset D\times\CW(\tilde{\Gamma}^-)$ be the locus cut out by the self-dual condition $\det\CV \cdot \boldsymbol{\kappa}^-\circ \mathrm{Ver}=1$. Here $\mathrm{Ver}: G_{\BQ}^{ab} \to G_{\CK}^{ab}$ is the Verschiebung map. Let $\tilde{\CV}$ be the rank two Galois representation on $\tilde{D}$.

Assume there exists an indefinite quaternion algebra $B$ such that $B_p$ is split and for all $v\nmid p\infty$,  $\epsilon(B_v)\chi_{v}(-1)\chi_{\CK}(-1)=\epsilon(1/2,\sigma_{v}\times\chi_{v})$.  Then attached to the balanced Rankin pair $(V_\sigma,\chi)$, one can define  cohomology class $z^{V_\sigma,\varphi_p,\chi} \in H^1_f(\CK^S,\, V_\sigma\otimes\chi)$ from the generalized Heegner cycle in the Kuga-Sato variety over the Shimura curve associated to $B$.  Here $S$ is a finite set of places containing $p$ and the ramified places of the Galois representation.

\begin{thm}\label{BigHeegner}There  exists a
class $z \in H^1\!\bigl(\CK^S,\tilde{\CV}\bigr)$ 
whose specialization at any balanced Rankin pair  $(V_\sigma,\chi)\in\tilde{D}$ is $z^{V_\sigma,\varphi_p,\chi}$.
\end{thm}
In the following, we shall follow the notations and conventions in Subsection \ref{NC}.  In particular, we   denote the injection $\CK\subset B$ by $\iota_0$. Set $\iota_m=g_m^{-1}\iota_0 g_m$ for $m\geq1$, where $g_m=\begin{pmatrix}p^{m} & 0 \\ 0 & 1\end{pmatrix}$.  Let  $\tau_0^*=-\frac{1}{\tau_0}$, $\tau_m^*=p^{-m}\tau_0^*$ and $\tau_m=p^{m}\tau_0$.  
\subsection{Eigen-distributions}
Denote the algebraic group $\CK^\times$ and $B^\times$ by $H$ and $G$ respectively. 
Let $V=L^2$ (realized as column vectors) denote the standard left $\GL_2(L)$-representation: 
\[g\cdot \begin{pmatrix}a \\ b\end{pmatrix}=g\begin{pmatrix}a \\ b\end{pmatrix}\quad g\in \GL_2(L)\]
Via the pairing \[L^2\times L^2,\quad (\begin{pmatrix}a \\ b \end{pmatrix},\begin{pmatrix}c \\ d \end{pmatrix})\mapsto ac+bd,\] 
one also realizes the dual representation $V^\vee$ as row vectors. For any $r\in \BN$, $\Sym^r V$ can be realized as the space of homogeneous polynomials of degree $r$ in two variables with the action
\[g\cdot f(X,Y)=f((X,Y)g)\quad g\in \GL_2(L)\]
Its dual $(\Sym^r V)^\vee$ can be realized as  the symmetric tensor of $V^\vee$
\[\TSym^r(V^\vee)=\{x\in V^{\vee,\otimes r}| g\cdot x=x\ \forall\ g\in S_r\}\]
With respect to the embedding $\iota_m:\ H\to G$,  $H$ acts on the vector $[\iota(\tau^*_m),1]^T$ and $[\iota^c(\tau^*_m),1]^T$ via $\iota$ and $\iota^c$. Thus $H$ acts on $e_m=[\tau_m, 1]^T$ and $\bar{e}_m=[\bar{\tau}_m, 1]$ via $\iota^{-1}$ and $\iota^{c,-1}$ respectively. Hence for $0\leq a\leq r$, $e_m^{[a,r-a]}:=e_m^a\cdot\bar{e}_m^{r-a}$ generates the unique eigen-line in $\TSym^r(V^\vee)$ on which $H$ acts via $\iota^{-a}(\iota^c)^{a-r}$. 

Fix the $\GL_2(\CO_L)$-stable lattice $T=\CO_L^2\subset V=L^2$. Let $A_{r,n}$ denote the space of power series in the variable $Z$ with coefficients in $L$ which converges on $p^n\CO_{\BC_p}$ and equipped with the $U_0(n)$-action 
\[\begin{pmatrix} a & b\\ c & d\end{pmatrix}\cdot f(Z)=(bZ+d)^rf(\frac{aZ+c}{bZ+d})\]
Note that  $g=\begin{pmatrix}p & 0 \\ 0 & 1\end{pmatrix}$ also acts on $A_{r,n}$ by sending $f(Z)$ to $f(pZ)$. Clearly,  $A_{r,n}$ is a Banach space with  respect to the supremum norm. Let  $A_{r,n}^\circ$  be the unit ball and view $\Sym^rT$ as a $(U_0(n),g)$-submodule of $A_{r,n}^\circ$ via the  injection \[\Sym^rT\hookrightarrow A_{r,n}^\circ,\quad f(X,Y)\mapsto f(Z,1).\]
Let $D_{r,n}^\circ$ be the space of distributions $\Hom_{\CO_L}(A_{r,n}^\circ,\CO_L)$ and let \[D_{r,n}:=D_{r,n}^\circ\otimes_{\CO_L}L=\Hom_{ct}(A_{r,n},L).\] By duality, one has the $(U_0(n),g^{-1})$-equivariant moment map \[\mom^r:\ D_{r,n}^\circ\to\TSym^r(T^\vee)\]

More generally, for any $p$-adically complete and separated $\CO_L$-algebra $R$ equipped with an $n$-analytic character 
$\nu:\ \BZ_p^\times\to R^\times$, i.e., there exists $v\in R[\frac{1}{p}]$ \[\nu(t)=\exp(v\log(t))=\sum_{k\geq0}{v\choose k}(t-1)^k\quad   \forall\ t\in 1+p^n\BZ_p\]
the space $A_{\nu,n}^\circ=\{\sum_{k\geq0}a_k(\frac{Z}{p^n})^k| a_k\in R\to 0\}$
admits an action of $U_0(n)$ and  $g$ given  by the formulae as above.  Similarly, let $D_{\nu,n}^\circ$ be the set of $R$-linear maps $A_{\nu,n}^\circ\to R$. Then $D_{\nu,n}^\circ$ admits an action of $U_0(n)$ and  $g^{-1}$. If there exists a continuous map $r:\ R\to \CO_L$ such that $r\circ \nu$ is the character \[r:\quad \BZ_p^\times\to\CO_L^\times,\quad x\mapsto x^r,\]
then one has the moment map:
\[\mom^r:\ D_{\nu,n}^\circ\hat{\otimes}_{R,r}\CO_L\to \TSym^rT^\vee.\]
We shall use this construction for $R$ being the bounded by 1 functions on an open disc in $ \CW_n\subset \CW$. Here $\CW$ is the weight space parameterizing  locally analytic characters of $\BZ_p^\times$ and $\CW_n\subset \CW$ is the locus consisting of all $n$-analytic characters.

For $m\geq n$,  $\iota(\tau_m)\in p^n\CO_L$ and thus the evaluation-at-$\iota(\tau_m)$
 map $A_{\nu,n}^\circ\to R$ defines an element  $e_{\nu,0,m}\in D_{\nu,n}^\circ$.  Thus $\iota_m(V(m))\subset U_0(n)$. By \cite[Proposition 4.3.1]{JLZ}, $V(m)$ acts on $e_{\nu,0,m}$ by $\iota^{-\nu}:=\nu^{-1}\circ\iota$ and $\mom^r(e_{\nu,0,m})=e_m^{[r,0]}$. Following \cite[Section 4.4]{JLZ}, let \[e_{\nu,j,m}=\prod_j(e_{\nu-j,0,m}\otimes e_m^{[0,j]})\in D_{\nu,n}\] where $\prod_j:\ D_{\nu-j,n}\otimes \TSym^j V^\vee\to D_{\nu,n}$ is the overconvergent projector for $j\in\BN$.  Then $\CO_{p^m}^\times$ acts on $e_{\nu,0,m}$ by $\iota^{-\nu+j}(\iota^c)^{-j}$ and $\mom^r(e_{\nu,j,m})=\begin{cases} e_m^{[r-j,j]} & 0\leq j\leq r\\
0 & j>r\end{cases}$.

\subsection{\'etale local systems and generalized Heegner classes}

Fix a (small enough) tame level $U^p\subset \BB^{p\infty,\times}$ and let $U=U^pU_p$ with $U_p=U_0(n)$ or $U_1(n)$. Let $Y_n$ be the (open) Shimura curve  whose complex points are $G(\BQ)\bs (\BC-\BR)\times G(\BA_f)/ U$. Then $U_0(n)$-representation $D_{R,n}^\circ$ and $\TSym^rT^\vee$ induce \'etale local systems on $Y_n$, which we denote by the same notations. Note that  $Y_n$  parameterizes (false) elliptic curves with certain level structures.  From this viewpoint, $T$ and $T^\vee$ is just (the $e$-component of) the relative \'etale cohomology and  the Tate module of the universal false elliptic curve respectively (with $\CO_L$-coefficients).

Let $A$ be the  underlying (false) elliptic curve corresponding to $P=\iota_0([1])$. Let $V_A$ be (the $e$-component of) the first \'etale cohomology of $A$ with $\CO_L$-coefficient. When $0<r\in2\BN$, let $W_r$ be the canonical desingularization of the $r/2$-fold (resp. $r$-fold) fiber product of the universal false  elliptic curve (resp. elliptic curve) over $X_n$ (the compactification of $Y_n$) with itself and $C_r=W_r\times A^{r/2}$ (resp. $W_r\times A^r$) when $B\neq M_2(\BQ)$ (resp. $B=M_2(\BQ)$). Let $P_m=\iota_m([1])$ and $A_m$ be the underlying (false) elliptic curve. As explained by \cite[Section 2.8.4]{Wal25}, there exists an isogeny  $ A\to A_m$ compatible with the level structures on $A$ and $A_m$ respectively. Denote by $\Gamma_m$  the graph of this isogeny.   By \cite[Section 6]{Bro14} and \cite[Section 2]{BDP}, there exists a projector $\epsilon_W\in \mathrm{Corr}(W_r,W_r)$ such that $\epsilon_WH^*_{et}(W_r,L)=H^1_{et}(X_n, \TSym^r(V))$. The specialization of $\epsilon_W$ at $P$ gives a projector $\epsilon_A\in\mathrm{Corr}(A^{r/2},A^{r/2})$ (resp. $\epsilon_A\in\mathrm{Corr}(A^{r},A^{ r})$) such that $\epsilon_AH_{et}^*(A^{r/2},L)=\TSym^{r}V_A$ (resp. $\epsilon_AH^*(A^{r},L)=\TSym^{r}V_A$). Let  $\epsilon=\epsilon_W\epsilon_A$. Consider   the power $\Gamma_m^{r/2}$ (resp. $\Gamma_m^r$), which is  a cycle in $C_r$ fibered over the CM point $P_m$. The generalized Heegner cycle $\Delta_m:=\epsilon\Gamma_m^{r/2}$ (resp. $\epsilon\Gamma_m^{r}$) has codimension $r+1$. Since $P_m$ and hence $\Delta_m$ is defined over $F_m$,  the image $\Cl(\Delta_m)$ of $\Delta_m$ under the \'etale class map lies in \[H_{et}^2((X_n)_{F_m},  \TSym^r(V^\vee)\otimes \TSym^r(V_A)(1))=\epsilon H^{2r+2}_{et}((C_r)_{F_m}, L)\]
Consider the decomposition $\TSym^r(V_A)=\oplus_{j=0}^r \iota^{r-j}(\iota^c)^j$ and the natural map \[H_{et}^2(X_n,  \TSym^r(V)\otimes \TSym^r(V_A))\to H_{et}^2(Y_n,  \TSym^r(V)\otimes \TSym^r(V_A)).\]
We denote the resulting image of $\Cl(\Delta_m)$ in $H^2_{et}((Y_n)_{F_m}, \TSym^r(V^\vee)\otimes\iota^{r-j}(\iota^c)^j)$ by $z_{m,n}^{[r-j,j]}$.

Let $\pi=\pi^\infty\otimes\pi_\infty$ be a cuspidal  automorphic $\BB^\times$-representation whose Jacquet-Langlands correspondence is $\sigma$. Note that $\pi_p=\sigma_p$. The irreducible smooth  $\BB_f^\times$-representation $\pi^\infty$ admits a model over $M_\pi$, the Hecke field of $\pi$. Base change this model to $L$ and by abuse of notations, denote the outcome also by $\pi^\infty$. Then by \cite[Proposition 2.5.2]{Dis22}, one has the $G_{\BQ}\times \BB^{\infty,\times}$-decomposition 
\[H^1_{et,par}(Y_{U,\bar{\BQ}}, \TSym^r(V)(1))\cong H_{1,et,par}(Y_{U,\bar{\BQ}}, \Sym^r(V))\cong \oplus \tilde{\pi}^{\vee,\infty,U}\otimes V_{\tilde{\pi}}\]
 where \begin{itemize}
 \item par stands for the parabolic part,
 \item $\tilde{\pi}$ runs over all cuspidal automorphic $\BB^\times$-representation  with $\tilde{\pi}_\infty=\pi_\infty$,
 \item $V_{\tilde{\pi}}$ is the $2$-dimensional $G_{\BQ}$-representation such that $L(s-1/2,V_{\tilde{\pi}})=L(s,\tilde{\pi})$.
 \end{itemize}   
Consider the norm $\tilde{z}_{m,n}^{[r-j,j]}=\frac{1}{\sharp Z_m}\mathrm{Cores}_{F_m/\CK}z_{m,n}^{[r-j,j]}$. When $r=0$ and $B\neq M_2(\BQ)$, we modify $\tilde{z}_{m,n}^{[r-j,j]}$ by the Hodge class to make it cohomologically trivial as \cite[Section 6.2]{Dis22}. Take any holomorphic vector $\varphi=\varphi_f\otimes\varphi_\infty\in\pi^U$ and let  $z_m^{\varphi,\chi}=\varphi_f(\tilde{z}_{m,n}^{[r-j,j]})\in H^1(\CK^S, V_\pi\otimes\chi)$ {\footnote{When $B=M_2(\BQ)$ and $\varphi$ corresponds to a newform $f$, our $z^{\varphi,\chi}$ equals $\sharp \CK^\times\BA^{\infty,\times}\bs \BA_{\CK}^{\infty,\times}/V$ times the  $\ z_{et}^{[f,\chi]}$ in \cite[Proposition 3.5.3]{JLZ}}. Note that when $r=0$ and $B=M_2(\BQ)$, the fact $\Hom^\circ(X_U, A_\pi)\otimes_{M_\pi} L\cong \pi^{\infty,U}$ implies   $z^{\varphi,\chi}$ coincides with the image of the modified Heegner class in \cite[Section 6.2]{Dis22} along $\varphi$. Here $A_\pi$ is the $\GL(2)$-type abelian variety attached to $\pi$, and  $\circ$ means sending the Hodge class to the origin. By \cite{NN16}, we actually have $z^{\varphi,\chi}\in H^1_f(\CK^S,V_\pi\otimes\chi)$.

Now take $\varphi_f\in \pi^U\subset \Hom_{G_{\BQ}}(V_\pi,H^1(Y_{n,\bar{\BQ}}, \Sym^r(V)))$ with $p$-component $\varphi_p$. Then by  \cite[Lemma A.2.3]{Dis22} (see also \cite[Proposition 5.1.2]{JLZ}) and tracking our normalization, $z_{m+1}^{\varphi,\chi}=\alpha z_m^{\varphi,\chi}$.  We shall denote $\frac{z_m^{\varphi,\chi}}{\alpha^{m}{r \choose j}}$ by $z^{V_\sigma,\varphi_p,\chi}$.

\begin{remark}\label{constantC}
Note that \[\varphi\mapsto \langle \log_p(z^{\varphi,\chi}), \omega_\pi\rangle\] 
defines an element in the one-dimensional space $\Hom_{\BA_\CK^{\infty,\times}}(\pi^{\infty}\otimes\chi,L)$. This allows us to compare $z^{\varphi,\chi}$ with the generalized Heegner class considered in Section \ref{Explicit reciprocity law} after applying $\langle \log_p(-), \omega_\pi\rangle$.  By the local computation in \cite[Appendix A]{Dis22} and Appendix \ref{Local test vector}, the automorphic form $\varphi$ used in the construction of these classes are local test vectors for all $(\pi_v,\chi_v)$. Thus by multiplicity one, the difference is a non-zero constant $C$ depending on the local data at $p$. Actually when $p$ splits in $\CK$ and $\pi_p\subset I(\xi_1,\xi_2)$, the local test vectors in Section \ref{Explicit reciprocity law} and this Appendix are  $\Vol(1+p^n\BZ_p)^{-1}\Char_{1+p^n\BZ_p}$ and $\xi_i|p^{-n}|^{1/2}\Char_{\BZ_p}(p^ny)$ (in the Kirillov model and $n\gg0$) respectively. Here $\xi_i$ depends on our choice of refinement.  Consider the linear functional
\[\pi_p\to\BC^\times,\quad W\in \pi_p\mapsto \int_{\BQ_p^\times\bs \CK_p^\times}W(t)\chi_p(t)d^\times t\]
and assume the local component $\chi_p=(\chi_1,\chi_2)$. Then by simple computation as in Appendix \ref{Local test vector II} together with \cite[Lemma A.3.1]{Dis22}, one can  show the difference is essentially $\xi_i\chi_2(-1)\gamma(1/2, \xi_i\chi_2,\psi)$.
\end{remark}

\subsection{Generalized Heegner classes in Coleman families}

To $p$-adically interpolate $z_{m,n}^{[r-j,j]}$, we give it an alternative description. Assume $V^p\subset \CK_{\BA}^{p\infty,\times}\cap U^p$ and let $Z_m$ be the Shimura set attached to $H$ of level $V_m=V^pV(m)$. For any $n$-analytic character $\nu:\ \BZ_p^\times\to R^\times$ and $j\in\BN$, the $\CO_{p^m}^\times$-representation $\iota^{\nu-j}(\iota^c)^j$ defines an \'etale local system on $Z_m$. Following \cite{JLZ}, denote the $G\times H$-representation $\Sym^{a+b}(V)\boxtimes \iota^{-a}\otimes(\iota^c)^{-b}$ by $V_{a,b}$ for $a,b\geq0$. Then $H^2_{et}(Y_n\times Z_m, V_{r-j,j}^\vee(1))=H^2_{et}((Y_n)_{F_m}, \TSym^r(T^\vee)\otimes\iota^{r-j}(\iota^c)^j)$.

Consider the map $Z_n\xrightarrow{\delta_m=(\iota_m,\id)} Y_n\times Z_m$ where 
\[\iota_m:\ Z_m\xrightarrow{\iota_0} Y_{g_m^{-1}Ug_m}\xrightarrow{T_{g_m}} Y_n,\quad g_m=\begin{pmatrix}p^m & 0\\ 0 & 1\end{pmatrix}\]
is the composition of the immersion $\iota_0$ induced by the injection $\iota:\ H\to G$ of algebraic groups and  the translation-by-$g_m$ map $T_{g_m}$. For $0\leq j\leq r$, view $e_m^{[r-j,j]}$ as a section of $H_{et}^0(Z_m, \delta_m^*(V_{r-j,j}^\vee))$.  Then by \cite[Section 3.4]{JLZ} and \cite[Section 3.5]{Wal25},
\[z_{m,n}^{[r-j,j]}=(\delta_m)_*(e_m^{[r-j,j]})\in H^2_{et}(Y_n\times Z_m, V_{r-j,j}^\vee(1))\]

View $e_{\nu,j,m}$ as section in $H^0_{et}(Z_m, \delta_m^*(D_{\nu,m}^\circ)\otimes \iota^{\nu-j}(\iota^c)^j))$ and let 
\[z_{\nu,m,n}^{[j]}=(\delta_m)_*(e_{\nu,j,m})\in H^2_{et}(Y_n\times Z_m, D_{\nu,n}^\circ(1)\otimes\iota^{\nu-j}(\iota^c)^j)\]
Easy to check that  $\mom^r(z^{[j]}_{\nu,m,n})=\begin{cases} z_{m,n}^{[r-j,j]} & 0\leq j\leq r\\
0 & j>r\end{cases}$. 

We now can prove Theorem \ref{BigHeegner}. Note that we assume the pair $(\pi, \varphi_p)$ is {\bf{noble}} in the sense of \cite{Han17} which includes the case the slope $v_p(\alpha)$ is non-critical.
\begin{proof}[Proof of Theorem \ref{BigHeegner}]
  Under the noble assumption, the theory of eigencurve $\CC$ (see \cite{Han17}) guarantees that there exists
\begin{itemize} 
\item a small open disc $ D=\Spf(R)\subset \CC$ containing $(\pi,\varphi_p)$ and $\mathrm{wt}(D)\subset \CW_n$ where $\mathrm{wt}:\ \CC\to\CW$ is the weight map 
\item a rank $2$  free $R$-module  $\CV$ with continuous $G_{\BQ}$-action whose specialization at $(\pi,\varphi_p)$ is $V_\pi$ 
\item a Coleman family $\CF\in\Hom_{G_\BQ}(\CV^\vee(1), H^1_{et}(Y_{n,\bar{\BQ}}, A_{\boldsymbol{\kappa}_D,n})^{\leq v_p(\alpha)})${\footnote{We expect the slope $\leq v_p(\alpha)$-part $\Hom_{G_\BQ}(\CV^\vee(1), H^1_{et}(Y_{n,\bar{\BQ}}, A_{\boldsymbol{\kappa}_D,n})^{\leq v_p(\alpha)})$ can be described by the smooth admissible representation attached to $\CV$ via the local Langlands correspondence in families (see \cite{EH14}). When $B=M_2(\BQ)$, this kind of local-global compatibility is basically established in \cite{Pau11}. We shall report its generalization to general $B$  elsewhere.}} with $U_p$-eigenvalue $\alpha_\CF$ such that $\CF_r:=\mom^r(\CF)$ coincides with the image of $\varphi$
in $ \Hom_{G_{\BQ}}(V_\pi,H^1(Y_{n,\bar{\BQ}}, \Sym^r(V)))$. Here $\boldsymbol{\kappa}_D$ is the universal character of $\mathrm{wt}(D)$.
\end{itemize}
The Coleman family $\CF$ gives a map $H^1(Y_{n,\bar{\BQ}}, D_{\boldsymbol{\kappa}_D,n}(1))^{\leq v_p(\alpha)}\to\CV$. For each $j$, let $z^{[j]}_{\CF,m}$ be the image of $z_{\boldsymbol{\kappa}_D,m,n}^{[j]}$ along $\CF$. Then  by \cite[Lemma A.2.3]{Dis22},
\[\frac{\mathrm{Cores}_{F_{m+1}/F_m} z_{\CF,m+1}^{[j]}}{\sharp \Gal(F_{m+1}/F_m)}=\alpha_\CF\cdot z^{[j]}_{\CF,m}\]
Take $m_0$ sufficiently large such that $\tilde{\Gamma}_{m_0}^{-}:=\Gal(F_\infty/F_{m_0})\cong \BZ_p$. Let $\CW(\tilde{\Gamma}_{m_0}^-)$ be its weight space and $\boldsymbol{\kappa}_{m_0}^-$ be the universal character. Denote by $a\in  \CW(\tilde{\Gamma}_{m_0}^-)$ the character $\iota^a/(\iota^c)^a$. Then 
by the argument in \cite[Proposition 5.3.1]{JLZ}, there exists  a cohomology class $Z_{\CF,\infty}\in H^1(F_{m_0}, \CV\otimes\iota^{\boldsymbol{\kappa}_D}\hat{\otimes}\CO(\CW(\tilde{\Gamma}_{m_0}^-))(-\boldsymbol{\kappa}_{m_0}^-))$ whose image in $H^1(F_{m_0}, \CV\otimes\iota^{\boldsymbol{\kappa}_D}(\iota^c)^j)$ at $j\geq0$ and $m\geq m_0$ is $\frac{z_{\CF,m}^{[j]}}{\alpha_\CF^{m}}$. 
Note that 
\[H^1(\CK^S, \CV\otimes Ind_{F_{m_0}}^\CK\iota^{\boldsymbol{\kappa}_D}\hat{\otimes}\CO(\CW(\tilde{\Gamma}_{m_0}^-))(-\boldsymbol{\kappa}_m^{-}))=H^1(\CK^S, \CV\otimes \iota^{\boldsymbol{\kappa}_D}\hat{\otimes}\CO(\CW(\tilde{\Gamma}^-))(-\boldsymbol{\kappa}^-))\]
The  cohomology class with coefficients in $R\hat{\otimes}\CO(\CW(\tilde{\Gamma}^-))$: \[z_\CF:=\frac{\mathrm{Cores}_{F_{m_0}/\CK}Z_{\CF,\infty} }{\sharp\Gal(F_{m_0}/\CK)}\in H^1(\CK^S, \CV\otimes \iota^{\boldsymbol{\kappa}_D}\hat{\otimes}\CO(\CW(\tilde{\Gamma}^-))(-\boldsymbol{\kappa}^-))\]
Clearly $z_{\CF}$ is supported on $\tilde{D}$. Then as in \cite[Theorem 5.4.1]{JLZ} (combining with the local computation in \cite[Appendix A]{Dis22}), the specialization of $z_\CF$ at the point $(\pi,\varphi_p,\chi)\in \tilde{D}$ is $z^{V_\sigma,\varphi_p,\chi}:=\frac{z^{\varphi,\chi}}{\alpha^m{r\choose j}}${\footnote{ To see the interpolation property in the case $r=0$ and $B\neq M_2(\BQ)$, employ the fact that $\Hom^\circ(Y_U, A_\pi)\otimes_{M_\pi} L\cong \pi^{\infty}$ and choose a point $x_0\in X_U$ such that $\varphi(x_0)=O$. One makes all constructions with $Y_U$ replaced by the affine variety $Y_U-{x_0}$. Then one obtains the interpolation property as in  \cite[Theorem 5.4.1]{JLZ}.}} 
\end{proof}

\printindex
\end{document}